\documentclass[10pt, reqno]{amsart} 
\usepackage{graphicx,subfig}
\usepackage{enumerate} 
\usepackage{amsmath,amsfonts, amscd, bbm, eucal, latexsym,amssymb, extarrows}
\usepackage[utf8]{inputenc} 
\usepackage[T1]{fontenc}
\usepackage{ae,aecompl}
\usepackage{dsfont} 
\usepackage{pxfonts}
\usepackage{microtype}
\usepackage{braket}

\usepackage{yhmath}

\usepackage{rotating}

\usepackage[all]{xy}

\usepackage[colorlinks=true, pdfstartview=FitV,
linkcolor=cyan, citecolor=red, urlcolor=blue]{hyperref}
 
\newcommand{\marging}[1]
\makeatletter
\newcommand{\xRightarrow}[2][]{\ext@arrow 0359\Rightarrowfill@{#1}{#2}}
\makeatother

\theoremstyle{plain}
				
\newtheorem{theorem}{Theorem}[section]				
\newtheorem{proposition}[theorem]{Proposition}		
\newtheorem{corollary}[theorem]{Corollary}
\newtheorem{lemma}[theorem]{Lemma}

\theoremstyle{definition}

\newtheorem{definition}[theorem]{Definition}
\newtheorem{remark}[theorem]{Remark}
\newtheorem{example}[theorem]{Example}

\newcommand{\ABbb}{\mathbb A}

\newcommand{\CBbb}{\mathbb C}

\newcommand{\HBbb}{\mathbb H}

\newcommand{\LBbb}{\mathbb L}

\newcommand{\PBbb}{\mathbb P}

\newcommand{\RBbb}{\mathbb R}

\newcommand{\ZBbb}{\mathbb Z}

\newcommand{\Acal}{\mathcal A}

\newcommand{\Ccal}{\mathcal C}

\newcommand{\Fcal}{\mathcal F}

\newcommand{\Ical}{\mathcal I}

\newcommand{\Lcal}{\mathcal L}
\newcommand{\Mcal}{\mathcal M}
\newcommand{\Ncal}{\mathcal N}
\newcommand{\Ocal}{\mathcal O}

\newcommand{\Xcal}{\mathcal X}

\newcommand{\PSL}{\mathsf{PSL}}
\newcommand{\GL}{\mathsf{GL}}

\newcommand{\U}{\mathsf{U}}

\DeclareMathOperator{\End}{End}
\DeclareMathOperator{\Hom}{Hom}

\DeclareMathOperator{\id}{id}
\DeclareMathOperator{\imag}{Im}
\DeclareMathOperator{\real}{Re}

\DeclareMathOperator{\tr}{tr}

\DeclareMathOperator{\Div}{div}
\DeclareMathOperator{\ord}{ord}
\DeclareMathOperator{\Int}{int}
\DeclareMathOperator{\an}{an}
\DeclareMathOperator{\PICRIG}{PICRIG}
\DeclareMathOperator{\PIC}{PIC}
\DeclareMathOperator{\HOM}{HOM}

\newcommand{\dbar}{\bar\partial}

\newcommand{\lra}{\longrightarrow}

\newcommand{\Pic}{{\rm Pic}}

\newcommand{\GM}{{\rm GM}}
\newcommand{\KS}{{\rm KS}}

\newcommand{\Inv}{{\rm Inv}}

\newcommand{\isorightarrow}{\xrightarrow{
   \,\smash{\raisebox{-0.5ex}{\ensuremath{\sim}}}\,}}
   
   \newcommand{\isoleftarrow}{\xleftarrow{
   \,\smash{\raisebox{-0.5ex}{\ensuremath{\sim}}}\,}}

\begin{document}


\title[Deligne pairings and rank one local systems on curves]{Deligne pairings and families of \\ rank one local systems on algebraic curves}

\author[Freixas i Montplet]{Gerard Freixas i Montplet}

\address{CNRS -- Institut de Math\'ematiques de Jussieu,
   Universit\'e de Paris 6, 4 Place Jussieu,
   75005 Paris}
\email{gerard.freixas@imj-prg.fr}
\thanks{G. F. supported in part by ANR grant ANR-12-BS01-0002.}

\author[Wentworth]{Richard A. Wentworth}

\address{Department of Mathematics,
   University of Maryland,
   College Park, MD 20742, USA}
\email{raw@umd.edu}
\thanks{R.W. supported in part by NSF grant DMS-1406513.}

\date{\today}

\begin{abstract} 
For smooth families of projective algebraic curves,  
 we extend  the notion of intersection pairing of metrized line bundles to a pairing on line bundles with flat relative connections. 
In this setting, we prove the existence of  a  canonical and functorial ``intersection'' connection on the Deligne pairing. A relationship is found with the holomorphic extension of analytic torsion, and in the case of trivial fibrations we show that the Deligne isomorphism is flat with respect to the connections we construct.  Finally, we give an application to the construction of a meromorphic connection on the hyperholomorphic line bundle over the twistor space of rank one flat connections on a Riemann surface.
\end{abstract}



\maketitle

\setcounter{tocdepth}{2}
\tableofcontents

\thispagestyle{empty}





\section{Introduction}
\medskip
\medskip

Let $\pi:\Xcal\to S$ be a smooth proper morphism of smooth quasi-projective complex varieties with $1$-dimensional connected fibers. Let $\Lcal$ be a holomorphic line bundle on $\Xcal$, and denote by $\omega_{\Xcal/S}$ the relative dualizing sheaf of the family $\pi$. In his approach to understanding work of Quillen \cite{Quillen:85} on determinant bundles of families of $\overline{\partial}$-operators on a Riemann surface, Deligne \cite{Deligne:87} established a canonical (up to sign) functorial isomorphism of line bundles on $S$
\begin{equation}\label{eq:deligne-iso}
	\det R\pi_{\ast}(\Lcal)^{\otimes 12}
	\isorightarrow
	\langle\omega_{\Xcal/S},\omega_{\Xcal/S}\rangle\otimes\langle\Lcal,\Lcal\otimes\omega_{\Xcal/S}^{-1}\rangle^{\otimes 6}.
\end{equation}
The isomorphism refines to the level of sheaves the Grothendieck-Riemann-Roch theorem in relative dimension 1. It relates the determinant of the relative cohomology of $\Lcal$ (on the left hand side of \eqref{eq:deligne-iso}) to certain  ``intersection bundles'' $\langle\Lcal,\Mcal\rangle\to S$ (on the right hand side of \eqref{eq:deligne-iso}), known as \emph{Deligne pairings}, which associate line bundles on $S$ to  pairs of holomorphic bundles $\Lcal, \Mcal\to \Xcal$. The relationship with Quillen's construction finds some inspiration in Arakelov geometry, where metrized line bundles play a central role. Given smooth hermitian metrics on $\omega_{\Xcal/S}$ and $\Lcal$, then there is an associated \emph{Quillen metric} on $\det R\pi_{\ast}(\Lcal)$. The relevant input in the definition of this metric is the (holomorphic) analytic torsion of Ray-Singer: a spectral invariant obtained as a zeta regularized determinant of the positive \emph{self-adjoint} $\overline{\partial}$-laplacians for $\Lcal$ and the chosen metrics. Also, the Deligne pairings in \eqref{eq:deligne-iso} inherit hermitian metrics, defined in the style of the archimedean contribution to Arakelov's arithmetic intersection pairing. Using the Chern connections associated to these hermitian metrics,  the cohomological equality
\begin{equation} \label{eqn:c1}
c_1(\langle \Lcal, \Mcal\rangle)=\pi_\ast \left( c_1(\Lcal)\cup c_1(\Mcal)\right)
\end{equation}
becomes an equality of forms for the Chern-Weil expressions of $c_1$ in terms of curvature. 
Moreover,  for these choices of metrics, the Deligne isomorphism \eqref{eq:deligne-iso} becomes an isometry, up to an overall topological constant. 
This picture has been  vastly generalized in several contributions by Bismut-Freed, Bismut-Gillet-Soul\'e, Bismut-Lebeau, and others. They lead to the proof of the Grothendieck-Riemann-Roch theorem in Arakelov geometry, by Gillet-Soul\'e \cite{GilletSoule:92}.

In another direction, Fay \cite{Fay:81} studied the Ray-Singer torsion as a function on unitary characters of the fundamental group of a marked compact Riemann surface $X$. He showed that this function admits a unique holomorphic extension to the complex affine variety of complex characters of $\pi_{1}(X)$. He goes on to prove that the divisor of this function 
determines the marked Riemann surface structure. As for the Ray-Singer torsion, the holomorphic extension of the analytic torsion function to the complex character variety can be obtained by a zeta regularization procedure, this time for non-self-adjoint elliptic operators. Similar considerations 
appear in \cite{Kim:07, McIntyreTeo:08}, and  in more recent work \cite{Hitchin:13}, where Hitchin uses these zeta regularized determinants of non-self-adjoint operators in the construction of a hyperholomorphic line bundle on the moduli space of Higgs bundles.

From a modern perspective, it is reasonable to seek  a common conceptual framework for the results of Deligne, Fay and Hitchin, where the object of study is the determinant of cohomology of a line bundle endowed with a flat relative  connection. Hence, on the left hand side of \eqref{eq:deligne-iso}, one would like to define a connection on the determinant of the cohomology in terms of the spectrum of some natural non-self-adjoint elliptic operators, specializing to the Quillen connection in the unitary case. On the right hand side of \eqref{eq:deligne-iso}, one would like to define natural connections on the Deligne pairings, specializing to Chern connections in the metric case. The aim would then be to show that the Deligne isomorphism is flat for these connections. This is the motivation of the present article, where we achieve  the core of this program. Specifically,
\begin{itemize}
\item we define an \emph{intersection connection} on the Deligne pairing  of line bundles on $\Xcal$ with flat relative  connections; 
\item in the case of trivial families $\Xcal=X\times S$, 
we build a holomorphic connection on the determinant of the cohomology  by spectral methods and show that the Deligne isomorphism is flat with respect to this connection and intersection connections on Deligne pairings;
\item we recover some of the results of Fay and Hitchin as applications of our results.
\end{itemize}
In a separate paper, we will deduce the flatness of Deligne's isomorphism for general fibrations $\Xcal\to S$ from the case of the trivial families addressed in this article. In addition, we will derive an arithmetic Riemann-Roch theorem for flat line bundles on arithmetic surfaces. 

We now state the main results and outline of this paper more precisely. Given  a smooth connection on $\Lcal\to\Xcal$ which is compatible with the holomorphic structure and flat on the fibers of $\pi$, 
we wish to define an associated compatible connection on the Deligne pairing $\langle\Lcal, \Mcal\rangle$.  The existence of the Deligne pairing relies on the Weil reciprocity law of meromorphic functions on Riemann surfaces. Similarly, the construction of a connection on $\langle\Lcal, \Mcal\rangle$ requires a corresponding property which we will call \emph{Weil reciprocity for connections}, or (WR) for short.  It turns out that not every connection satisfies this condition, but  a given smooth family of flat relative  connections can always be extended to a connection that satisfies (WR) and which is functorial with respect to tensor products and base change (we shall simply say ``functorial''). 
This extension is unique once the bundle is rigidified;
in general  we characterize the space of all such extensions. It is important to stress that the extension is in general not holomorphic, even if the initial flat relative connection is.
The  result may then be formulated as follows.

\begin{theorem}[\sc Trace connection] \label{thm:main}
  Let  $\Lcal, \Mcal$ be holomorphic line bundles on $\pi:\Xcal\to S$.  Assume we are given:
  \begin{itemize}
  \item  a section $\sigma: S\to \Xcal$;
  \item a rigidification $\sigma^\ast\Lcal\simeq \Ocal_S$;
  \item a flat relative  connection  $\nabla_{\Xcal/S}$, compatible with $\Lcal$ (see Definition \ref{def:rel-hol}).
  \end{itemize}
  Then the following hold:
  \begin{enumerate}
  \item
  there is a unique extension of $\nabla_{\Xcal/S}$ to a smooth connection on $\Lcal$ that is compatible with the holomorphic structure, satisfies (WR) universally (i.e. after any base change $T\rightarrow S$) and induces the trivial connection on $\sigma^\ast\Lcal$;
  \item  consequently, $\nabla_{\Xcal/S}$ uniquely determines a functorial connection $\nabla^{tr}_{{\langle\Lcal, \Mcal\rangle}}$ on the Deligne pairing $\langle\Lcal, \Mcal\rangle$;
  \item in the case where $\nabla_{\Xcal/S}$ is the fiberwise restriction of the Chern connection for a hermitian structure on $\Lcal$, then
  $\nabla^{tr}_{{\langle\Lcal, \Mcal\rangle}}$ coincides with the connection for Deligne's metric on $\langle\Lcal, \Mcal\rangle$ (and any metric on $\Mcal$). 
\end{enumerate}
\end{theorem}

We shall 
use the term  \emph{trace connection} for the
connections $\nabla^{tr}_{{\langle\Lcal, \Mcal\rangle}}$ that arise from Theorem \ref{thm:main}  (see Definition \ref{def:trace-connection} for a precise definition). We will even see it exists in the absence of rigidificaiton. The extension result makes use of the moduli space of line bundles with flat relative  connections and the infinitesimal deformations of such, which we call \emph{Gauss-Manin invariants} $\nabla_{\GM}\nu$. These are $1$-forms on $S$ with values in the local system $H^1_{dR}(\Xcal/S)$ that are canonically associated to a flat relative  connection $\nabla_{\Xcal/S}$ (see Section \ref{sec:gm}).
The several types of connections and their moduli spaces are discussed in Section \ref{section:connections}. In Section \ref{sec:Deligne} we formalize the notion of Weil reciprocity and trace connection, and we formulate general existence and uniqueness theorems in terms of Poincar\'e bundles. In Section \ref{section:proofs} we attack the proof of Theorem \ref{thm:main}. The main result is Theorem \ref{thm:extension}, where we show that a certain \emph{canonical extension} of $\nabla_{\Xcal/S}$ satisfies all the necessary requirements.
The method is constructive and exploits several reciprocity laws for differential forms. Actually, in degenerate situations, one can see that the extension theorem encapsulates reciprocity laws for differentials of both the first and third kinds. The advantage of our procedure is that it admits a closed expression for the curvature of a trace connection on $\langle\Lcal,\Mcal\rangle$. This is the content of Proposition \ref{prop:curvature-trace}. The formula involves the Gauss-Manin invariant and the derivative of the period map for  the family of curves.

In the symmetric situation where both $\Lcal$ and $\Mcal$ are endowed with flat relative connections, this construction leads to an intersection pairing which generalizes that of Deligne (Definition \ref{def:intersection-connection}).  We have the following

\begin{theorem}[\sc Intersection connection] \label{thm:intersection}
Let  $\Lcal, \Mcal$, be holomorphic line bundles on $\pi:\Xcal\to S$ with flat relative  connections $\nabla^L_{\Xcal/S}$ and $\nabla^M_{\Xcal/S}$.  Then there is a uniquely determined connection $\nabla^{int}_{{\langle\Lcal, \Mcal\rangle}}$ on $\langle\Lcal, \Mcal\rangle$ satisfying:
\begin{enumerate}
\item  $\nabla^{int}_{{\langle\Lcal, \Mcal\rangle}}$  is functorial, and it is symmetric with respect to the isomorphism $$\langle\Lcal, \Mcal\rangle\simeq \langle\Mcal, \Lcal\rangle\ ;$$
\item  the curvature of $\nabla^{int}_{{\langle\Lcal, \Mcal\rangle}}$ is given by 
\begin{equation} \label{eqn:intersection-curvature}
 F_{\nabla^{int}_{\langle\Lcal, \Mcal\rangle}}=\frac{1}{2\pi i}\, \pi_\ast\left( \nabla_{\GM}\nu_L \cup \nabla_{\GM}\nu_M\right)
\end{equation}
where $\nabla_{\GM}\nu_L$ and $\nabla_{\GM}\nu_M$ are the Gauss-Manin invariants of $\Lcal$ and $\Mcal$, respectively, and the cup product is defined in \eqref{eqn:cup};
\item in the case where $\nabla^M_{\Xcal/S}$ is the fiberwise restriction of the Chern connection for a  hermitian structure on $\Mcal$, then $\nabla^{int}_{{\langle\Lcal, \Mcal\rangle}}=\nabla^{tr}_{{\langle\Lcal, \Mcal\rangle}}$;
\item in the case where both $\nabla^L_{\Xcal/S}$ and $\nabla^M_{\Xcal/S}$ are the restrictions of Chern connections for hermitian structures, then $\nabla^{int}_{{\langle\Lcal, \Mcal\rangle}}$ is the Chern connection for Deligne's metric on $\langle\Lcal, \Mcal\rangle$. 
\end{enumerate}
\end{theorem}
We will call $\nabla^{int}_{{\langle\Lcal, \Mcal\rangle}}$ the \emph{intersection connection} of $\nabla^L_{\Xcal/S}$ and $\nabla^M_{\Xcal/S}$. 
Using the Chern-Weil forms for the extensions of $\nabla^L_{\Xcal/S}$ and $\nabla^M_{\Xcal/S}$ obtained in Theorem \ref{thm:main}, the intersection connection realizes \eqref{eqn:c1}  as an equality of forms. Intersection connections are first introduced in Section \ref{sec:Deligne}, as an abutment of the theory of Weil reciprocity and trace connections.

The intersection connection on the Deligne pairing may therefore be regarded as an extension of Deligne's construction to encompass all flat relative connections and not just unitary ones.  This is a nontrivial enlargement of the theory, and in Section \ref{section:examples} we illustrate this point 
in the case of a trivial family $\Xcal=X\times S$, where  the definition of the connections on $\langle\Lcal, \Mcal\rangle$ described in Theorems \ref{thm:main} and \ref{thm:intersection} can be made very explicit. 
Given a holomorphic relative connection on $\Lcal\to X\times S$ (see Definition \ref{def:rel-hol}),  there is a classifying map $S\to \Pic^{0}(X)$, and $\det R\pi_\ast(\Lcal)$ is the pull-back to $S$ of the corresponding determinant of cohomology.  Viewing $\Pic^{0}(X)$ as the character variety of $\U(1)$-representations of $\pi_1(X)$, the determinant of cohomology carries a natural Quillen metric and associated Chern connection (in this case called the Quillen connection).  
If we choose a theta characteristic $\kappa$ on $X$,  $\kappa^{\otimes 2}=\omega_X$,  and consider instead the map $S\to \Pic^{g-1}(X)$ obtained from the family  $\Lcal\otimes \kappa\to \Xcal$, then $\det R\pi_\ast(\Lcal\otimes \kappa)$ is the pull back of $\Ocal(-\Theta)$.
 Using a complex valued holomorphic version of the  analytic torsion of Ray-Singer, $T(\chi\otimes\kappa)$, we show that the tensor product of the
determinants of cohomology for $X$ and $\overline X$  admits a canonical holomorphic connection.
On the other hand, in this situation the intersection connection on $\langle\Lcal, \Lcal\rangle$ is also holomorphic.
   In Theorem \ref{thm:flat} we show that the Deligne isomorphism, which relates these two bundles, is flat with respect to these connections.

Finally, again in the case of a trivial fibration, we point out a link with some of the ideas in the recent paper \cite{Hitchin:13}.
The space $M_{dR}(X)$ of flat rank 1 connections on $X$ has a hyperk\"ahler structure. Its twistor space $\lambda: Z\to\PBbb^1$ carries a holomorphic line bundle $\Lcal_Z$, which may be interpreted as a determinant of cohomology via Deligne's characterization of $Z$ as the space of $\lambda$-connections.  On $\Lcal_Z$ there is a meromorphic connection with simple poles along the divisor at $\lambda^{-1}(0)$ and $\lambda^{-1}(\infty)$, and whose curvature gives a holomorphic symplectic form on the other fibers. In Theorem \ref{thm:hitchin}, we show how this connection is obtained from the intersection connection on the Deligne pairing of the universal bundle on $M_{dR}(X)$.  Similar methods 
will potentially produce a higher rank version of this result; this will be the object of future research.

We end this introduction by noting that considerations similar to the central theme of this paper have been discussed previously by various authors. We mention here the work of Bloch-Esnault \cite{BlochEsnault:00} on the determinant of deRham cohomology and Gauss-Manin connections in the algebraic setting, and of Beilinson-Schechtman \cite{BeilinsonSchechtman:88}. Complex valued extensions of analytic torsion do not seem to play a role in these papers.  Gillet-Soul\'e \cite{GilletSoule:89} also initiated a study of Arakelov geometry for bundles with holomorphic connections, but left as an open question the possibility of a Riemann-Roch type theorem.

\medskip
\emph{Acknowledgments.}  The authors would like to thank Ignasi Mundet i Riera for an important comment concerning rigidification,  Dennis Eriksson for his valuable suggestions concerning connections on Deligne pairings, and
Scott Wolpert for pointing out reference \cite{Fay:81}.  R.W. is also grateful for the generous support of the Centre National de la Recherche Scientifique of France, and for the warm hospitality of the faculty and staff of the Institut de Math\'ematiques de Jussieu, where a portion of this work was completed.

\section{Relative Connections and Deligne Pairings}\label{section:connections}

\subsection{Preliminary definitions}  
Let $\pi:\Xcal\to S$ be a submersion of smooth manifolds, whose fibers are compact and two dimensional. We suppose that the relative complexified tangent bundle $T_{\pi,\mathbb{C}}$ comes equipped with a relative complex structure $J:T_{\pi,\mathbb{C}}\rightarrow T_{\pi,\mathbb{C}}$, so that the fibers of $\pi$ have a structure of compact Riemann surfaces. It  then makes sense to introduce sheaves of $(p,q)$ relative differential forms $\Acal^{p,q}_{\Xcal/S}$. There is a relative Dolbeault operator $\overline{\partial}:\Acal^{0}_{\Xcal}\rightarrow\Acal^{0,1}_{\Xcal/S}$, which is just the projection of the exterior differential to $\Acal^{0,1}_{\Xcal/S}$.  The case most relevant in this paper and to which we shall soon restrict ourselves is, of course,  when $\pi$ is a holomorphic map of complex manifolds. Then, the relative Dolbeault operator is the projection to relative forms of the Dolbeault operator on $\Xcal$. 

Let $L\to \Xcal$ be a $\mathcal{C}^\infty$ line bundle. We may consider several additional structures on $L$. The first one is relative holomorphicity.\begin{definition}
A \emph{relative holomorphic structure} on $L$ is the choice of a \emph{relative Dolbeault operator on $L$}: a $\mathbb{C}$-linear map $\overline{\partial}_{L}:\Acal^0_\Xcal(L)\rightarrow\Acal^{0,1}_{\Xcal/S}(L)$ that satisfies the Leibniz rule with respect to the relative $\overline{\partial}$-operator.
 We will write $\Lcal$ for a pair $(L,\overline{\partial}_{L})$, and call it a \emph{relative holomorphic line bundle}. 
\end{definition}
\begin{remark}
In the holomorphic (or algebraic) category we shall always assume the relative Dolbeault operator is the fiberwise restriction of a global integrable operator $\overline{\partial}_{L}:\Acal^0_\Xcal(L)\rightarrow\Acal^{0,1}_{\Xcal}(L)$, so that $\Lcal\to\Xcal$ is a holomorphic bundle. In order to stress the distinction, we will sometimes refer to a \emph{global} holomorphic line bundle on $\Xcal$.
\end{remark}
The second kinds of structure to be considered are various   notions of connections.
\begin{definition} \label{def:rel-hol}
\hspace{2em}
\begin{enumerate}
\item
A \emph{relative connection on $L\to\Xcal$} is a 
$\CBbb$-linear map 
$\nabla_{\Xcal/S} : \Acal^0_\Xcal(L) \to \Acal^1_{\Xcal/S}(L)$
 satisfying the Leibniz rule with respect to the relative exterior differential $d:\Acal^{0}_{\Xcal}\rightarrow\Acal^{1}_{\Xcal/S}$.
 \item Given the structure of a relative holomorphic line bundle $\Lcal=(L,\overline{\partial}_{L})$, a \emph{relative connection on} $\Lcal$ is a relative connection on the underlying $\mathcal{C}^\infty$ bundle $L$ that is compatible with $\Lcal$, in the sense that the vertical $(0,1)$ part satisfies: $(\nabla_{\Xcal/S})''=\dbar_L$ (relative operator). 
 \item A relative connection on $L\to\Xcal$ is called \emph{flat} if the induced connection on $L\bigr|_{\Xcal_s}$ is flat for each $s\in S$. 
 \item If $\pi$ is a holomorphic map of complex manifolds and $\Lcal\to\Xcal$ is a global holomorphic line bundle on $\Xcal$, a relative connection on $\Lcal$ is called \emph{holomorphic} if it induces a map $\nabla_{\Xcal/S}: \Lcal\to \Lcal\otimes \Omega_{\Xcal/S}^{1}$. Such connections are automatically flat.
 \item Finally, if $\nabla_{\Xcal/S}$ is a relative connection on a global holomorphic line bundle $\Lcal\to \Xcal$, then a smooth connection $\nabla$ on $L$ is called \emph{compatible} with  $\nabla_{\Xcal/S}$ if $\nabla^{0,1}=\dbar_L$ (global operator) and the projection to relative forms makes the following diagram commute:
$$
\xymatrix{
L \ar[r]^{\nabla\quad}   \ar@/_1pc/@{>}[rr]_{\nabla_{\Xcal/S}} &\Acal^{1}_\Xcal(L) \ar[r] &  \Acal^{1}_{\Xcal/S}(L) 
}
$$
\end{enumerate}
\end{definition}

\begin{remark}  \label{rem:holomorphic}
\hspace{2em}
\begin{enumerate}
\item If $\Lcal$ is a relative holomorphic line bundle and $\nabla_{\Xcal/S}$ is a flat relative connection, then its restrictions to fibers are holomorphic connections. This is important to keep in mind.
\item The  important special case (iv)  above  occurs, for example, when $\nabla_{\Xcal/S}$ is  the fiberwise restriction of a holomorphic connection 
on $\Lcal$. This is perhaps the most natural situation from the algebraic point of view.  
However, the more general case of flat relative connections  considered in this paper is far more flexible and is necessary for applications, as the next example illustrates (see also Remark \ref{rem:betti} below).
\end{enumerate}
\end{remark}

\begin{example} \label{ex:chern}
Suppose $\pi$, $\Xcal$, $S$ holomorphic and $\Lcal\to \Xcal$ is a global holomorphic line bundle with relative degree zero.  Then there is  a smooth hermitian metric on $\Lcal$ such that the restriction of the Chern connection $\nabla_{ch}$ to each fiber is flat, and for a rigidified bundle (i.e. the choice of a trivialization along a given section) this metric and connection can be uniquely normalized (by imposing triviality along the section). Abusing terminology slightly, we shall refer to the connection $\nabla_{ch}$ as  \emph{the Chern connection of} $\Lcal\to \Xcal$.
The fiberwise restriction of $\nabla_{ch}$ then gives a flat relative connection $\nabla_{\Xcal/S}$.  Note that outside of some trivial situations  it is essentially never the case that $\nabla_{\Xcal/S}$ is holomorphic in the sense of Definition \ref{def:rel-hol} (iv).
 \end{example}
\subsection{Gauss-Manin invariant} \label{sec:gm}
Let $\pi:\Xcal\rightarrow S$ be as in the preceding discussion, and suppose it comes equipped with a fixed section $\sigma:S\rightarrow\Xcal$. The problem of extending relative connections to global connections requires infinitesimal deformations of line bundles with relative connections. In our approach, it is convenient to introduce a moduli point of view. We set $M_{dR}(\Xcal/S)=\{\text{moduli of flat relative connections on $\Xcal$}\}$ on a fixed $\mathcal{C}^{\infty}$ bundle $L\to \Xcal$ that is topologically trivial on the fibers.  
Consider the functor of points:
$
\{T\to S\}\mapsto M_{dR}(\Xcal_T/T)$,
where $T\rightarrow S$ is a morphism of smooth manifolds and $\Xcal_T$ is the base change of $\Xcal$ to $T$. This functor can be represented by a smooth fibration in Lie groups over $S$. To describe it, let us consider the relative deRham cohomology $H^{1}_{dR}(\Xcal/S)$. This is a complex local system on $S$, whose total space may be regarded as a $\mathcal{C}^{\infty}$ complex vector bundle. The local system $R^{1}\pi_{\ast}(2\pi i\underline{\mathbb{Z}})\to S$ is contained and is discrete in $H^{1}_{dR}(\Xcal/S)$. We can thus form the quotient:
\begin{displaymath}
	H^{1}_{dR}(\Xcal/S)/R^{1}\pi_{\ast}(2\pi i\underline{\mathbb{Z}})\rightarrow S.
\end{displaymath}
This space represents $T\mapsto M_{dR}(\Xcal_{T}/T)$. Therefore, given a pair $(\Lcal,\nabla_{\Xcal/S})$ (or more generally $(\Lcal,\nabla_{\Xcal_{T}/T})$) formed by a relative holomorphic line bundle together with a flat relative  connection, there is a classifying $\mathcal{C}^{\infty}$ morphism
\begin{displaymath}
	\nu:S\rightarrow H^{1}_{dR}(\Xcal/S)/R^{1}\pi_{\ast}(2\pi i\underline{\mathbb{Z}}).
\end{displaymath}
Locally on $S$, this map lifts to $\widetilde{\nu}:U\rightarrow H^{1}_{dR}(\Xcal/S)$. For future reference (e.g.\ Proposition \ref{prop:curvature-trace}), we note that $\real \tilde\nu$ is well-defined independent of the lift.
Applying the Gauss-Manin connection gives an element 
$$
\nabla_{\GM}\tilde\nu\in H^1_{dR}(\Xcal/S)\otimes \Acal^1_U.
$$
Now since $\nabla_{\GM}R^1\pi_\ast(2\pi i\underline{\ZBbb})=0$, it follows that the above expression is actually well-defined globally,  independent of the choice of  lift (and we therefore henceforth omit the tilde from the notation).
We define the \emph{Gauss-Manin invariant} of $(\Lcal,\nabla_{\Xcal/S})$ by
\begin{equation} \label{eqn:GM}
	\nabla_{\GM}\nu\in H^1_{dR}(\Xcal/S)\otimes \Acal^1_S.
\end{equation}
In case $\pi:\Xcal\rightarrow S$ is a proper morphism of smooth complex algebraic varieties, then the moduli space of flat connections on $L$ can be related to the so called universal vectorial extension \cite{MazurMessing:74}. Set
$$
J := J(\Xcal/S)=\text{Pic}^0(\Xcal/S).
$$
The universal vectorial extension $E(\Xcal/S)$ is a smooth algebraic group variety over $S$, sitting in an exact sequence of algebraic group varieties
\begin{equation} \label{eqn:universal-extension}
0\lra V(\pi_\ast\Omega^1_{\Xcal/S})\lra E(\Xcal/S)\lra J(\Xcal/S)\lra 0.
\end{equation}
Here $V(\pi_\ast\Omega^1_{\Xcal/S})$ denotes the vector bundle associated to $\pi_\ast\Omega^1_{\Xcal/S}$. In the algebraic (or holomorphic) category, $E(\Xcal/S)$ is universal for the property of being an extension of the relative Jacobian by a vector bundle, and is a fine moduli space for line bundles with relative holomorphic connections (up to isomorphism). Restricted to the $\mathcal{C}^{\infty}$ category, $E(\Xcal/S)$ also represents $T\mapsto M_{dR}(\Xcal_{T}/T)$. Therefore, even if $E(\Xcal/S)$ is actually a smooth complex algebraic variety in this case, a relative holomorphic line bundle with flat connection $(\Lcal,\nabla_{\Xcal/S})$ corresponds to a $\mathcal{C}^{\infty}$ map
\begin{displaymath}
	\nu:S\rightarrow E(\Xcal/S).
\end{displaymath}
The connection $\nabla_{\Xcal/S}$ is holomorphic exactly when this classifying map is holomorphic. In this case, the Gauss-Manin invariant as defined above is an element of:
\begin{displaymath}
	\nabla_{\GM}\nu\in H^{1}_{dR}(\Xcal/S)\otimes\Omega^{1}_{S}.
\end{displaymath}
We mention an intermediate condition that is also natural:
\begin{definition} \label{def:type}
A flat relative connection will be called  \emph{of type $(1,0)$} if $\nabla_{\GM}\nu\in H^1_{dR}(\Xcal/S)\otimes \Acal^{1,0}_S$.
\end{definition}

It will be useful to recall the following (cf.\ \cite{Voisin:07}). A local expression for $\nabla_{\GM}\nu$ is computed as follows:  let $s_i$ be local coordinates on $U\subset S$ and $\widetilde{\partial_{s_i}}$ a lifting to $\Xcal_{U}$ of the vector field $\partial/\partial s_i$.  Suppose $\nabla$ is a connection with curvature $F_\nabla$ such that the restriction of $\nabla$ to the fibers in $U$ coincides with the relative connection $\nabla_{\Xcal/S}$.  Then 
\begin{equation} \label{eqn:GM-expression}
\nabla_{\GM}\nu=\sum_i\left[ \Int_{\widetilde{\partial_{s_i}}}(F_\nabla)\bigr|_{\rm fiber}\right]\otimes ds_i\in H^1_{dR}(\Xcal/S)\otimes \Acal^1_U.
\end{equation}
This formula is an infinitesimal version of Stokes theorem. With this formula in hand, one easily checks  that the Gauss-Manin invariant is compatible with base change. Let $\varphi:T\rightarrow S$ be a morphism of manifolds. Then there is a natural pull-back map
\begin{displaymath}
	\varphi^{\ast}:H^{1}_{dR}(\Xcal/S)\otimes\Acal^{1}_{S}\rightarrow H^{1}_{dR}(\Xcal_{T}/T)\otimes\Acal^{1}_{T}.
\end{displaymath}
Under this map, we have
\begin{equation}\label{eqn:GM-pullback}
	\varphi^{\ast}(\nabla_{\GM}\nu)=\nabla_{\GM}(\varphi^{\ast}\nu),
\end{equation}
where $\varphi^{\ast}\nu$ corresponds to the pull-back of $(\Lcal,\nabla_{\Xcal/S})$ to $\Xcal_{T}$.

Finally, we introduce the following notation.  Let 
\begin{align}
(\nabla_{GM}\nu)'&=\Pi' \nabla_{GM}\nu\in H_{dR}^{1,0}(\Xcal/S)\otimes \Acal_S^1 \label{eqn:nu'}\\
(\nabla_{GM}\nu)''&=\Pi'' \nabla_{GM}\nu\in H_{dR}^{0,1}(\Xcal/S)\otimes \Acal_S^1 \label{eqn:nu''}
\end{align}
where $\Pi'$, $\Pi''$ are the projections onto  the $(1,0)$ and $(0,1)$ parts  of $\nabla_{GM}\nu$ under the relative Hodge decomposition of $\mathcal{C}^{\infty}$ vector bundles
\begin{displaymath}
	H^{1}_{dR}(\Xcal/S)=H^{1,0}(\Xcal/S)\oplus H^{0,1}(\Xcal/S)\ .
\end{displaymath}

\subsection{Deligne pairings, norm and trace}
Henceforth, we suppose that $\pi:\Xcal\rightarrow S$ is a smooth proper morphism of smooth quasi-projective complex varieties, with connected fibers of relative dimension 1. Let $\Lcal,\Mcal\to \Xcal$ be algebraic line bundles. The \emph{Deligne pairing} $\langle\Lcal,\Mcal\rangle\to S$ is a line bundle defined as follows. As an $\Ocal_S$-module, it can be described locally for the Zariski or \'etale topologies on $S$ (at our convenience), in terms of generators and relations. In this description, we may thus localize $S$ for any of these topologies, without any further comment:
  \begin{itemize}
  \item Generators:  local generators  of $\langle\Lcal,\Mcal\rangle\to S$ are given by symbols $\langle\ell, m\rangle$ where $\ell$, $m$ are rational sections of $\Lcal$, $\Mcal$, respectively,  with disjoint divisors that are finite and flat over $S$. We say that $\ell$ and $m$ are in general position.
  \item Relations: for $f\in \CBbb(\Xcal)^\times$ and rational sections $\ell$, $m$, such that $f\ell$, $m$ and $\ell$, $m$ are in general position, 
  \begin{equation} \label{eqn:relations}
  \langle f\ell, m\rangle=N_{\Div m/S}(f)\langle \ell, m\rangle
  \end{equation}
   and similarly for $\langle \ell, fm\rangle$. Here $N_{\Div m/S}: \Ocal_{\Div m}\to \Ocal_S$ is the norm morphism.
  \end{itemize}
The Deligne pairing has a series of properties (bi-additivity, compatibility with base change, cohomological construction \`a la Koszul, etc.) that we will not recall here; instead, we refer to  \cite{Elkik:89}  for a careful and general discussion. 

\begin{remark}
There is a holomorphic variant of Deligne's pairing in the analytic category, defined analogously, which we denote temporarily by $\langle\cdot,\cdot\rangle^{\an}$. If ``$\an$'' denotes as well the analytification functor from algebraic coherent sheaves to analytic coherent sheaves, there is a canonical isomorphism, compatible with base change,
\begin{displaymath}
	\langle\Lcal,\Mcal\rangle^{\an}\isorightarrow\langle\Lcal^{\an},\Mcal^{\an}\rangle^{\an}.
\end{displaymath}
Actually, there is no real gain to working in the analytic as opposed to the algebraic category, since we assume our varieties to be quasi-projective. Indeed, the relative Picard scheme of degree $d$ line bundles $\Pic^{d}(X/S)$ is quasi-projective as well. By use of a projective compactification $\overline{S}$ of $S$ and $\overline{P}$ of $\Pic^{d}(X/S)$ and Chow's lemma, we see that holomorphic line bundles on $\Xcal$ of relative degree $d$ are algebraizable. For instance, if $\Lcal$ is holomorphic on $\Xcal$, after possibly replacing $S$ by a connected component, it corresponds to a graph $\Gamma$ in $S^{\an}\times\Pic^{d}(\Xcal/S)^{\an}$. By taking the Zariski closure in $\overline{S}^{\an}\times \overline{P}^{\an}$, we see that $\Gamma$ is an algebraic subvariety of $S^{\an}\times\Pic^{d}(X/S)^{\an}$, and then the projection isomorphism $\Gamma\to S^{\an}$ is necessarily algebraizable. Therefore, the classifying morphism of $\Lcal$, $S^{\an}\to\Pic^{d}(X/S)^{\an}$, is algebraizable.
For the rest of the paper we shall  interchangeably  speak of algebraic or holomorphic line bundles on $\Xcal$ (or simply line bundles). Similarly,  we suppress the index ``$\an$" from the notation.
\end{remark}

The following  is undoubtedly well-known, but for 
 lack of a precise reference we provide the statement and proof.
\begin{lemma}
Let $\pi:\Xcal\rightarrow S$, $\Lcal$, $\Mcal$ be as above. Locally Zariski over $S$, the Deligne pairing $\langle\Lcal,\Mcal\rangle$ is generated by symbols $\langle\ell,m\rangle$, with rational sections $\ell,m$ whose divisors are disjoint, finite and \'etale over $S$. In addition, if $\sigma:S\rightarrow\Xcal$ is a given section, one can suppose that $\Div\ell$ and $\Div m$ avoid $\sigma$. 
\end{lemma}
\begin{proof}
It is enough to prove the first statement in the presence of a section $\sigma$ (base change of rational sections defined over $S$). The lemma is an elaboration of Bertini's theorem. After taking the closure of the graph of $\pi$ in a suitable projective space and desingularizing by Hironaka's theorem, we can assume that $\pi$ extends to a morphism of smooth and projective algebraic varieties $\tilde\pi:\widetilde{\Xcal}\rightarrow\widetilde{S}$. By this we mean that $S\subset\widetilde{S}$ is a dense open Zariski subset, and $\pi$ is the restriction (on the base) of $\tilde\pi$ to $S$. We can also assume that $\Lcal$ and $\Mcal$ extend to line bundles on $\widetilde{\Xcal}$, denoted $\widetilde{\Lcal}$ and $\widetilde{\Mcal}$. By the projectivity of $\widetilde{\Xcal}$ and the bi-additivity of the Deligne pairing, we can thus take $\widetilde{\Lcal}$ and $\widetilde{\Mcal}$ to be very ample line bundles. Fix a fiber $\Xcal_{s}$ of $\pi$, for $s\in S$. Hence $\Xcal_{s}\hookrightarrow\widetilde{\Xcal}$ is a closed immersion of smooth varieties. By Bertini's theorem, we can find a global section $\ell$ of $\widetilde{\Lcal}$ whose divisor in $\widetilde{\Xcal}$ is smooth, irreducible and intersects $\Xcal_{s}$ transversally, in a finite number of points $F$. We can also assume it avoids $\sigma(s)$. Notice that by the choice of the section $\ell$, the map $\Div\ell\rightarrow\widetilde{S}$ is smooth and finite at $s$. Consequently, there is an open Zariski neighborhood of $s$, say $V\subset S$, such that $(\Div\ell){\mid_V}$ is finite \'etale over $V$ and disjoint with $\sigma(V)$. Because $F$ is a finite set of points, we can also find a global section $m$ of $\widetilde{\Mcal}$, with smooth divisor $\Div m$, intersecting $\Xcal_{s}$ transversally and avoiding $F\cup \lbrace\sigma(s)\rbrace$. After possibly restricting $V$, we can assume that $(\Div m){\mid_V}$ is also finite \'etale over $V$ and is disjoint with $(\Div\ell){\mid_V}\cup\sigma(V)$. 
\end{proof}

The relevance of the lemma will be apparent later when we discuss connections on Deligne pairings. While the defining relations in the Deligne pairing make use of the norm morphism of rational functions, the construction of connections will require traces of differential forms. This is possible when our divisors are finite \'etale over the base: for a differential form $\omega$ defined on an open neighborhood of an irreducible divisor $D\hookrightarrow \Xcal$ that is finite \'etale on $S$, the trace $\tr_{D/S}(\omega)$ is the map induced by inverting the map $\pi^\ast:\Acal_S^i\to \Acal_D^i$ (which is possible because $D\to S$ is finite \'etale). The trace is extended by linearity to Weil divisors whose irreducible components are finite \'etale over the base. The following is then clear:
\begin{lemma} \label{lem:trace-norm}
If $D$ is a Weil divisor in $\Xcal$ whose irreducible components are finite \'etale over $S$, then
$
d\log N_{D/S}(f)=\tr_{D/S}(d\log f)
$. 
\end{lemma}

\subsection{Metrics and connections}\label{subsection:metrics-connections}
We continue with the previous notation.  Suppose now that $\Lcal$, $\Mcal$ are endowed with smooth hermitian metrics $h,k$, respectively. For both we shall denote the associated norms $\Vert\cdot\Vert$. Then Deligne \cite{Deligne:87} defines a metric on $\langle\Lcal, \Mcal\rangle$ via the following formula: 
\begin{equation} \label{eqn:deligne-metric}
\log\Vert\langle\ell, m\rangle\Vert=\pi_\ast\left(\log\Vert m\Vert \, c_1(\Lcal, h)+ \log\Vert \ell\Vert\, \delta_{\Div m}\right)
\end{equation}
where $c_1(\Lcal, h)=(i/2\pi) F_\nabla$ is the Chern-Weil form of the Chern connection $\nabla$ of $(\Lcal, h)$.  Note that the expression in parentheses above is $\log\Vert\ell\Vert\ast\log\Vert m\Vert$ as defined in \cite{GilletSoule:92}.  If  $\nabla$  is flat on the fibers of $\Xcal$, then 
$$
\log\Vert\langle\ell, m\rangle\Vert^2=\pi_\ast\left(\log\Vert \ell\Vert^{2} \delta_{\Div m}\right)
=\tr_{\Div m/S}\left(\log\Vert\ell\Vert^2\right)
$$
and 
\begin{equation} \label{eqn:deligne-connection}
\partial\log\Vert\langle\ell, m\rangle\Vert^2=\tr_{\Div m/S}\left(\partial\log\Vert\ell\Vert^2\right)
=\tr_{\Div m/S}\left(\frac{\nabla\ell}{\ell}\right).
\end{equation}
Given a flat relative  connection on $\Lcal$, not necessarily unitary, we wish to take the right hand side of \eqref{eqn:deligne-connection} as the definition of a \emph{trace connection} on the pairing $\langle \Lcal, \Mcal\rangle$.  In this case, we  \emph{define}
\begin{equation} \label{eqn:derivative}
\nabla\langle \ell, m\rangle := \langle \ell, m\rangle\otimes \tr_{\Div m/S}\left(\frac{\nabla\ell}{\ell}\right).
\end{equation}
We extend this definition to the free $\mathcal{C}^\infty(S)$-module generated by the symbols, by enforcing the Leibniz rule:
\begin{equation} \label{eqn:leibniz}
\nabla (\varphi\langle \ell, m\rangle) := d\varphi\otimes \langle \ell, m\rangle+ \varphi\nabla\langle \ell, m\rangle
\end{equation}
for all $\varphi\in C^\infty(S)$.  Later, in Section \ref{sec:Deligne}, we will see that this is the only sensible definition whenever we neglect the connection on $\Mcal$. To show that \eqref{eqn:derivative} gives a well-defined connection on $\langle\Lcal, \Mcal\rangle$, we must verify compatibility with the definition of the Deligne pairing.  Because of the asymmetry of the pair, this amounts to two conditions: compatibility with the change of frame $\ell$, which is always satisfied, and compatibility with the choice of section $m$,  which is not. 

Let us address the first issue. Consistency between  \eqref{eqn:relations} and \eqref{eqn:leibniz} requires the following statement:
 
\begin{lemma} \label{lem:leibniz}
With $\nabla$ defined as in \eqref{eqn:derivative} and \eqref{eqn:leibniz}, then
$$
\nabla\langle f\ell, m\rangle = dN_{\Div m/S}(f)\otimes\langle \ell, m\rangle 
+N_{\Div m/S}(f)\nabla\langle \ell, m\rangle 
$$
for all meromorphic $f$ on $\Xcal$ for which  the Deligne symbols are defined.
\end{lemma}

\begin{proof}
By Lemma \ref{lem:trace-norm}, the  right hand side above is
\begin{align*}
\tr_{\Div m/S}(df)\otimes\langle \ell, m\rangle +N_{\Div m/S}(f) \tr_{\Div m/S}&\left(\frac{\nabla\ell}{\ell}\right)\otimes
\langle \ell, m\rangle  \\
&= \tr_{\Div m/S}\left(\frac{df}{f}  + \frac{\nabla\ell}{\ell}\right)\langle f\ell, m\rangle\ .
\end{align*}
By \eqref{eqn:deligne-connection}, the left hand side is
$$
\tr_{\Div m/S}\left(\frac{\nabla (f\ell)}{f\ell}\right)\otimes
\langle f\ell, m\rangle 
=\tr_{\Div m/S}\left(\frac{df}{f}  + \frac{\nabla\ell}{\ell}\right)\langle f\ell, m\rangle
$$
by the Leibniz rule for the connection on $\Lcal$.
\end{proof}

The second relation is consistency with the change of frame $m\mapsto fm$, $f\in \CBbb(\Xcal)^\times$ (whenever all the symbols are defined).  By Lemma \ref{lem:trace-norm}, we require
\begin{equation} \label{eqn:second-relation}
\frac{\nabla\langle \ell, fm\rangle}{\langle \ell, fm\rangle}=
\frac{\nabla\langle \ell, m\rangle}{\langle \ell, m\rangle}
+
\tr_{\Div \ell/S}\left(\frac{df}{f} \right).
\end{equation}
By \eqref{eqn:deligne-connection}, 
$$
\frac{\nabla\langle \ell, fm\rangle}{\langle \ell, fm\rangle}=
 \tr_{\Div(fm)/S}\left(\frac{\nabla\ell}{\ell}\right)
 =\frac{\nabla\langle \ell, m\rangle}{\langle \ell, m\rangle}
+
 \tr_{\Div f/S}\left(\frac{\nabla\ell}{\ell}\right).
$$
So \eqref{eqn:second-relation} is satisfied if and only if
$$
I(f,\ell, \nabla) :=\tr_{\Div f/S}\left(\frac{\nabla\ell}{\ell}\right)-\tr_{\Div \ell/S}\left(\frac{df}{f} \right)=0.
$$
Note that if $\ell\mapsto g\ell$, with $\Div g$ in general position, then
$$
I(f,g\ell, \nabla)=I(f,\ell,\nabla)+\tr_{\Div f/S}\left(\frac{dg}{g} \right)
-\tr_{\Div g/S}\left(\frac{df}{f} \right)
$$
But Weil reciprocity implies
$$
N_{\Div f/S}(g)=N_{\Div g/S}(f) 
\stackrel{d\log}{\xRightarrow{\hspace*{1cm}}}
\tr_{\Div f/S}\left(\frac{dg}{g} \right)
=\tr_{\Div g/S}\left(\frac{df}{f} \right)
$$
so $I=I(f,\nabla)$ is actually independent of $\ell$ and defined for all $f$. In particular, $I(f,\nabla)$ depends only on the isomorphism class of $\nabla$.
 Moreover,
$$
I(fg,\nabla)=I(f,\nabla)+ I(g,\nabla)
$$
Thus, extending the trace trivially on vertical divisors, we have the following 
\begin{align*}
I(\nabla): \CBbb(\Xcal)^\times \longrightarrow \Acal^1_{\CBbb(S)} &:=\lim_{\longrightarrow}\{\Gamma(U,\Acal^1_{S}) : \text{ Zariski open } U\subset S \}\\
f&\longmapsto I(f,\nabla)
\end{align*}

We will say that  a connection  $\nabla$ on $\Lcal\to\Xcal$ satisfies \emph{Weil reciprocity} (WR) if $I(\nabla)=0$. In the next section we elaborate on this notion as well as a functorial version, whose importance will be seen in the uniqueness issue. So far, we have shown the following
\begin{proposition}\label{proposition:Leibniz-WR}
If $\nabla$ as above satisfies (WR), then for any line bundle $\Mcal\to \Xcal$, $\nabla$  induces a connection on $\langle\Lcal, \Mcal\rangle$. 
\end{proposition}

\begin{example} \label{ex:chern-reciprocity}
The Chern connection $\nabla_{ch}$ on $\Lcal$ induces a well-defined (Chern) connection $\langle\Lcal,\Mcal\rangle$ for all $\Mcal$, by using Deligne's metric. Notice from \eqref{eqn:deligne-metric} that this is independent of a choice of metric on $\Mcal$.
We  then clearly have $I(\nabla_{ch})=0$.
To see this explicitly, note that if $h$ is the metric on $\Lcal$ in the frame $\ell$, then 
\begin{align*}
\tr_{\Div f/S}\left(\frac{\nabla\ell}{\ell}\right)&= \tr_{\Div f/S}(\partial h) =\partial (N_{\Div f/S} \log h )
=\partial\, \pi_\ast\left( \log\Vert\ell\Vert^2\delta_{\Div f}\right) \\
&=\partial\, \pi_\ast\left( \log\Vert\ell\Vert^2\left(\frac{1}{2\pi i}\dbar\partial \log|f|^2\right)\right)\qquad\qquad\text{(by Poincar\'e-Lelong)} \\
&=\partial\, \pi_\ast\left(\left(\frac{1}{2\pi i}\dbar\partial  \log\Vert\ell\Vert^2\right)\log|f|^2\right) \\
&=\partial\, \pi_\ast\left( \log|f|^2\delta_{\Div \ell}\right) \qquad\qquad\qquad\text{(since $\nabla_{\Xcal/S}$ is relatively flat)}\\
&=\tr_{\Div \ell/S}\left(\frac{df}{f}\right)
\end{align*}
\end{example}

\section{Trace Connections and Intersection Connections} \label{sec:Deligne}
In this section we formalize the notion of connections satisfying (WR). We characterize them in terms of \emph{trace connections} (see Section \ref{subsection:weil-trace}). These are connections on Deligne pairings $\langle\Lcal,\Mcal\rangle$, arising from  a connection on $\Lcal$ satisfying (WR). This interpretation leads to a structure theorem for the space of all connections satisfying (WR) on a given line bundle (Theorem \ref{theorem:existence-unicity}). Trace connections should be viewed as an intermediary construction, leading ultimately to the \emph{intersection connections} of Section \ref{subsection:intersection}. In this case, both $\Lcal$ and $\Mcal$ are endowed with a connection satisfying (WR). This is the core of this section, and it will play a prominent role in Section \ref{section:examples}, for instance, when we establish the flatness of Deligne type isomorphisms.

\subsection{Weil reciprocity and trace connections}\label{subsection:weil-trace}
Consider a smooth and proper morphism of smooth quasi-projective complex varieties $\pi:\Xcal\rightarrow S$, with connected fibers of relative dimension 1. We suppose we are given a section $\sigma:S\rightarrow \Xcal$. Let $\Lcal\to \Xcal$ be a holomorphic line bundle. Assume also that $\Lcal$ is rigidified along $\sigma$; that is, there is  an isomorphism $\sigma^{\ast}(\Lcal)
\isorightarrow\Ocal_S$, fixed once for all. 

Let $\nabla:\Acal^0_\Xcal(L)\rightarrow \Acal^{1}_\Xcal(L)$ be a connection on $\Lcal$ (recall for holomorphic line bundles we usually assume compatibility with holomorphic structures). We say that $\nabla$ is rigidified along $\sigma$ if it corresponds to the trivial derivation through the fixed isomorphism $\sigma^{\ast}(\Lcal)\isorightarrow\Ocal_S$. 

\begin{definition} \label{def:WR}
We say that the rigidified connection $\nabla$ satisfies  \emph{Weil reciprocity} (WR) if for every meromorphic section $\ell$ of $\Lcal$ and meromorphic function $f\in\mathbb{C}(\Xcal)^{\times}$, whose divisors $\Div\ell$ and $\Div f$ are \'etale and disjoint over a Zariski open subset $U\subset S$ with $\Div\ell$  disjoint from $\sigma$, the following identity of smooth differential forms on $U$ holds:
\begin{equation}\label{eq:2}
	\tr_{\Div f/U}\left(\frac{\nabla\ell}{\ell}\right)=\tr_{\Div\ell/U}\left(\frac{df}{f}\right).
\end{equation}
We say that $\nabla$ satisfies the (WR) \emph{universally} if for every morphism of smooth quasi-projective complex varieties $p:T\rightarrow S$, the pull-back (rigidified) connection $p^{\ast}(\nabla)$ on $p^{\ast}(\Lcal)$ also satisfies Weil reciprocity.
\end{definition}
To simplify the presentation, we will write $\nabla$ instead of $p^{\ast}(\nabla)$.
\begin{remark}
\hspace{2em}
\begin{enumerate}
\item There is a nonrigidified version of this definition. It is also possible to not assume a priori any compatibility with the holomorphic structure of $\Lcal$ (in this case, (WR) is much a stronger condition).
\item The assumption that $\Div\ell$ be disjoint from $\sigma$ is not essential, but it simplifies the proof of the theorem below. 
\item Condition \eqref{eq:2} is highly nontrivial: it relates a \emph{smooth} $(1,0)$ differential 1-form to a \emph{holomorphic} 1-form.
\end{enumerate}
\end{remark}

\begin{definition}\label{def:trace-connection}
A \emph{trace connection} for $\Lcal$ consists in giving, for every morphism of smooth quasi-projective complex varieties $p:T\rightarrow S$ and every holomorphic line bundle $\Mcal$ on $\Xcal_T$ of relative degree 0, a connection $\nabla_{\Mcal}$ on $\langle p^{\ast}(\Lcal),\Mcal\rangle$, compatible with the holomorphic structure on $\Lcal$, subject to the following conditions:
\begin{itemize}
	\item ({\sc Functoriality}) If $q:T'\rightarrow T$ is a morphism of smooth quasi-projective complex varieties, the base changed connection $q^{\ast}(\nabla_{\Mcal})$ corresponds to $\nabla_{q^{\ast}(\Mcal)}$ through the canonical isomorphism
	\begin{displaymath}
		q^{\ast}\langle p^{\ast}(\Lcal),\Mcal\rangle\isorightarrow \langle q^{\ast}p^{\ast}(\Lcal),q^{\ast}(\Mcal)\rangle.
	\end{displaymath}
	\item ({\sc Additivity}) Given $\nabla_{\Mcal}$ and $\nabla_{\Mcal'}$ as above, the connection $\nabla_{\Mcal\otimes \Mcal'}$ corresponds to the ``tensor product connection'' $\nabla_{\Mcal}\otimes\id+\id\otimes\nabla_{\Mcal'}$ through the canonical isomorphism
	\begin{displaymath}
		\langle p^{\ast}(\Lcal),\Mcal\otimes \Mcal'\rangle\isorightarrow\langle p^{\ast}(\Lcal),\Mcal\rangle\otimes\langle p^{\ast}(\Lcal),\Mcal'\rangle.
	\end{displaymath}
	\item ({\sc Compatibility with isomorphisms}) Given an isomorphism of line bundles (of relative degree 0) $\varphi:\Mcal\rightarrow \Mcal'$ on $\Xcal_T$, the connections $\nabla_\Mcal$ and $\nabla_\Mcal'$ correspond through the induced isomorphism on Deligne pairings
	\begin{displaymath}
		\langle\id,\varphi\rangle:\langle p^{\ast}(\Lcal),\Mcal\rangle\isorightarrow\langle p^{\ast}(\Lcal),\Mcal'\rangle.
	\end{displaymath}
\end{itemize}
We shall express a trace connection as an assignment $(p:T\rightarrow S,\Mcal)\mapsto\nabla_\Mcal$, or just $\Mcal\mapsto\nabla_\Mcal$.
\end{definition}

\begin{remark}
\hspace{2em}
\begin{enumerate}
\item It is easy to check that the additivity axiom implies that $\nabla_{\Ocal_\Xcal}$ corresponds to the trivial connection through the canonical isomorphism 
\begin{displaymath}
	\langle \Lcal,\Ocal_{\Xcal}\rangle\isorightarrow\Ocal_{S}.
\end{displaymath}
\item The compatibility with isomorphisms implies that $\nabla_{\Mcal}$ is invariant under the action of automorphisms of $\Mcal$ on $\langle p^{\ast}\Lcal,\Mcal\rangle$.
\end{enumerate}
\end{remark}
In case that $\Lcal$ is of relative degree 0, a trace connection for $\Lcal$ automatically satisfies an extra property that we will need in the next section. In this situation, for a line bundle on $\Xcal$ coming from the base $\pi^{\ast}(\Ncal)$, there is a canonical isomorphism
\begin{displaymath}
	\langle \Lcal,\pi^{\ast}(\Ncal)\rangle\isorightarrow\Ocal_{S}
\end{displaymath}
(see the proof in the lemma below). With these preliminaries at hand, we can state:
\begin{lemma}\label{lemma:trivial-connection}
Suppose that $\Lcal$ is of relative degree 0. Then, for any $(p:T\rightarrow S,\Mcal)$ as above, with $\Mcal=\pi_{T}^{\ast}(\Ncal)$, the connection $\nabla_{\Mcal}$ corresponds to the trivial connection through the canonical isomorphism $\langle p^{\ast}\Lcal,\Mcal\rangle\isorightarrow\Ocal_{T}$.
\end{lemma}
\begin{proof}
The statement is local for the Zariski topology on $T$, so we can localize and suppose there is a trivialization $\varphi:\Ncal\isorightarrow\Ocal_{T}$. This trivialization induces a trivialization $\widetilde{\varphi}:\pi_{T}^{\ast}(\Ncal)\isorightarrow\Ocal_{\Xcal_T}$. The isomorphism $\langle p^{\ast}\Lcal,\Mcal\rangle\isorightarrow\Ocal_{T}$ is such that there is commutative diagram:
\begin{displaymath}
	\xymatrix{
		\langle p^{\ast}\Lcal,\Mcal\rangle\ar[r]^{
		 \quad \smash{\raisebox{-0.5ex}{\ensuremath{\sim}}}
				}\ar[d]_{\langle\id,\widetilde{\varphi}\rangle}	&\Ocal_{T}\ar[d]^{\id}\\
		\langle p^{\ast}\Lcal,\Ocal_{\Xcal_T}\rangle\ar[r]^{\quad \smash{\raisebox{-0.5ex}{\ensuremath{\sim}}}}	&\Ocal_{T}.
	}
\end{displaymath}
Observe that if we change $\varphi$ by a unit in $\Ocal_{T}^{\times}$, then the degree 0 assumption on $\Lcal$ ensures $\langle\id,\widetilde{\varphi}\rangle$ does not change! This is compatible with the rest of the diagram being independent of $\varphi$. Now we combine: a) the compatibility of trace connections with isomorphisms, b) the triviality of $\nabla_{\Ocal_{\Xcal_T}}$ through the lower horizontal arrow, c) the commutative diagram. We conclude that $\nabla_{\Mcal}$ corresponds to the canonical connection through the upper horizontal arrow.
\end{proof}

Now for the characterization of connections satisfying (WR) universally in terms of trace connections:
\begin{theorem}\label{theorem:char-trace}
There is a bijection between the following type of data:
\begin{itemize}
	\item A rigidified connection $\nabla$ on $\Lcal$ satisfying (WR) universally.
	\item A trace connection for $\Lcal$.
\end{itemize}
Moreover, the correspondence between both is given by the following rule. Let $\nabla$ be a rigidified connection on $\Lcal$ satisfying (WR) universally. Let $p:T\rightarrow S$ be a morphism of smooth quasi-projective complex varieties, and $\Mcal$ a line bundle on $\Xcal_T$ of relative degree 0. If $\ell$ and $m$ are rational sections of $p^{\ast}(\Lcal)$ and $\Mcal$, whose divisors are \'etale and disjoint over an open Zariski subset $U\subset T$, and $\Div \ell$ is disjoint with $p^{\ast}\sigma$, then
\begin{equation}\label{eq:1}
	\nabla_{\Mcal}\langle \ell,m\rangle=\langle \ell,m\rangle\otimes\tr_{\Div m/U}\left(\frac{\nabla \ell}{\ell}\right).
\end{equation}
\end{theorem}
\begin{remark}
In Section  \ref{subsection:metrics-connections} we guessed the formula \eqref{eqn:derivative} from the Chern connection of metrics on Deligne pairings. The theorem above  shows that this is indeed the only possible construction of trace connections, once we impose some functoriality. The functoriality requirement is  natural, since Deligne pairings behave well with respect to base change. 
\end{remark}
\begin{proof}[Proof of Theorem \ref{theorem:char-trace}]
Given a rigidified connection satisfying (WR) universally, we already know that the rule \eqref{eq:1} defines a trace connection for $\Lcal$. Indeed, the condition (WR) guarantees that this rule is compatible  with both the Leibniz rule and the relations defining the Deligne pairing (Proposition \ref{proposition:Leibniz-WR}). The compatibility with the holomorphic structure of the trace connection is direct from the definition.

Now, let us consider a trace connection for $\Lcal$, i.e. the association 
\begin{displaymath}
	(p:T\rightarrow S,\Mcal)\mapsto\nabla_{\Mcal}\quad\text{on}\quad\langle p^{\ast}(\Lcal),\Mcal\rangle,
\end{displaymath}
for $\Mcal$ a line bundle on $\Xcal_T$ of relative degree 0. Let us consider the particular base change $\pi:\Xcal\rightarrow S$. The new family of curves is given by $p_1:\Xcal\times_{S} \Xcal\rightarrow \Xcal$, the projection onto the first factor. The base change of $\Lcal$ to $\Xcal\times_{S} \Xcal$ is the pull-back $p_{2}^{\ast}(\Lcal)$. The family $p_1$ comes equipped with two sections. The first one is the diagonal section, that we denote $\delta$. The second one, is the base change (or lift) of the section $\sigma$, that we write $\widetilde{\sigma}$. Hence, at the level of points, $\widetilde{\sigma}(x)=(x,\sigma\pi(x))$. The images of these sections are Cartier divisors in $\Xcal\times_{S}\Xcal$, so that they determine line bundles that we denote $\Ocal(\delta)$, $\Ocal(\widetilde{\sigma})$. Let us take for $\Mcal$ the line bundle $\Ocal(\delta-\widetilde{\sigma})$, namely $\Ocal(\delta)\otimes\Ocal(\widetilde{\sigma})^{-1}$. By the properties of the Deligne pairing, there is a canonical isomorphism
\begin{displaymath}
	\langle p_2^{\ast}(\Lcal),\Ocal(\delta-\widetilde{\sigma})\rangle\isorightarrow\delta^{\ast}p_{2}^{\ast}(\Lcal)\otimes\widetilde{\sigma}^{\ast}p_{2}^{\ast}\Lcal^{-1}.
\end{displaymath}
But now, $p_2 \delta=\id$, and $p_{2}\widetilde{\sigma}=\sigma\pi$. Using the rigidification $\sigma^{\ast}(\Lcal)\isorightarrow\Ocal_{S}$, we obtain an isomorphism 
\begin{equation}\label{eq:4}
	\langle p_2^{\ast}(\Lcal),\Ocal(\delta-\widetilde{\sigma})\rangle\isorightarrow\Lcal.
\end{equation}
Through this isomorphism, the connection $\nabla_{\Ocal(\delta-\widetilde{\sigma})}$ corresponds to a connection on $\Lcal$, that we temporarily write $\nabla_{S}$. It is compatible with the holomorphic structure, as trace connections are by definition. More generally, given a morphism of smooth quasi-projective complex varieties $p:T\rightarrow S$, the same construction applied to the base changed family $\Xcal_T\rightarrow T$ (with the base changed section $\sigma_T$) produces a connection $\nabla_T$ on $p^{\ast}(\Lcal)$, and it is clear that $\nabla_T=p^{\ast}(\nabla_S)$. We shall henceforth simply write $\nabla$ for this compatible family of connections. It is important to stress the role of the rigidification in the construction of $\nabla$.

First, we observe that the connection $\nabla$ is rigidified along $\sigma$. Indeed, on the one hand, by the functoriality of the Deligne pairing with respect to base change, there is a canonical isomorphism
\begin{equation}\label{eq:3}
	\sigma^{\ast}\langle p_{2}^{\ast}(\Lcal),\Ocal(\delta-\widetilde{\sigma})\rangle\isorightarrow
		\langle \Lcal,\sigma^{\ast}\Ocal(\delta-\widetilde{\sigma})\rangle\simeq\langle \Lcal,\Ocal_{S}\rangle.
\end{equation}
Here we have used the fact that the pull-backs of $\delta$ and $\widetilde{\sigma}$ by $\sigma$ both coincide with the section $\sigma$ itself, so that $\sigma^{\ast}\Ocal(\delta-\widetilde{\sigma})\simeq\Ocal_{S}$. On the other hand, the functoriality assumption on trace connections and compatibility with isomorphisms ensure that through the isomorphism \eqref{eq:3} we have an identification
\begin{displaymath}
	\sigma^{\ast}(\nabla_{\Ocal(\delta-\widetilde{\sigma})})=\nabla_{\Ocal_S}.
\end{displaymath}
But we already remarked that $\nabla_{\Ocal_S}$ corresponds to the trivial connection through the isomorphism
\begin{equation}\label{eq:5}
	\langle \Lcal,\Ocal_{S}\rangle\isorightarrow\Ocal_S.
\end{equation}
Now, the rigidification property for $\nabla$ follows, since the composition of $\sigma^{\ast}$\eqref{eq:4}--\eqref{eq:5} gives back our fixed isomorphism $\sigma^{\ast}(\Lcal)\isorightarrow\Ocal_S$.

Second, we show that $\nabla$ satisfies  (WR) universally. Actually, we will see that for $(p:T\rightarrow S,\Mcal)$, the connection $\nabla_{\Mcal}$ is given by the rule
\begin{displaymath}
	\nabla_{\Mcal}\langle \ell,m\rangle=\langle \ell,m\rangle\otimes\tr_{\Div m/U}\left(\frac{\nabla \ell}{\ell}\right),
\end{displaymath}
for sections $\ell$ and $m$ as in the statement. 
Using the fact that $\nabla_\Mcal$ is a connection (and hence satisfies the Leibniz rule) and imposing the relations defining the Deligne pairing, this ensures that  (WR) for $\nabla$ is satisfied. 

To simplify the discussion, and because the new base $T$ will be fixed from now on, we may just change the meaning of the notation and write $S$ instead of $T$. Also, observe that the equality of two differential forms can be checked after \'etale base change (because \'etale base change induces isomorphisms on the level of differential forms). Therefore, after possibly localizing $S$ for the \'etale topology, we can suppose that
\begin{displaymath}
	\Div m=\sum_{i}n_{i}D_i\ ,
\end{displaymath} 
where the divisors $D_i$ are given by sections $\sigma_i$, and the $n_i$ are integers with $\sum_i n_i=0$. Because
\begin{displaymath}
	\Mcal\simeq\Ocal(\sum_{i}n_{i}D_i)\simeq\bigotimes_{i}\Ocal(\sigma_{i}-\sigma)^{\otimes n_i},
\end{displaymath}
the additivity of the trace connection and the trace $\tr_{\Div m/S}$ with respect to $m$, and the compatibility with isomorphisms, we reduce to the case where $\Mcal=\Ocal(\sigma_{i}-\sigma)$ and $m$ is the canonical rational section ${\mathbbm 1}$ with divisor $\sigma_i-\sigma$. In order to trace back the definition of $\nabla$, we effect the base change $\Xcal\rightarrow S$. By construction of the connection $\nabla$ on $\Lcal$ (that, recall, involves the rigidification), we have
\begin{equation}\label{eq:6}
	\frac{\nabla_{\Ocal(\delta-\widetilde{\sigma})}\langle p_{2}^{\ast}\ell,{\mathbbm 1}\rangle}{\langle p_{2}^{\ast}\ell,{\mathbbm 1}\rangle}
	=\frac{\nabla \ell}{\ell}-\pi^{\ast}\left(\frac{d\sigma^{\ast}(\ell)}{\sigma^{\ast}\ell}\right),
\end{equation}
where we identify $\sigma^{\ast}(\ell)$ with a rational function on $S$ through the rigidification. We now pull-back the identity \eqref{eq:6} by $\sigma_i$, and for this we remark that the pull-back of $\delta$ by $\sigma_i$ becomes $\sigma_i$, while the pull-back of $\widetilde{\sigma}$ by $\sigma_i$ becomes $\sigma$! Taking this into account, together with the functoriality of Deligne pairings and trace connections, we obtain
\begin{displaymath}
	\begin{split}
		\frac{\nabla_{\Ocal(\sigma_i-\sigma)}\langle \ell,{\mathbbm 1}\rangle}{\langle \ell,{\mathbbm 1}\rangle}&=\sigma_{i}^{\ast}\left(\frac{\nabla \ell}{\ell}\right)
	-\sigma_{i}^{\ast}\pi^{\ast}\left(\frac{d\sigma^{\ast}(\ell)}{\sigma^{\ast}\ell}\right)\\
	&=\sigma_{i}^{\ast}\left(\frac{\nabla \ell}{\ell}\right)-\sigma^{\ast}\left(\frac{\nabla \ell}{\ell}\right)\\
	&=\tr_{\Div{\mathbbm 1}/S}\left(\frac{\nabla \ell}{\ell}\right).
	\end{split}
\end{displaymath}
In the second inequality we used that $\nabla$ is rigidified along $\sigma$, that we already showed above. This completes the proof of the theorem.
\end{proof}

\begin{remark}
Notice that the above notions do not require  $\Lcal$ to have relative degree 0. It may well be that the objects we introduce do not exist at such a level of generality. In the relative degree 0 case we have shown that connections satisfying (WR) universally on $\Lcal$ do indeed exist and can be constructed from relative connections that are compatible with the holomorphic structure of $\Lcal$. The latter, of course, always exist by taking  the Chern connection of a hermitian metric.
In the next section, we confirm the existence in relative degree 0 by other methods, and we classify them all.
\end{remark}

\begin{corollary}
Let $\Mcal\mapsto\nabla_{\Mcal}$ be a trace connection for the rigidified line bundle $\Lcal$. Then there is a unique extension of the trace connection to line bundles $\Mcal$ of arbitrary relative degree, such that $\nabla_{\Ocal(\sigma)}$ corresponds to the trivial connection through the isomorphism $\langle\Lcal,\Ocal(\sigma)\rangle\simeq\sigma^{\ast}(\Lcal)\simeq\Ocal_{S}$. This extension satisfies the following properties:
\begin{enumerate}
	\item if $\nabla$ is the connection on $\Lcal$ determined by Theorem \ref{theorem:char-trace}, the extension is still given by the rule \eqref{eq:1};
	\item the list of axioms of Definition \ref{def:trace-connection}, i.e. functoriality, additivity and compatibility with isomorphisms.
\end{enumerate}
\end{corollary}
\begin{proof}
Let $\nabla$ be the rigidified connection on $\Lcal$ corresponding to the trace connection $\Mcal\mapsto\nabla_{\Mcal}$. Then we extend the trace connection to arbitrary $\Mcal$ by the rule \eqref{eq:1}. The claims of the corollary are straightforward to check. 
\end{proof}

\subsection{Reformulation in terms of Poincar\'e bundles}\label{subsection:structure}
In case $\Lcal$ is of relative degree 0, the notion of trace connection can be rendered more compact by the introduction of a Poincar\'e bundle on the relative jacobian. Let $\pi:\Xcal\rightarrow S$ be our smooth fibration in proper curves, with a fixed section $\sigma:S\rightarrow \Xcal$. We write $p:J\rightarrow S$ for the relative jacobian $J=J(\Xcal/S)$, and $\mathcal{P}$ for the Poincar\'e bundle on $\Xcal\times_{S} J$, rigidified along the lift $\widetilde{\sigma}:J\rightarrow \Xcal\times_{S}J$ of the section $\sigma$. This rigidified Poincar\'e bundle has a neat compatibility property with respect to the group scheme structure of $J$. Let us introduce the addition map
\begin{displaymath}
	\mu:J\times_{S} J\longrightarrow J.
\end{displaymath}
If $T\rightarrow S$ is a morphism of schemes, then at the level of $T$ valued points the addition map is induced by the tensor product of line bundles on $\Xcal_{T}$. If $p_1,p_2$ are the projections of $J\times_{S} J$ onto the first and second factors, then there is an isomorphism of line bundles on $\Xcal\times_{S}J\times_{S}J$
\begin{equation}\label{eq:7}
	\mu^{\ast}\mathcal{P}\isorightarrow (p_{1}^{\ast}\mathcal{P})\otimes (p_{2}^{\ast}\mathcal{P}).
\end{equation}
In particular, given a line bundle $\Lcal$ on $\Xcal$ and its pull-back $\widetilde{\Lcal}$ to $\Xcal\times_{S}J\times_{S}J$, there is an induced canonical isomorphism of Deligne pairings
\begin{displaymath}
	\langle \widetilde{\Lcal},\mu^{\ast}\mathcal{P}\rangle\isorightarrow \langle\widetilde{\Lcal},p_{1}^{\ast}\mathcal{P}\rangle
	\otimes\langle\widetilde{\Lcal},p_{1}^{\ast}\mathcal{P}\rangle.
\end{displaymath}

Let $\Mcal\mapsto\nabla_{\Mcal}$ be a trace connection for $\Lcal$. Then, we can evaluate it on the data $(p:J\rightarrow S,\mathcal{P})$, thus providing a connection $\nabla_{\mathcal{P}}$ on $\langle p^{\ast}\Lcal,\mathcal{P}\rangle$. By the functoriality of trace connections, we have
\begin{displaymath}
	\mu^{\ast}\nabla_{\mathcal{P}}=\nabla_{\mu^{\ast}\mathcal{P}},\quad p_{1}^{\ast}\nabla_{\mathcal{P}}=\nabla_{p_{1}^{\ast}\mathcal{P}},
	\quad p_{2}^{\ast}\nabla_{\mathcal{P}}=\nabla_{p_{2}^{\ast}\mathcal{P}}.
\end{displaymath}
Furthermore, by the compatibility with isomorphisms and additivity, there is an identification through \eqref{eq:7}
\begin{displaymath}
	\mu^{\ast}\nabla_{\mathcal{P}}=(p_{1}^{\ast}\nabla_{\mathcal{P}})\otimes\id + \id\otimes (p_{2}^{\ast}\nabla_{\mathcal{P}}).
\end{displaymath}
We claim the data $\Mcal\mapsto\nabla_{\Mcal}$ is determined by $\nabla_{\mathcal{P}}$. For if $q:T\rightarrow S$ is a morphism of smooth quasi-projective complex varieties, and $\Mcal$ a line bundle on $\Xcal_T$ of relative degree 0, then we have a classifying morphism $\varphi:T\rightarrow J$ and an isomorphism
\begin{displaymath}
	\Mcal\isorightarrow(\varphi^{\ast}\mathcal{P})\otimes(\pi_{T}^{\ast}\sigma_{T}^{\ast}\Mcal)
\end{displaymath}
This isomorphism induces an isomorphism on Deligne pairings
\begin{displaymath}
	\langle q^{\ast}\Lcal,\Mcal\rangle\isorightarrow\langle q^{\ast}\Lcal, (\varphi^{\ast}\mathcal{P})\otimes(\pi_{T}^{\ast}\sigma_{T}^{\ast}\Mcal)\rangle.
\end{displaymath}
But now, because $\Lcal$ is of relative degree 0, there is a canonical isomorphism
\begin{displaymath}
	\langle q^{\ast}\Lcal,\pi_{T}^{\ast}\sigma_{T}^{\ast}\Mcal\rangle\isorightarrow\Ocal_{T},
\end{displaymath}
and the connection $\nabla_{\pi_{T}^{\ast}\sigma_{T}^{\ast}\Mcal}$ is trivial by Lemma \ref{lemma:trivial-connection}. Hence, through the resulting isomorphism on Deligne pairings
\begin{displaymath}
	\langle q^{\ast}\Lcal,\Mcal\rangle\isorightarrow \varphi^{\ast}\langle p^{\ast}\Lcal,\mathcal{P}\rangle,
\end{displaymath}
we have an identification of connections
\begin{displaymath}
	\varphi^{\ast}\nabla_{\mathcal{P}}=\nabla_{\Mcal}.
\end{displaymath}
Observe moreover that this identification does not depend on the precise isomorphism
\begin{displaymath}
	\Mcal\isorightarrow(\varphi^{\ast}\mathcal{P})\otimes(\pi_{T}^{\ast}\sigma_{T}^{\ast}\Mcal),
\end{displaymath}
by the compatibility of trace connections with isomorphisms of line bundles (and hence with automorphisms of line bundles). The next statement is now clear.
\begin{proposition}
Suppose that $\Lcal$ is of relative degree 0. The following data are equivalent:
\begin{itemize}
	\item A trace connection for $\Lcal$.
	\item A connection $\nabla_{\mathcal{P}}$ on the Deligne pairing $\langle p^{\ast}\Lcal,\mathcal{P}\rangle$ (compatible with the holomorphic structure), satisfying the following compatibility with addition on $J$:
	\begin{equation}\label{eq:8}
			\mu^{\ast}\nabla_{\mathcal{P}}=(p_{1}^{\ast}\nabla_{\mathcal{P}})\otimes\id + \id\otimes (p_{2}^{\ast}\nabla_{\mathcal{P}}). 
	\end{equation}
\end{itemize}
The correspondence is given as follows. Let $\nabla_{\mathcal{P}}$ be a connection on $\langle p^{\ast}\Lcal,\mathcal{P}\rangle$ satisfying \eqref{eq:8}. Let $q:T\rightarrow S$ be a morphism of smooth quasi-projective complex varieties, $\Mcal$ a line bundle on $\Xcal_T$ of relative degree 0 and $\varphi:T\rightarrow J$ its classifying map, so that there are isomorphisms
\begin{displaymath}
	\Mcal\isorightarrow \varphi^{\ast}\mathcal{P}\otimes(\pi_{T}^{\ast}\sigma_{T}^{\ast}\Mcal),\quad\langle q^{\ast}\Lcal,\Mcal\rangle\isorightarrow \varphi^{\ast}\langle p^{\ast}\Lcal,\mathcal{P}\rangle.
\end{displaymath}
Then, through these identifications, the rule
\begin{displaymath}
	\Mcal\longmapsto \varphi^{\ast}\nabla_{\mathcal{P}}\quad\text{on}\quad \langle q^{\ast}\Lcal,\Mcal\rangle
\end{displaymath}
defines a trace connection for $\Lcal$.
\end{proposition}
\noindent The proposition justifies  calling $\nabla_{\mathcal{P}}$ a \emph{universal trace connection}.

The formulation of trace connections in terms of Poincar\'e bundles makes it easy to deal with the uniqueness issue. Let $\Mcal\mapsto\nabla_{\Mcal}$ and $\Mcal\mapsto\nabla_{\Mcal}^{\prime}$ be trace connections. These are determined by the respective ``universal'' connections $\nabla_{\mathcal{P}}$ and $\nabla_{\mathcal{P}}^{\prime}$. Two connections on a given holomorphic line bundle, compatible with the holomorphic structure, differ by a smooth $(1,0)$ differential one form. Let $\theta$ be the smooth $(1,0)$ form on $J$ given by $\nabla_{\mathcal{P}}-\nabla_{\mathcal{P}}^{\prime}$. Then, the compatibility of universal trace connections with additivity imposes the restriction on $\theta$
\begin{displaymath}
	\mu^{\ast}\theta=p_{1}^{\ast}\theta + p_{2}^{\ast}\theta.
\end{displaymath}
We say that $\theta$ is a translation invariant form, and we write the space of such differential forms $\Inv(J)^{(1,0)}\subset \Gamma(J, A^{1,0}_{J})$.
\begin{proposition}
The space of trace connections for a line bundle $\Lcal$ on $\Xcal$, of relative degree 0, is a torsor under $\Inv(J)^{(1,0)}$.
\end{proposition}
It remains to consider the problem of existence. The line bundle $\langle p^{\ast}\Lcal,\mathcal{P}\rangle$ on $J$ is rigidified along the zero section and compatible with the relative addition law on $J$. In particular, $\langle p^{\ast}\Lcal,\mathcal{P}\rangle$ lies in $J^{\vee}(S)$, the $S$-valued points of the dual abelian scheme to $J$. Equivalently, it is a line bundle on $J$ of relative degree 0. Then $\langle p^{\ast}\Lcal,\mathcal{P}\rangle$ admits a smooth unitary connection compatible with the addition law. This connection is compatible with the holomorphic structure. We thus arrive at the following theorem.
\begin{theorem}[\sc Structure theorem] \label{theorem:existence-unicity}
The space of trace connections for $\Lcal$, and thus of rigidified connections on $\Lcal$ satisfying (WR) universally, is a nonempty torsor under $\Inv(J)^{(1,0)}$.
\end{theorem}
In the Section \ref{section:proofs} we will provide a constructive approach to Theorem \ref{theorem:existence-unicity}. 
\begin{remark}
For the amateurs of stacks, we sketch  a restatement of Theorem \ref{theorem:char-trace} in a categorical language. Let us consider the category $(Sch/S)$ of schemes over $S$, endowed with the \'etale topology, and the following categories fibered in groupoids over $(Sch/S)$:
\begin{enumerate}
	\item $\PICRIG^{0}(\Xcal/S)$ is the category of line bundles on (base changes of) $\Xcal$, of relative degree 0 and rigidified along $\sigma$. This is where our line bundle $\Lcal$ naturally lives;
	\item $\Pic^{0}(\Xcal/S)$ is the usual Picard functor of relative degree 0 line bundles, modulo algebraic equivalence;
	\item $\PIC(S)$ is the groupoid of line bundles on schemes over $S$.
\end{enumerate}
A similar strategy as in the proof of Theorem \ref{theorem:char-trace} shows that Deligne's pairing provides an isomorphism
\begin{displaymath}
	\begin{split}
		\PICRIG(\Xcal/S)&\overset{\sim}{\longrightarrow}\HOM(\Pic^{0}(\Xcal/S),\PIC(S))\\
		\Lcal &\longmapsto (\Mcal\mapsto\langle\Lcal,\Mcal\rangle).
	\end{split}
\end{displaymath}
For simplicity, we have ignored base changes. Let us now restrict these groupoids to the full subcategory of $(Sch/S)$ formed by smooth schemes over $\mathbb{C}$ (not necessarily smooth over $S$). There are variants of our groupoids that implement connections:
\begin{enumerate}
	\item $\PICRIG^{(WR)}(\Xcal/S)$ is the category of line bundles on (base changes of) $\Xcal$, of relative degree $0$ and rigidified along $\sigma$, equipped with compatible rigidified connections satisyfing (WR) universally;
	\item $\PIC^{\sharp}(S)$ the groupoid of line bundles with compatible connections.
\end{enumerate}
In these groupoids, natural arrows are by definition flat isomorphisms. The content of Theorem \ref{theorem:char-trace} (for line bundles of relative degree 0) is that Deligne's pairing establishes an isomorphism:
\begin{displaymath}
	\begin{split}
		\PICRIG^{(WR)}(\Xcal/S)&\overset{\sim}{\longrightarrow}\HOM(\Pic^{0}(\Xcal/S),\PIC^{\sharp}(S))\\
			(\Lcal,\nabla)&\longmapsto (\Mcal\mapsto (\langle\Lcal,\Mcal\rangle,\nabla_{\Mcal})).
	\end{split}
\end{displaymath}
Implicit in this statement  is the extra property satisfied by trace connections shown in Lemma \ref{lemma:trivial-connection}. It is in the application of the lemma that we need the relative degree 0 assumption on $\Lcal$.
\end{remark}

\subsection{Intersection connections}\label{subsection:intersection}
We continue with the  notation of the previous sections. In Theorem \ref{theorem:char-trace} we have related rigidified connections on $\Lcal$ satisfying (WR) universally and trace connections for $\Lcal$, as equivalent notions. We have also given a structure theorem for the space of such objects (Theorem \ref{theorem:existence-unicity}). There is, of course,  a lack of symmetry in the definition of a trace connection $\Mcal\mapsto\nabla_{\Mcal}$, since the holomorphic line bundles $\Mcal$ require no extra structure. In this section, we show that given a relatively flat connection on $\Lcal$ satisfying (WR) universally, we can build an \emph{intersection connection} ``against'' line bundles with connections $(\Mcal,\nabla^{\prime})$. Moreover, if $\nabla'$ is relatively flat and also satisfies (WR), then the resulting connection is symmetric with respect to the symmetry of the Deligne pairing.
In the following, $F_{\nabla}$ stands for the curvature of a connection $\nabla$.

\begin{theorem}\label{theorem:int-connection}
Let $\nabla$  be a  connection on $\Lcal\to\Xcal$ satisfying (WR), not necessarily rigidified, and such that its vertical projection $\nabla_{\Xcal/S}$ is flat. Let $\nabla'$ be a connection on another line bundle $\Mcal\to \Xcal$ of arbitrary relative degree. Then:
\begin{enumerate}
\item on the Deligne pairing $\langle\Lcal,\Mcal\rangle$, the following rule defines a connection compatible with the holomorphic structure:
\begin{equation} \label{eqn:int-connection}
	\frac{D\langle\ell,m\rangle}{\langle\ell,m\rangle}=\frac{i}{2\pi}\pi_{\ast}\left(\frac{\nabla^{\prime}m}{m}\wedge F_{\nabla}\right)+\tr_{\Div m/S}\left(\frac{\nabla\ell}{\ell}\right);
\end{equation}
\item if both connections $\nabla$ and $\nabla^{\prime}$ are unitary, then $D$ is the Chern connection of the corresponding metrized Deligne pairing;
\item if $\nabla$ satisfies (WR) universally, then the construction of $D$ is compatible with base change and coincides with a trace connection when restricted to line bundles with unitary connections $(\Mcal,\nabla^{\prime})$.
\end{enumerate}
\end{theorem}
\begin{proof}
To justify that the rule $D$ defines a connection, it is enough to show the compatibility between Leibniz' rule and the relations defining the Deligne pairing. One readily checks the compatibility for the change $\ell\mapsto f\ell$ for $f$ a rational function. For the change $m\mapsto fm$, the trace term in the definition of $D$ already satisfies (WR). We thus have to show the invariance of the fiber integral under the change $m\mapsto fm$, or equivalently
\begin{equation}\label{eq:vanishing-integral}
	\pi_{\ast}\left(\frac{df}{f}\wedge F_{\nabla}\right)=0.
\end{equation}
It will be useful to compare $\nabla$ to the Chern connection $\nabla_{ch}$ on $\Lcal$, which is relatively flat (see Example \ref{ex:chern}; the rigidification is irrelevant for this discussion). The connection $\nabla_{ch}$ also satisfies (WR) (see Example \ref{ex:chern-reciprocity}). We write
\begin{displaymath}
	\nabla=\nabla_{ch}+\theta,\quad F_{\nabla}=F_{\nabla_{ch}}+d\theta.
\end{displaymath}
Then $\theta$ is of type $(1,0)$ and has vanishing trace along rational divisors (we call this the \emph{Weil vanishing property}). We exploit this fact, together with the observation that since $F_{\nabla_{ch}}$ is of type $(1,1)$ it is $\partial$-closed. Write:
\begin{align}
		\pi_{\ast}\left(\frac{df}{f} \wedge F_{\nabla}\right)&=\pi_{\ast}\left(\frac{df}{f} \wedge F_{\nabla_{ch}}\right)
		+\pi_{\ast}\left(\frac{df}{f}\wedge d\theta\right) \notag \\
		&=\partial\pi_{\ast}(\log|f|^{2} F_{\nabla_{ch}})
		+\pi_{\ast}\left(\frac{df}{f}\wedge d\theta\right). \label{eq:vanishing-integral-2}
\end{align}
Because $F_{\nabla_{ch}}$ is  flat on fibers,
\begin{equation}\label{eq:vanishing-integral-3}
	\pi_{\ast}(\log|f|^{2} F_{\nabla_{ch}})=0.
\end{equation}
Furthermore, by the Poincar\'e-Lelong formula for currents on $\Xcal$
\begin{displaymath}
	\overline{\partial}\partial\log |f|^{2}=2\pi i\cdot\delta_{\Div f}.
\end{displaymath}
Hence, by type considerations,
\begin{align}
	\pi_{\ast}\left(\frac{df}{f}\wedge d\theta\right)&=
	\pi_{\ast}\left(\partial\log|f|^2\wedge \dbar\theta\right) \notag \\
	&= \pi_\ast\left(\dbar\partial \log|f|^2\wedge\theta\right)-\dbar\pi_{\ast}\left(\partial\log|f|^2\wedge \theta\right)\notag \\
	&=
	2\pi i\tr_{\Div f/S}(\theta)
	-\dbar\pi_{\ast}\left(\partial\log|f|^2\wedge \theta\right)\ .\label{eq:vanishing-integral-4}
\end{align}
The trace term vanishes by the Weil vanishing property of $\theta$, and the second term vanishes because the integrand it is of type $(2,0)$ and $\pi$ reduces types by $(1,1)$. Combining \eqref{eq:vanishing-integral-2}--\eqref{eq:vanishing-integral-4} we conclude with the desired \eqref{eq:vanishing-integral}.

The second assertion follows by construction of $D$, and the third  item is immediate.
\end{proof}
\begin{definition}[\sc Intersection Connection] \label{def:intersection-connection}
The connection $\nabla^{int}_{ \langle\Lcal,\Mcal\rangle}:= D$ on $\langle\Lcal,\Mcal\rangle$ constructed in Theorem \ref{theorem:int-connection} is called \emph{the intersection connection attached to $(\Lcal,\nabla)$ and $(\Mcal,\nabla^{\prime})$}. We write $\langle (\Lcal,\nabla),(\Mcal,\nabla^{\prime})\rangle$ for the Deligne pairing of $\Lcal$ and $\Mcal$ with the intersection connection.
\end{definition}

We have the following
\begin{proposition} \label{prop:curv-inter}
The curvature of the intersection connection $\nabla^{int}_{ \langle\Lcal,\Mcal\rangle}$ attached to $(\Lcal,\nabla)$ and $(\Mcal,\nabla^{\prime})$ is given by
\begin{displaymath}
	F_{\nabla^{int}_{ \langle\Lcal,\Mcal\rangle}}=\frac{i}{2\pi}\pi_{\ast}(F_{\nabla}\wedge F_{\nabla^{\prime}}).
\end{displaymath}
\end{proposition}
\begin{proof}
Let $\langle\ell,m\rangle$ be a symbol providing a local frame for the Deligne pairing $\langle\Lcal,\Mcal\rangle$ (after possibly shrinking $S$). The current $T$ of integration against $\nabla^{\prime}m/m$ satisfies a Poincar\'e-Lelong type equation
\begin{equation}\label{eq:Poincare-Lelong}
	\frac{i}{2\pi}dT+\delta_{\Div m}=\frac{i}{2\pi}F_{\nabla^{\prime}}.
\end{equation}
This is a routine consequence of Stokes' theorem. Therefore, we find an equality
\begin{equation}\label{eq:int-conn-1}
	\frac{i}{2\pi}d\pi_{\ast}\left(\frac{\nabla^{\prime}m}{m}\wedge F_{\nabla}\right)=\frac{i}{2\pi}\pi_{\ast}(F_{\nabla}\wedge F_{\nabla^{\prime}})-
	\tr_{\Div m/S}(F_{\nabla}).
\end{equation}
On the other hand
\begin{equation}\label{eq:int-conn-2}
	d\tr_{\Div m/S}\left(\frac{\nabla\ell}{\ell}\right)=\tr_{\Div m/S}(F_{\nabla}).
\end{equation}
From equations \eqref{eq:int-conn-1}--\eqref{eq:int-conn-2} and the definition of $\nabla^{int}_{ \langle\Lcal,\Mcal\rangle}$, we conclude
\begin{align*}
	F_{\nabla^{int}_{ \langle\Lcal,\Mcal\rangle}}&=d\left(\frac{D\langle\ell,m\rangle}{\langle\ell,m\rangle}\right) \\
	&=\frac{i}{2\pi}\pi_{\ast}(F_{\nabla}\wedge F_{\nabla^{\prime}})-
	\tr_{\Div m/S}(F_{\nabla})+\tr_{\Div m/S}(F_{\nabla})\\
	&=\frac{i}{2\pi}\pi_{\ast}(F_{\nabla}\wedge F_{\nabla^{\prime}}),
\end{align*}
as was to be shown.
\end{proof}
Intersection connections satisfy the expected behavior with respect to tensor product and flat isomorphisms. Furthermore, if $\Lcal$ and $\Mcal$ are line bundles endowed with connections $\nabla$, $\nabla^{\prime}$ that satisfy (WR), one might expect a symmetry of the intersection connections on $\langle\Lcal,\Mcal\rangle$ and $\langle\Mcal,\Lcal\rangle$, through the canonical isomorphism of Deligne pairings
\begin{equation}\label{eq:deligne-sym}
	\langle\Lcal,\Mcal\rangle\isorightarrow\langle\Mcal,\Lcal\rangle.
\end{equation}
This is indeed the case.
\begin{proposition}\label{prop:sym-int-conn}
Suppose the connections $\nabla$, $\nabla^{\prime}$ on $\Lcal$ and $\Mcal$ both satisfy (WR) and are relatively flat. Then, the symmetry isomorphism \eqref{eq:deligne-sym} is flat with respect to the intersection connections.
\end{proposition}
\begin{proof}
Let $\ell,m$ be a couple of sections providing bases elements $\langle\ell,m\rangle$ and $\langle m,\ell\rangle$ of $\langle\Lcal,\Mcal\rangle$ and $\langle\Mcal,\Lcal\rangle$, respectively. We denote by $T_{\ell}$ and $T_{m}$ the currents of integration against $\nabla\ell/\ell$ and $\nabla^{\prime}m/m$, respectively. These currents have disjoint wave front sets. The same holds for the Dirac currents $\delta_{\Div\ell}$ and $\delta_{\Div m}$, as well as $\delta_{\Div\ell}$ and $\nabla^{\prime} m/m$, etc. For currents with disjoint wave front sets, the usual wedge product rules and Stokes' formulas for differential forms remain true. Applying the Poincar\'e-Lelong type equations for $T_{\ell}$ and $T_{m}$ (see \eqref{eq:Poincare-Lelong}), we find the chain of equalities
\begin{displaymath}
	\begin{split}
		\frac{\nabla^{int}_{ \langle\Lcal,\Mcal\rangle}\langle\ell,m\rangle}{\langle\ell,m\rangle}&:=\frac{i}{2\pi}\pi_{\ast}\left(\frac{\nabla^{\prime}m}{m}\wedge F_{\nabla}\right)+\tr_{\Div m/S}\left(\frac{\nabla\ell}{\ell}\right)\\
		&=\frac{i}{2\pi}\pi_{\ast}\left(\frac{\nabla^{\prime}m}{m}\wedge dT_{\ell}\right)+\tr_{\Div\ell/S}\left(\frac{\nabla^{\prime}m}{m}\right)
		+\tr_{\Div m/S}\left(\frac{\nabla\ell}{\ell}\right)\\
		&=\frac{i}{2\pi}\pi_{\ast}\left(\frac{\nabla\ell}{\ell}\wedge dT_{m}\right)+\tr_{\Div m/S}\left(\frac{\nabla\ell}{\ell}\right)
		+\tr_{\Div\ell/S}\left(\frac{\nabla^{\prime}m}{m}\right)\\
		&=\frac{i}{2\pi}\pi_{\ast}\left(\frac{\nabla \ell}{\ell}\wedge F_{\nabla^{\prime}}\right)+\tr_{\Div\ell/S}\left(\frac{\nabla^{\prime} m}{m}\right)\\
		&=:\frac{\nabla^{int}_{ \langle\Mcal,\Lcal\rangle}\langle m,\ell\rangle}{\langle m,\ell\rangle}.
	\end{split}
\end{displaymath}
The proof is complete.
\end{proof}
For later use, it will be useful to study the change of intersection connections under change of connection on $\Lcal$.
\begin{proposition}\label{prop:indep-int-conn}
Assume that $\Mcal$ has relative degree 0. Let $\theta\in\Gamma(S,\Acal^{1,0}_{S})$. Then
\begin{displaymath}
	\langle (\Lcal,\nabla+\pi^{\ast}\theta),(\Mcal,\nabla^{\prime})\rangle
	=\langle(\Lcal,\nabla),(\Mcal,\nabla^{\prime})\rangle.
\end{displaymath}
\end{proposition}
\begin{proof}
Let $\langle\ell,m\rangle$ be a local basis of the Deligne pairing. We observe that
\begin{displaymath}
	\pi_{\ast}\left(\frac{\nabla^{\prime}m}{m}\wedge d\left(\pi^{\ast}\theta\right)\right)
	=\pi_{\ast}\left(\frac{\nabla^{\prime}m}{m}\right)\wedge d\theta=0.
\end{displaymath}
The vanishing of the last fiber integral is obtained by counting types: $\nabla^{\prime}m/m$ is of type $(1,0)$ and $\pi$ reduces types by $(1,1)$. Also, because $\Div m$ is of degree 0, we have
\begin{displaymath}
	\tr_{\Div m/S}(\pi^{\ast}\theta)=(\deg\Div m)\theta=0.
\end{displaymath}
These observations together imply the proposition.
\end{proof}

\section{Proof of the Main Theorems}\label{section:proofs} 
We place ourselves in the setting of Theorems \ref{thm:main} and \ref{thm:intersection}. Hence, $\pi:\Xcal\rightarrow S$ is a smooth proper morphism of smooth quasi-projective complex varieties with connected fibers of relative dimension 1 and genus $g$. We suppose $\sigma:S\rightarrow\Xcal$ is a given section. Let $\Lcal$ be a line bundle on $\Xcal$, rigidified along $\sigma$. And let  $\nabla_{\Xcal/S}$ be a flat relative connection. As noted previously,  the restrictions of $\nabla_{\Xcal/S}$ to fibers are holomorphic connections. In this section we construct a global extension of $\nabla_{\Xcal/S}$ that satisfies (WR) universally and is rigidified along $\sigma$. By Theorem \ref{theorem:existence-unicity}, this is equivalent to attaching to $(\Lcal,\nabla_{\Xcal/S})$ a trace connection for $\Lcal$. It also gives rise to intersection connections. The construction is local on the base for the analytic topology, and makes use of Gauss-Manin invariants (deformation theory). This is carried out in Section \ref{sec:canonical}. The uniqueness of the extension is discussed in Section \ref{sec:unique}, and is based on a general vanishing lemma for differential forms satisfying a ``Weil vanishing property''. The local construction given in Section \ref{sec:canonical} is well suited to curvature computations of trace and intersection connections on Deligne pairings, and an important case is studied in Section \ref{sec:curvature}.

\subsection{The canonical extension: local description and properties} \label{sec:canonical}
In this step, we work locally on $S$ for the analytic topology. We replace $S$ by a contractible open subset $S^{\circ}$. Hence, local systems over $S^{\circ}$ are trivial. We write $\Xcal^{\circ}$ for the restriction of $\Xcal$ to $S^{\circ}$, but to ease the notation, we still denote $\Lcal$ for the restriction of $\Lcal$. For later use (in the proof of Theorem \ref{thm:extension}), we fix a family of symplectic bases $\{\alpha_i,\beta_i\}_{i=1}^g$ for $H_1(\Xcal_s)$, that is flat with respect to the Gauss-Manin connection. Observe that this trivially determines a symplectic horizontal basis after any base change $T\rightarrow S^{\circ}$. We may assume that these are given by closed curves based at $\sigma(s)$. We view these curves as the polygonal  boundary of a fundamental domain $\Fcal_s\subset \widetilde\Xcal_s$
in the local relative universal cover $\widetilde\Xcal\to \Xcal^{\circ}$, in the usual way:

\setlength{\unitlength}{1cm}
\begin{picture}(12,5)
\put(7.05,3.65){$\beta_1$}
\put(7.05,1.4){$\beta_1$}
\put(7.35,2.6){$\alpha_1$}
\put(5.4,.7){$\alpha_1$}
\put(4.4,.95){$\bullet$}
\put(4.2,.6){$\tilde\sigma(s)$}
\put(3.5,1){
{\scalebox{.35}{\includegraphics{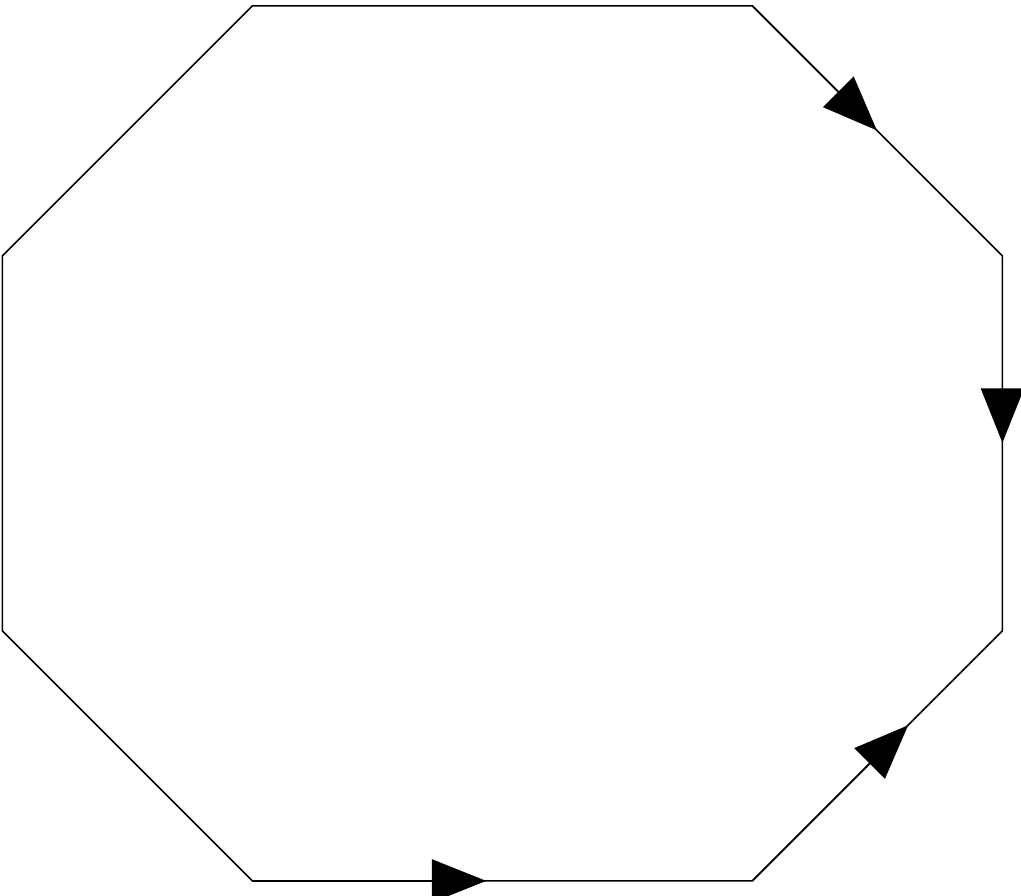}}}}
\end{picture}\\
In the figure we have written $\tilde{\sigma}(s)$ for a lift of the section $\sigma(s)$ to a fundamental domain $\Fcal_{s}$. Let $\nu:S\rightarrow E(\Xcal/S)$ be the $\Ccal^{\infty}$ classifying map corresponding to $(\Lcal,\nabla_{\Xcal/S})$. By the choice of $S^{\circ}$, $\nu\mid_{S^{\circ}}$ lifts to $\tilde{\nu}:S^{\circ}\rightarrow H^{1}_{dR}(\Xcal/S)$. We identify the fundamental groups $\pi_{1}(\Xcal_{s},\sigma(s))$, $s\in S^{\circ}$, to a single $\Gamma$. Then to $\tilde{\nu}$ there is associate a smooth family of complex valued characters of $\Gamma$
\begin{displaymath}
	\chi_{s}(\gamma)=\exp\left(\int_\gamma\tilde{\nu}_s\right),\quad\gamma\in\Gamma, \quad s\in S^{\circ}.
\end{displaymath}
Observe that the definition of $\chi_{s}$ only depends on $\nu$, and not the particular choice of the lift $\tilde{\nu}$. We can thus write $\nu$ in the integral. With this understood, local smooth sections $\ell$ of $\Lcal\to \Xcal^{\circ}$ are identified with functions $\tilde \ell$ on $\widetilde \Xcal$ satisfying the equivariance law
\begin{equation} \label{eqn:equivariant1}
\tilde \ell(\gamma \tilde{z}, s)=\chi_{s}(\gamma)\tilde \ell(\tilde{z},s),\quad\gamma\in\Gamma,\quad s\in S^{\circ}.
\end{equation}
Rational sections are meromorphic in $\tilde{z}\in\widetilde{\Xcal}_{s}$, for fixed $s$. Notice that for every $s\in S^{\circ}$, we have a holomorphic structure $\overline{\partial}_{s}$ on $\widetilde{\Xcal}_s$.

\begin{remark} \label{rem:lift}
\hspace{2em}
\begin{enumerate}
\item We clarify this important construction. Choose a lift $\tilde\sigma(s)$ to $\widetilde \Xcal_s$ lying in $\Fcal_s$. The rigidification $\sigma^\ast\Lcal\simeq \Ocal_S$ gives a nonzero element ${\bf e}\in \widetilde \Lcal\bigr|_{\tilde\sigma(s)}$. Using the relative flat connection $\nabla_{\Xcal_s}$, ${\bf e}$ extends to a global frame $\widetilde{\bf e}$ of $ \widetilde \Lcal_{s}\to \widetilde\Xcal_s$. Then the pullback of a section $\ell$ can be written $\tilde \ell(\tilde{z})\widetilde{\bf e}$.
\item If two lifts of $\sigma(s)$ are related by $\tilde\sigma_2(s)=\gamma\cdot \tilde\sigma_1(s)$, then $\widetilde{\bf e}_2=\chi(\gamma)\widetilde{\bf e}_1$, and therefore $\tilde\ell_2=\chi(\gamma)^{-1}\tilde\ell_1$.
\item
With the identification above, the relative connection $\nabla_{\Xcal/S}$ is given by
$$
\ell \longleftrightarrow \tilde \ell\quad :\qquad \frac{\nabla_{\Xcal/S}\ell}{\ell} \longleftrightarrow \frac{d\tilde \ell}{\tilde \ell}\quad\text{projected to }\Acal^{1,0}_{\widetilde{\Xcal}/S^{\circ}}.
$$
\end{enumerate}
\end{remark}
To extend the relative connection we must differentiate $\tilde \ell$ with respect to $s$, as well as   the factor
$$
s\mapsto \exp\left(\int_\gamma\nu_s\right)
$$
 all in a way which preserves the condition \eqref{eqn:equivariant1}.  Note that the dependence on $\gamma$ factors through homology.  Hence, we regard  $\gamma$ as a horizontal section of $(R^1\pi_\ast\underline \ZBbb)^{\vee}$ on $S^{\circ}$. Then, by the very definition of the Gauss-Manin connection,
 \begin{equation}\label{eqn:derivative-GM}
 d\exp\left(\int_\gamma \nu\right)
=\exp\left(\int_\gamma \nu\right) d\int_\gamma\tilde{\nu}
=\exp\left(\int_\gamma \nu\right) \int_\gamma\nabla_{\GM}\nu.
 \end{equation}
Here, the ``integral'' over $\gamma\subset \Xcal_s$ means the integral of the part of $\nabla_{\GM}\nu$ in $H^1_{dR}(\Xcal/S)$. Now choose any frame $\{[\eta_i]\}_{i=1}^{2g}$ of the local system $H^1_{dR}(\Xcal/S)$ on $S^{\circ}$.  Then we may write: 
$$\nabla_{\GM}\nu=\sum_{i=1}^{2g}[\eta_i]\otimes \theta_i\ ,\ \theta_i\in \Acal^1_{S^{\circ}}.$$
For each $s\in S^{\circ}$ and $i=1,\ldots, 2g$, there is a unique harmonic representative $\eta_i(z,s)$ of $[\eta_i](s)$ on the fiber $\Xcal_s$. For $\tilde{z}\in\widetilde{\Xcal}_s$, we consider a path joining $\widetilde{\sigma}(s)$ to $\tilde{z}$ in $\widetilde{\Xcal}_s$.  Then we set
\begin{equation} \label{eqn:connection-1form}
\int_{\tilde{\sigma}(s)}^{\tilde{z}}{\nabla_{\GM}\nu} := \sum_{i=1}^{2g}\left\{ \int_{\tilde{\sigma}(s)}^{\tilde{z}} \eta_i(z,s)
\right\} \theta_i,
\end{equation}
where we have abused notation and wrote $\eta_{i}$ for its lift to the universal cover. This expression varies smoothly in $s\in S^{\circ}$ and $\tilde{z}$. It is independent of the choice of local frame. Indeed, if $[\tilde \eta_i]=A_{ij}[\eta_j]$ for a (constant) matrix $A_{ij}$, then $\tilde \eta_i(z,s)=A_{ij}\eta_j(z,s)$ (uniqueness of harmonic representatives), and $A_{ij}\tilde\theta_i=\theta_j$ (since $\nabla_{\GM}\nu$ is intrinsically defined). It follows that
\begin{align*}
\sum_{i=1}^{2g}\left\{ \int_{\tilde{\sigma}(s)}^{\tilde{z}}\tilde \eta_i(z,s)
\right\}\tilde \theta_i &=
\sum_{i,j=1}^{2g}\left\{ \int_{\tilde\sigma(s)}^{\tilde z} A_{ij} \eta_j(z,s)
\right\}\tilde \theta_i \\
&=
\sum_{i,j=1}^{2g}\left\{ \int_{\tilde{\sigma}(s)}^{\tilde z}  \eta_j(z,s)
\right\}A_{ij}\tilde \theta_i \\
&= \sum_{j=1}^{2g}\left\{ \int_{\tilde{\sigma}(s)}^{\tilde z} \eta_j(z,s)
\right\} \theta_j.
\end{align*}
With this understood, on $\widetilde{\Xcal}$
 we extend the relative connection $\nabla_{\Xcal/S}$ by the following (see Remark \ref{rem:lift}): if $\tilde{z}\in\widetilde{\Xcal}_{s}$,
\begin{equation} \label{eqn:extension}
\frac{\nabla\ell}{\ell}(\tilde{z},s) :=\frac{d\tilde \ell}{\tilde \ell}(\tilde{z},s)- \int_{\tilde{\sigma}(s)}^{\tilde{z}}{\nabla_{\GM}\nu}.
\end{equation}
We claim that this expression descends to a $1$-form on $\Xcal^{\circ}$ and is independent of the choice of lift of $\sigma(s)$.
Both facts follow from the same argument.
Suppose, for example, that $\tilde\sigma_2(s)=\gamma\tilde\sigma_1(s)$ are two choices of local lifts.
 Then by Remark \ref{rem:lift} (ii),   It follows that $\tilde\ell_2(\tilde z)=\chi(\gamma)^{-1}\tilde\ell_1(\tilde z)$, and therefore
$$
\frac{d\tilde \ell_2}{\tilde \ell_2}=\frac{d\tilde \ell_1}{\tilde \ell_1}-d\log\chi(\gamma)
$$
On the other hand, 
\begin{align*}
\int_{\tilde{\sigma}_2}^{\tilde{z}}{\nabla_{\GM}\nu}&=\int_{\gamma\tilde{\sigma}_{1}}^{\tilde z}{\nabla_{\GM}\nu} \\
&=\int_{\tilde{\sigma}_{1}}^{\tilde z}{\nabla_{\GM}\nu} -
\int_{\tilde{\sigma}_{1}}^{\gamma\tilde\sigma_1}{\nabla_{\GM}\nu} 
\\
&=\int_{\tilde{\sigma}_{1}}^{\tilde z}{\nabla_{\GM}\nu} -
\int_\gamma{\nabla_{\GM}\nu} 
\\
&=\int_{\tilde{\sigma}_{1}}^{\tilde z}{\nabla_{\GM}\nu} -
d\log\chi(\gamma)
\end{align*}
where we have used \eqref{eqn:derivative-GM}.
It follows that $\nabla\ell/\ell$ in \eqref{eqn:extension} is independent of the lift.  The fact that  $\nabla\ell/\ell$ descends to a $1$-form on  $\Xcal^\circ$ follows by the same argument. This proves the claim.

From the previous discussion
it also follows  that given overlapping contractible open subsets of $S$, say $S^{\circ}_1$ and $S^{\circ}_2$, the corresponding extensions $\nabla_1$ and $\nabla_2$ agree over the intersection, so that they can be glued together. Therefore, there exists a smooth connection $\nabla:L\rightarrow L\otimes\Acal^{1}_{\Xcal}$ that, locally on contractible open subsets of $S$, is of the form \eqref{eqn:extension}.
\begin{definition}
The extended connection $\nabla$ on $L\to \Xcal$ given by the  procedure above will be called \emph{the canonical extension} of $\nabla_{\Xcal/S}$ to $\Xcal$.
\end{definition}
\begin{remark}
\hspace{2em}
\begin{enumerate}
\item It is immediate that $\nabla$ indeed satisfies the Leibniz rule and is  a smooth connection. It is also trivially rigidified along $\sigma$. 
\item It is, however, perhaps not so clear the $\nabla$ is compatible with the holomorphic structure on $\Lcal$, and this will be checked below in Theorem \ref{thm:extension}. 
\item From now on, for notational convenience we confuse points on fibers of $\Xcal^{\circ}\rightarrow S^{\circ}$ with their lifts to fundamental domains of universal covers. Therefore, we will write expressions such as
\begin{displaymath}
	\int_{\sigma(s)}^{z}\nabla_{\GM}\nu\ .
\end{displaymath}
\end{enumerate}
\end{remark}

\begin{lemma}\label{lemma:nabla-pullback}
The construction of the canonical extension $\nabla$ is compatible with base change.
\end{lemma}
\begin{proof}
The lemma follows from the expression \eqref{eqn:extension} and the compatibility of the Gauss-Manin invariant with base change \eqref{eqn:GM-pullback}.
\end{proof}

Because the line bundle $\Lcal$ is of relative degree 0, we can endow it with the rigidified Chern connection $\nabla_{ch}$, which is flat on fibers. We next show that the Chern connection is the canonical extension of its vertical projection, as given by the preceding construction.
\begin{lemma} \label{lem:chern}
Suppose that the flat relative  connection $\nabla_{\Xcal/S}$ is induced by a hermitian structure $h$ on $\Lcal$ whose  Chern connection is flat on the fibers. Then the 
 Chern connection $\nabla_{ch}$ of $(\Lcal, h)$ coincides with the canonical extension \eqref{eqn:extension} of $\nabla_{\Xcal/S}$.
\end{lemma}

\begin{proof}
We work locally on contractible subsets $S^{\circ}$ of $S$. Let $p: \widetilde \Xcal\to \Xcal^{\circ}$ denote the fiberwise universal cover of $\Xcal^{\circ}$, and
choose a lift $\tilde \sigma$ of the section $\sigma$. 
Choose a smooth trivialization $\bf 1$ of the underlying $\Ccal^\infty$ line bundle $L\to
\Xcal^{\circ}$, compatible with the trivialization along $\sigma$, 
and write the $\dbar$-operator associated to $\Lcal$ as
$\dbar_L=\dbar+\alpha$, for a $(0,1)$-form $\alpha$ on $\Xcal$.  By changing
the trivialization we may assume the restriction of $\alpha$ to each fiber
is harmonic.  Let $\nabla_0=d+A$, $A=\alpha-\bar\alpha$.  Then $\nabla_0$ is
the Chern connection for the hermitian metric $\Vert{\bf 1}\Vert=1$.
Moreover, by the assumption on $\alpha$, the restriction of $A$ to every fiber is
closed, and so $\nabla_0$ is relatively flat and is therefore the Chern
connection of $\Lcal$.
Let $\nabla_{\GM}\nu_0$ be the associated Gauss-Manin invariant.
Set
$$
{\bf e}_0=\exp\left(-\int_{\sigma(s)}^z A\right)\cdot {\bf 1}
$$
Then ${\bf e}_0$ is a vertically flat section, well-defined on
$\widetilde\Xcal$.  Moreover, 
$$
\nabla_0{\bf e}_0=\left(-A-\int_{\sigma(s)}^z dA\right){\bf e}_0+ A{\bf
e}_0= \left(-\int_{\sigma(s)}^z dA\right){\bf e}_0
$$
Now by definition of the integral in \eqref{eqn:connection-1form},
$$
\int_{\sigma(s)}^z dA=\int_{\sigma(s)}^z\nabla_{\GM}\nu_0 
$$
and we conclude that
$$
\frac{\nabla_0{\bf e}_0}{{\bf e}_0}=-\int_{\sigma(s)}^z\nabla_{\GM}\nu_0 
$$
The equivariant function $\tilde\ell_0$ on $\widetilde\Xcal$ associated with 
the holomorphic section $\ell$ is given by: $\ell =\tilde \ell_0 {\bf e}_0$, and so
\begin{equation} \label{eqn:l0}
\frac{\nabla_0\ell}{\ell}=
\frac{d\tilde \ell_0}{\tilde\ell_0}+
\frac{\nabla_0{\bf e}_0}{{\bf e}_0}=
\frac{d\tilde \ell_0}{\tilde\ell_0}
-\int_{\sigma(s)}^z\nabla_{\GM}\nu_0 
\end{equation}
This completes the proof.
\end{proof}

\begin{theorem}[{\sc Canonical extension}] \label{thm:extension}
The canonical extension $\nabla$ is  compatible with the holomorphic structure on $\Lcal$, and it  satisfies (WR) universally. Moreover, it is rigidified along the section $\sigma$.
\end{theorem}

\begin{proof}
Again we work locally on contractible open subsets $S^{\circ}$ of $S$. For the first statement, it suffices to show that for any meromorphic section $\ell$ of $\Lcal$, 
$$\left(\frac{\nabla \ell}{\ell}\right)^{0,1}=0$$
The restriction of this form to the fibers of $\Xcal^{\circ}$ vanishes; hence, with respect to local holomorphic
coordinates $\{s_i\}$ on $S^{\circ}$ we may write
\begin{equation} \label{eqn:local}
\left(\frac{\nabla \ell}{\ell}\right)^{0,1}= \sum_i \varphi_i d\bar s_i
\end{equation}
for functions $\varphi_i$ on $\Xcal^{\circ}$.  We wish to prove that the $\varphi_i$ vanish identically.
Write $\nabla=\nabla_0+\theta$, where $\nabla_0$ is the Chern connection as in the proof of Lemma \ref{lem:chern}. 
By the construction \eqref{eqn:extension} of the canonical extensions, 
\begin{equation} \label{eqn:theta01}
\theta^{0,1}=\frac{\dbar\tilde\ell}{\tilde\ell}-
\frac{\dbar\tilde\ell_0}{\tilde\ell_0}+\int_{\sigma(s)}^z \left(\nabla_{\GM}\nu_0\right)^{0,1}
-\int_{\sigma(s)}^z \left(\nabla_{\GM}\nu\right)^{0,1}
\end{equation}
 Now from the definition of the Gauss-Manin integral \eqref{eqn:connection-1form}, for fixed $s$ the expression
$$
\int_{\sigma(s)}^z (\nabla_{\GM}\nu)^{0,1}
$$
is a harmonic function of $z$.   Similarly for $\nu_0$.  Taking $\partial\dbar$ of  \eqref{eqn:theta01}, it then follows that the $dz\wedge d\bar z$ term of $\partial\dbar \theta^{0,1}$ vanishes.  But by Lemma \ref{lem:chern}, 
$$
\theta^{0,1}=\left(\frac{\nabla \ell}{\ell}\right)^{0,1}-\left(\frac{\nabla_0 \ell}{\ell}\right)^{0,1}=\left(\frac{\nabla \ell}{\ell}\right)^{0,1}
$$
and so in \eqref{eqn:local},  $\partial_z\partial_{\bar z}\varphi_i=0$ for all $i$.  Hence, the $\varphi_i$ are harmonic, and therefore constant along the fibers of $\Xcal^{\circ}$.
  But
they also vanish along $\sigma$, and so vanish identically, and the first statement of the theorem
follows.

 It remains to prove that $\nabla$ satisfies (WR) universally. By the compatibility of the construction of $\nabla$ with base change (Lemma \ref{lemma:nabla-pullback}), it is enough to work over $S^{\circ}$. Also, in the proof we are allowed to do base changes of $S^{\circ}$ induced by \'etale base changes of $S$, since equalities of differential forms are local for this topology. For the proof we follow the argument for classical Weil reciprocity for Riemann surfaces.  Recall that for a  holomorphic differential $\omega$ and nonzero meromorphic function $f$ on a Riemann surface $X$ with homology basis $\{\alpha_i,\beta_i\}$, we have (cf.\ \cite[Reciprocity Law I, p.\ 230]{GriffithsHarris:78})
\begin{equation} \label{eqn:reciprocityI}
2\pi i\sum_{p\in \Div(f)}\ord_p(f)\int_{\sigma}^p \omega =
\sum_{i=1}^g\left(\int_{\alpha_i}\omega\int_{\beta_i}\frac{df}{f} -\int_{\alpha_i}\frac{df}{f}\int_{\beta_i}\omega\right)
\end{equation}
where the left hand side is independent of the base point $\sigma$ 
because $\deg\Div(f)=0$. The divisor of $f$ is understood to be restricted to the fundamental domain delimited by the curves representing the homology basis. We note two generalizations of this type of formula:
\begin{enumerate}
\item in \eqref{eqn:reciprocityI}, we may use an anti-holomorphic form $\overline \omega$ instead of $\omega$. Indeed, the periods of $df/f$ are pure imaginary, and the assertion follows by conjugating both sides;
\item in families, if $\Div f/S$ is finite \'etale over $S$, then after \'etale base change we can assume the irreducible components are given by sections.  Applying this, the previous comment, and the Hodge decomposition to the cohomological part of $\nabla_{\GM}\nu$, we have for each $s$
\begin{align}
\begin{split} \label{eqn:WR-families}
2\pi i\sum_{p(s)\in \Div(f(\cdot, s))}\ord_{p(s)}(f)&\int_{\sigma(s)}^{p(s)} {\nabla_{\GM}\nu} =\\
&\sum_{i=1}^g\left(\int_{\alpha_i}{\nabla_{\GM}\nu} \int_{\beta_i}\frac{df}{f} -\int_{\alpha_i}\frac{df}{f}\int_{\beta_i}{\nabla_{\GM}\nu} \right).
\end{split}
\end{align}
Here $\lbrace\alpha_i,\beta_i\rbrace$ is a horizontal symplectic basis of $(R^{1}\pi_{\ast}\mathbb{Z})^{\vee}$ on $S^{\circ}$ as fixed in the beginning of Section  \ref{sec:canonical}.
\end{enumerate}
There is a second version of Weil reciprocity (cf.\ \cite[p.\ 243]{GriffithsHarris:78}) for pairs $f,g$ of meromorphic functions
\begin{align}
\begin{split} \label{eqn:reciprocityII}
2\pi i\sum_{q\in \Div(g)}\ord_q(g)\log f(q)&-
2\pi i\sum_{p\in \Div(f)}\ord_p(f)\log g(p) =\\
&\sum_{i=1}^g\left( \int_{\alpha_i}\frac{df}{f}\int_{\beta_i}\frac{dg}{g}-\int_{\alpha_i}\frac{dg}{g}\int_{\beta_i}\frac{df}{f}\right)
\end{split}
\end{align}
which applies also to families (after possible \'etale base change, to assume the components of the divisors are given by sections). Again, the divisors are taken in the fundamental domain delimited by the homology basis.

We now apply \eqref{eqn:reciprocityII} in the case where $\Div f/S$ is finite \'etale and $g=\tilde \ell$ (regarded as an equivariant meromorphic function fiberwise).  We observe that the periods of $df/f$ are constant functions on the base $S^{\circ}$ (they belong to $2\pi i\mathbb{Z}$). Taking derivatives and appealing to \eqref{eqn:derivative-GM} and \eqref{eqn:WR-families}, results in the string of equalities 
\begin{align*}
2\pi i \tr_{\Div \ell/S^{\circ}}\left(\frac{df}{f}\right)-2\pi i \tr_{\Div f/S^{\circ}}\left(\frac{d\tilde \ell}{\tilde \ell}\right)&=
d\left[\sum_{i=1}^g\left( \int_{\alpha_i}\frac{df}{f}\int_{\beta_i}\frac{d\tilde \ell}{\tilde \ell}-\int_{\alpha_i}\frac{d\tilde \ell}{\tilde \ell}\int_{\beta_i}\frac{df}{f}\right)\right]\\
&=\sum_{i=1}^g\left( \int_{\alpha_i}\frac{df}{f}\left\lbrace d\int_{\beta_i}\tilde{\nu}_s\right\rbrace-\left\lbrace d\int_{\alpha_i}\tilde{\nu}_s\right\rbrace\int_{\beta_i}\frac{df}{f}\right)\\
&=\sum_{i=1}^g\left( \int_{\alpha_i}\frac{df}{f}\int_{\beta_i}\nabla_{\GM}\nu-\int_{\alpha_i}\nabla_{\GM}\nu\int_{\beta_i}\frac{df}{f}\right)\\
&=2\pi i \tr_{\Div f/S^{\circ}}\left(\int_{\sigma}^z \nabla_{\GM}\nu\right).
\end{align*}
Therefore by \eqref{eqn:extension},
$$
\tr_{\Div \ell/S^{\circ}}\left(\frac{df}{f}\right)=\tr_{\Div f/S^{\circ}}\left(\frac{\nabla \ell}{ \ell}\right)
$$
In other words, $I(\nabla)=0$.

The rigidification property is immediate from the construction \eqref{eqn:extension}. This completes the proof of Theorem \ref{thm:extension}.
\end{proof}

\subsection{The canonical extension: uniqueness} \label{sec:unique}
In this section  we prove the uniqueness of the extension obtained in the previous section.  In fact, we will prove a little more.  
Let $\theta\in \Acal^i_{\Xcal}$ satisfying the following properties:

\begin{itemize}
	\item[(V1)] rigidification: $\sigma^{\ast}(\theta)=0$;
	\item[(V2)] the pull-back of $\theta$ to any fiber $\Xcal_s$, $s\in S$, vanishes;
	\item[(V3)] vanishing along rational divisors, universally: given a \emph{smooth} morphism of quasi-projective complex varieties $p:T\rightarrow S$, and a meromorphic function $f$ on the base change $\Xcal_T$ whose divisor is finite \'etale over $T$, we have
	\begin{displaymath}
		\tr_{\Div  f/T}(p^{\ast}\theta)=0.
	\end{displaymath}
	Here we write $p^{\ast}\theta$ for the pull-back of $\theta$ to $\Xcal_T$ by the induced morphism $\Xcal_T\rightarrow \Xcal$. This is the Weil vanishing property that appeared in the proof of Theorem \ref{theorem:int-connection}. 
\end{itemize}


\begin{proposition}[{\sc Vanishing lemma}]\label{prop:vanishing}
Let $\theta$ be a smooth complex differential 1-form on $\Xcal$, satisfying properties (V1)--(V3) above. Then $\theta$  vanishes identically on $\Xcal$.
\end{proposition}
\begin{proof}
The vanishing of a differential form is a local property, so we may assume that $\Omega_{S}^{1}$ is a free sheaf on $S$. Let $\theta_1,\ldots,\theta_d$ be a holomorphic frame for $\Omega_{S}^{1}$. Then, $\theta_1,\ldots,\theta_n,\overline{\theta}_1,\ldots,\overline{\theta}_n$ is a frame for  $\Acal^1_S$. Because $\theta$ vanishes on fibers by (V2), on $\Xcal$ we can write
\begin{displaymath}
	\theta=\sum_{i}f_{i}\pi^{\ast}\theta_{i}+\sum_{i}g_{i}\pi^{\ast}\overline{\theta}_i,
\end{displaymath}
for some smooth functions $f_i, g_i$ on $\Xcal$. Observe that $f_i$, $g_i$ vanish along the section $\sigma$ (V1) (by the independence of $\theta_1,\ldots,\theta_n,\overline{\theta}_1,\ldots,\overline{\theta}_n$). They also satisfy the Weil vanishing property (V3). For this, we need to observe that for a smooth morphism $p:T\rightarrow S$, the differential forms $p^{\ast}\theta_{i}$, $p^{\ast}\overline{\theta}_{i}$ are still stalk-wise independent, so are their pull-backs to $\Xcal_T$ (because $\Xcal_T\rightarrow T$ is smooth).\footnote{Note, however, that this property is lost in general if $p:T\rightarrow S$ is not smooth. This explains the restriction to smooth base change  in (V3).} We want to show these functions identically vanish.  We are thus required to prove that a smooth complex function $\varphi:\Xcal\rightarrow\CBbb$ satisfying (V1) and (V3) automatically satisfies (V2), and therefore vanishes.

Let us observe that  Weil vanishing for functions implies something more. Let $D$ be a divisor in $\Xcal_T$, finite and flat over $T$. Then the trace: $
	\tr_{D/T}(\varphi)
$,
can still be defined as a continuous function on $T$, by averaging on fibers and taking multiplicities into account (for this one does not even need $T\rightarrow S$ to be smooth). Hence, if $\tr_{D/T}(\varphi)$ vanishes over a Zariski dense open subset of $T$, then it vanishes everywhere by continuity. This is the case for $D=\Div  f$, where $f$ is a rational function on $\Xcal_T$ with finite flat divisor over $T$. Indeed, there is a dense (Zariski) open subset $U\subseteq T$ such that $D$ is finite \'etale over $U$. This means that the Weil vanishing property holds for rational divisors whose components are only finite and flat.

Recall that
the relative jacobian $J:=J(\Xcal/S)\rightarrow S$
 is a fibration of abelian varieties over $S$, representing the functor $T\mapsto J(T)$ of line bundles on $\Xcal_T$ of relative degree 0, modulo line bundles coming from the base. Here, we will exploit the fact that the total space $J$ is smooth (because $S$ is smooth), and therefore can be covered by Zariski open subsets $U$, which are smooth and quasi-projective over $S$! The natural inclusion of a Zariski open subset $U\hookrightarrow J$ corresponds to the universal rigidified (along $\sigma$) Poincar\'e bundle restricted to $\Xcal_U$, and for small enough $U$, one can suppose this line bundle is associated to a divisor in $\Xcal_U$, finite flat over $U$. We will call this ``a universal'' finite flat divisor over $U$. It is well defined only up to rational equivalence (through rational divisors which are finite flat over the base).

We proceed to extend $\varphi:\Xcal\rightarrow\CBbb$ to a continuous function $\widetilde{\varphi}:J\rightarrow\CBbb$, whose restriction on fibers is a continuous morphism of (topological) groups (for the analytic topology). 
Let $U$ be a Zariski open subset of $J$, such that $\Xcal_U$ affords a ``universal'' finite flat divisor of degree 0 over $U$. Denote this divisor $D_U$. Then, $\tr_{D_U/U}(\varphi)$ is a continuous function on $U$. Moreover, it only depends on the rational equivalence class of $D_U$, by the Weil vanishing property (extended to finite flat rational divisors). Because of this, given $U$ and $V$ intersecting open subsets in $J(\Xcal/S)$, we also have
\begin{displaymath}
	\tr_{D_U/U}(\varphi)\mid_{U\cap V}=\tr_{D_V/V}(\varphi)\mid_{U\cap V}.
\end{displaymath}
Therefore these functions glue into a continuous function $\widetilde{\varphi}:J\rightarrow \CBbb$. The linearity of the trace function with respect to sums of divisors, guarantees that $\widetilde{\varphi}$ is compatible with the group scheme structure. Namely, given the addition law
\begin{displaymath}
	\mu:J\times_{S} J\rightarrow J,
\end{displaymath}
and the two projections $p_i:J\times_{S} J\rightarrow J$, the following relation holds:
\begin{displaymath}
	\mu^{\ast}\widetilde{\varphi}=p_{1}^{\ast}\widetilde{\varphi}+p_{2}^{\ast}\widetilde{\varphi}.
\end{displaymath}
This in particular implies that $\widetilde{\varphi}$ is a topological group morphism on fibers and immediately leads to the vanishing of $\widetilde{\varphi}$ on fibers; hence, everywhere. Indeed, a given fiber $J_{s}$ ($s\in S$) can be uniformized as $\CBbb^{g}/\Lambda$, for some lattice $\Lambda$. The corresponding arrow $\CBbb^{g}\rightarrow\CBbb$ induced from $\widetilde{\varphi}$ is a continuous morphism of topological abelian groups, and it is therefore  a linear map of real vector spaces! Because the map factors through $J_{s}$, which is compact, its image is compact, and hence is reduced to $\{0\}$.

Finally, let $\iota: \Xcal\hookrightarrow J$ be the closed immersion given by the section $\sigma$. Because of the rigidification of $\varphi$, we have $\varphi=\iota^{\ast}\widetilde{\varphi}=0$. This concludes the proof of the proposition.
\end{proof}

\begin{corollary}[{\sc Uniqueness}]
Suppose that we are given $\nabla_1,\nabla_2$ are smooth connections on $L\to \Xcal$ (hence non-necessarily compatible with the holomorphic structure) satisfying the following properties:
\begin{itemize}
	\item[(E1)] they are both rigidified along the section $\sigma$;
	\item[(E2)] they coincide on fibers $\Xcal_s$, $s\in S$;
	\item[(E3)] they satisfy the Weil reciprocity for connections, universally.
\end{itemize}
Then $\nabla_1=\nabla_2$. Therefore, the canonical extension is unique.
\end{corollary}
\begin{proof}
Indeed, we can write $\nabla_1=\nabla_2+\theta$, where $\theta$ is a smooth 1-form. Then properties (E1)--(E3) ensure that $\theta$ satisfies (V1)--(V3). By the vanishing lemma, $\theta=0$, so $\nabla_1=\nabla_2$ as required. The consequences for the canonical extension follow, since they satisfy (E1)--(E3).
\end{proof}
\begin{remark}
After the corollary, it is justified to talk about \emph{the} canonical extension.
\end{remark}

A second application of the vanishing lemma is an alternative proof of  the compatibility of a connection $\nabla$ on $\Lcal$ with the holomorphic structure (see Theorem \ref{thm:extension}). 
\begin{corollary}[{\sc Compatibility}]
Let $\Lcal\to \Xcal$ be a holomorphic line bundle 
with  connection $\nabla$ which:
\begin{itemize}
	\item[(H1)] is rigidified along the section $\sigma$;
	\item[(H2)] is holomorphic on fibers;
	\item[(H3)] and satisfies the Weil reciprocity for connections, universally.
\end{itemize}
Then $\nabla$ is compatible with the holomorphic structure on $\Lcal$. In particular, the canonical extension is compatible with the holomorphic structure of $\Lcal$.
\end{corollary}
\begin{proof}
Because $\nabla$ is holomorphic on fibers, $\Lcal$ is of relative degree 0 and can be endowed with a Chern connection $\nabla_{ch}$. We can suppose $\nabla_{ch}$ is rigidified along the section $\sigma$ (H1) (because $\Lcal$ is rigidified). We already know that $\nabla_{ch}$ satisfies (H2)--(H3) (see Example \ref{ex:chern-reciprocity}). Also,  $\nabla_{ch}$ is compatible with the holomorphic structure $\Lcal$, by definition of Chern connections. Hence it is enough to compare $\nabla$ and $\nabla_{ch}$. Let us write $\nabla=\nabla_{ch}+\theta$. We decompose $\theta$ into types $(1,0)$ and $(0,1)$: $\theta=\theta^{\prime}+\theta^{\prime\prime}$, and  we wish to see that $\theta^{\prime\prime}=0$. But now, observe the following facts:
\begin{itemize}
	\item $\sigma^{\ast}\theta=0$ and pull-back by $\sigma$ respects types, so that $\sigma^{\ast}\theta^{\prime}=-\sigma^{\ast}\theta^{\prime\prime}$ has to vanish;
	\item $\theta^{\prime\prime}$ vanishes along the fibers, because $\nabla$ and $\nabla_{ch}$ are holomorphic along the fibers;
	\item $\theta$ satisfies the Weil vanishing universally. Because the trace along divisors $\tr_{D/T}$ respects types of differential forms, we deduce that it vanishes for $\theta^{\prime\prime}$.
\end{itemize}
Hence,  $\theta^{\prime\prime}$ satisfies the properties (V1)--(V3) above, and it therefore vanishes. It follows that $\nabla=\nabla_{ch}+\theta^{\prime}$ is compatible with  $\Lcal$.
\end{proof}
\subsection{Variant in the absence of rigidification}
In case the morphism $\pi:\Xcal\rightarrow S$ does not come with a rigidification, we can still pose the problem of extending connections and impose (WR) universally. We briefly discuss this situation. Locally for the \'etale topology, the morphism $\pi$ admits sections. \'Etale morphisms are local isomorphisms in the analytic topology. Therefore, given a relative connection $\nabla_{\Xcal/S}$, there is an analytic open covering $U_i$ of $S$, and connections $\nabla_{i}$ on $\Lcal_{\Xcal_{U_i}}$ extending $\nabla_{\Xcal/S}$ and satisfying (WR) universally. On an overlap $U_{ij}:=U_i\cap U_j$, the connections $\nabla_i$ and $\nabla_j$ differ by a smooth $(1,0)$-form $\theta_{ij}$ on $\Xcal_{U_{ij}}$. The differential form $\theta_{ij}$ satisfies the vanishing properties (V2)--(V3) of Section  \ref{sec:unique}. By the vanishing lemma (Proposition \ref{prop:vanishing}), $\theta_{ij}$ comes from a differential form on $U_{ij}$: $\theta_{ij}\in\Gamma(U_{ij},\Acal^{1,0}_{S})$. This family of differential forms obviously verifies the 1-cocycle condition, and hence gives a cohomology class in $H^{1}(S,\Acal^{1,0}_{S})$. But $\Acal^{1,0}_{S}$ is a fine sheaf, because it is a $\mathcal{C}^{\infty}(S)$-module. Therefore, this cohomology group vanishes and the cocycle $\lbrace\theta_{ij}\rbrace$ is trivial. This means that, after possibly modifying the connections $\nabla_{i}$ by suitable $(1,0)$ differential forms coming from the base, we can glue them together into a connection $\nabla$ on $\Lcal$, extending $\nabla_{\Xcal/S}$. Because any differential form coming from the base $S$ has vanishing trace along divisors of rational functions (more generally, along relative degree 0 divisors), this connection $\nabla$ still satisfies (WR) universally. Two such connections differ by a differential form in $\Gamma(S,\Acal^{1,0}_{S})$. Again, differential forms coming from $S$ have zero trace along degree zero divisors, and so this implies the induced trace connections on Deligne pairings $\langle\Lcal,\Mcal\rangle$ don't depend on the particular extension, as long as $\Mcal$ has relative degree 0. We summarize the discussion in a statement.
\begin{proposition}\label{prop:non-rig-ext}
In the absence of a section of $\pi:\Xcal\rightarrow S$, the space of extensions of $\nabla_{\Xcal/S}$ satisfying (WR) universally is a nonempty torsor under $\Gamma(S,\Acal_{S}^{1,0})$. To $\nabla_{\Xcal/S}$ there is an intrinsically attached trace connection on relative degree zero line bundles $\Mcal$ (still denoted $\nabla_{\langle\Lcal,\Mcal\rangle^{tr}}$).
\end{proposition}
\begin{remark}
The proposition refines Theorem \ref{theorem:existence-unicity}, when we are interested in connections on $\Lcal$ satisfying (WR) universally and extending a given flat relative connection $\nabla_{\Xcal/S}$.
\end{remark}

\begin{proof}[Completion of the Proof of Theorem \ref{thm:intersection}]
Let now $\nabla_{\Xcal/S}$ be a flat relative connection, and let $\nabla$ be any extension satisfying (WR) universally. Attached to $(\Lcal,\nabla)$ there is a trace connection $\Mcal\mapsto\nabla_{\Mcal}$, on line bundles of relative degree 0. This trace connection does not depend on the choice of extension $\nabla_{\Xcal/S}$, as we saw above. By a similar argument, if $(\Lcal,\nabla_{\Xcal/S}^{L})$ and $(\Mcal,\nabla_{\Xcal/S}^{M})$ line bundles with relatively flat connections on $\Xcal\rightarrow S$, Propositions \ref{prop:non-rig-ext}, \ref{prop:sym-int-conn} and \ref{prop:indep-int-conn} together show that there is an intrinsically attached intersection connection $\nabla_{\langle\Lcal,\Mcal\rangle}^{int}$ on $\langle\Lcal,\Mcal\rangle$.
\end{proof}

\begin{remark}
We can restate the intrinsic construction of trace and intersection connections above in a compact categorical language. Let us for instance treat the case of intersection connections. 
Let $\PIC^{\sharp}(\Xcal/S)$ be the category fibered in groupoids (over smooth schemes over $\mathbb{C}$ factoring through $S$), whose objects are line bundles $\Lcal$ over $\Xcal$ (or base changes of it) together with flat relative connections $\nabla_{\Xcal/S}$. Morphisms are given by flat isomorphisms. Similarly, $\PIC^{\sharp}(S)$ is the Picard category of line bundles with compatible connections. The intersection connection construction is a morphism of 
categories fibered in groupoids
\begin{displaymath}
	\begin{split}
		\PIC^{\sharp}(\Xcal/S)\times\PIC^{\sharp}(\Xcal/S)&\longrightarrow\PIC^{\sharp}(S)\\
		((\Lcal,\nabla_{\Xcal/S}^{L}),(\Mcal,\nabla_{\Xcal/S}^{M}))&\longmapsto (\langle\Lcal,\Mcal\rangle,\nabla_{\langle\Lcal,\Mcal\rangle}^{int}).
	\end{split}
\end{displaymath}
This picture encodes both the independence of the auxiliary extensions of the connections, as well as the various functorialities of the intersection connections.
\end{remark}

\subsection{Relation between trace and intersection connections} \label{sec:relation}
We now fill in the proof of Theorem \ref{thm:intersection} (iii), which asserts that if $\Mcal\to\Xcal$ has relative degree $0$, then the trace connection on $\langle\Lcal, \Mcal\rangle$ is a special case of an intersection connection. We state the precise result in  the following
\begin{proposition}
Let $\Lcal\to \Xcal$ be equipped with a flat relative connection $\nabla_{\Xcal/S}^L$.  Let $\Mcal\to \Xcal$ be a hermitian, holomorphic line bundle with Chern connection $\nabla^M_{ch}$ whose restriction $\nabla^M_{\Xcal/S}$ to  the fibers of $\pi:\Xcal\to S$ is a flat relative connection. Let $\nabla^{tr}_{ \langle\Lcal, \Mcal\rangle}$ be the trace connection associated to $\nabla_{\Xcal/S}^L$, and
$\nabla^{int}_{ \langle\Lcal, \Mcal\rangle}$  the intersection connection associated to $\nabla_{\Xcal/S}^L$ and $\nabla^M_{\Xcal/S}$. Then $\nabla^{tr}_{ \langle\Lcal, \Mcal\rangle}=\nabla^{int}_{ \langle\Lcal, \Mcal\rangle}$.
\end{proposition}

\begin{proof}
An equality of connections is local for the \'etale topology, and our constructions are compatible with base change. Therefore, we can assume there is a section $\sigma$ and that $\Lcal$ is rigidified along $\sigma$. Let $\nabla_L$ be the canonical extension of $\nabla_{\Xcal/S}^L$.  By the definition eq.\ \eqref{eqn:int-connection}, it suffices to show
that for all rational sections  $m$ of $\Mcal$,  
\begin{equation} \label{eqn:trace-chern-vanishing}
\pi_\ast\left( \frac{\nabla^M_{ch} (m)}{m}\wedge F_{\nabla_L}\right)=0
\end{equation}
Let $\Vert\cdot\Vert$ denote the metric on $\Mcal$,
 and write $\nabla_L=\nabla^L_{ch}+\theta$, where $\nabla^L_{ch}$ is the Chern connection, and $\theta$ is of type $(1,0)$.  
Then  using the fact that $F_{\nabla^L_{ch}}$ is $\partial$-closed,  we have
\begin{align*}
\pi_\ast\left( \frac{\nabla^M_{ch} (m)}{m}\wedge F_{\nabla_L}\right)&=\pi_\ast\left(\partial\log\Vert m\Vert^2\wedge F_{\nabla_L}\right)\notag \\
&=\partial \pi_\ast\left(\log\Vert m\Vert^2\cdot F_{\nabla^L_{ch}}\right)+\pi_\ast \left(\partial\log\Vert m\Vert^2\wedge d\theta\right)
\end{align*}
The first term vanishes, since $\nabla^L_{ch}$ is flat on the fibers.
For the second term,  as in the proof of Theorem \ref{thm:intersection} we find
\begin{align}
\pi_\ast \left(\partial\log\Vert m\Vert^2\wedge d\theta\right)&=\pi_\ast \left(\partial\log\Vert m\Vert^2\wedge \dbar\theta\right) \notag \\
&=-\dbar\pi_\ast \left(\partial\log\Vert m\Vert^2\wedge \theta\right)+\pi_\ast \left(\dbar\partial\log\Vert m\Vert^2\wedge\theta\right) \label{eqn:m} \\
&=\pi_\ast(F_{\nabla_{ch}^M}\wedge \theta)+2\pi i \tr_{\Div m/S}(\theta)\label{eqn:M-reciprocity}
\end{align}
where we have used that the first term on the right hand side of \eqref{eqn:m} vanishes because of type, and we apply the Poincar\'e-Lelong formula to the second term on the right hand side of \eqref{eqn:m} to derive \eqref{eqn:M-reciprocity}.  Locally on contractile open subsets of $S$, we can apply \eqref{eqn:reciprocityI} to obtain
\begin{equation}\label{eq:periods-theta} 
2\pi i\sum_{p\in \Div(m)}\ord_p(m)\int_{\sigma}^p \theta =
\sum_{i=1}^g\left(\int_{\alpha_i}\theta\int_{\beta_i}\frac{d\tilde m}{\tilde m} -\int_{\alpha_i}\frac{d\tilde m}{\tilde m}\int_{\beta_i}\theta\right)
\end{equation}
Let $\nabla_\GM \nu_M$ denote the Gauss-Manin invariant for $\nabla^M_{\Xcal/S}$. We now differentiate the equation above, and obtain:
\begin{equation}\label{eq:periods-theta-2}
	\begin{split}
	2\pi i\tr_{\Div m}(\theta)+2\pi i&\sum_{p\in \Div(m)}\ord_p(m)\int_{\sigma}^p \nabla_{\GM}\theta=\\
		&\sum_{i=1}^g\left(\int_{\alpha_i}\nabla_{\GM}\theta\int_{\beta_i}\frac{d\tilde m}{\tilde m} -\int_{\alpha_i}\frac{d\tilde m}{\tilde m}\int_{\beta_i}\nabla_{\GM}\theta\right)\\
	&+\sum_{i=1}^g\left(\int_{\alpha_i}\theta\int_{\beta_i}\nabla_\GM\nu_M -\int_{\alpha_i}\nabla_\GM\nu_M\int_{\beta_i}\theta\right).
	\end{split}
\end{equation}
A comment is in order to clarify the meaning of $\nabla_{\GM}\theta$ and its path integrations. By construction, the differential form $\theta$ is closed on fibers, and even holomorphic. Hence, it defines a relative cohomology class to which we can apply $\nabla_{\GM}$. This we write $\nabla_{\GM}\theta$. The meaning of integration along non-closed paths involves representing $\nabla_{\GM}\theta$ in terms of a family of harmonic forms, exactly as in Section \ref{sec:canonical}. After this clarification, we also note that equation \eqref{eq:periods-theta} holds for $\nabla_{\GM}\theta$ as well. We subtract this variant from \eqref{eq:periods-theta-2}, and deduce
\begin{align*}
2\pi i \tr_{\Div m/S}(\theta)&=\sum_{i=1}^g\left(\int_{\alpha_i}\theta\int_{\beta_i}\nabla_\GM\nu_M -\int_{\alpha_i}\nabla_\GM\nu_M\int_{\beta_i}\theta\right) \\
&=\pi_\ast\left(\theta\wedge \nabla_\GM\nu_M\right)\\
&=-\pi_\ast\left(\theta \wedge F_{\nabla^M_{ch}}\right)\ ,
\end{align*}
where to obtain the last line we use the fact that $\nabla^M_{ch}$ corresponds to the canonical extension of $\nabla^M_{\Xcal/S}$ (Lemma \ref{lem:chern}).
Hence, the right hand side of \eqref{eqn:M-reciprocity} vanishes, and so therefore does \eqref{eqn:trace-chern-vanishing}.  This completes the proof.
\end{proof}

\subsection{Curvatures}  \label{sec:curvature}
  In this section we compute the curvature of intersection and trace connections on $\langle\Lcal, \Mcal\rangle$. 
  Let  $\Lcal, \Mcal\to \Xcal$ be rigidified line bundles with flat relative connections $\nabla_{\Xcal/S}^L$ and $\nabla_{\Xcal/S}^M$, respectively (the rigidification is not essential here).  Denote the Gauss-Manin invariants by $\nabla_\GM\nu_L$ and $\nabla_\GM\nu_M$, and recall from Section \ref{sec:gm} that $\real \nu_L$, $\real\nu_M$ are well-defined.
  
 We will also need the following. Let
 \begin{equation}\label{eqn:ks}
 \KS(\Xcal/S)=\nabla_{GM}\Pi'=-\nabla_{\GM}\Pi'' \in \End\left(
 H_{dR}^1(\Xcal/S)\right)\otimes \Acal_S^{1,0}
 \end{equation}
 denote the derivative of the period map of the fibration $\Xcal\to S$, where $\Pi'$, $\Pi''$ are as in \eqref{eqn:nu'} and \eqref{eqn:nu''} (cf.\ \cite{Voisin:07}).
Finally,  define the operation
\begin{equation} \label{eqn:cup}
\bigl(H^1_{dR}(\Xcal/S)\otimes \Acal^i_S\bigr) \times \bigl(H^1_{dR}(\Xcal/S)\otimes \Acal^j_S\bigr)\stackrel{\cup}{\lra} H^2_{dR}(\Xcal/S)\otimes \Acal^{i+j}_S
\end{equation}
is given by the cup product on relative cohomology classes 
and the wedge product on forms, whereas $\pi_\ast$ denotes the fiber integration:
$
 H^2_{dR}(\Xcal/S)\to\Ccal^{\infty}(S)
$.  One easily verifies that
\begin{align}
\overline{\KS(\alpha)} & = -\KS(\overline\alpha) \label{eqn:ks-cc} \\
\pi_\ast\left( \alpha\cup \KS(\beta)\right) &=-\pi_\ast\left( \KS(\alpha)\cup \beta\right) \label{eqn:ks-skew} 
\end{align}
 
 With this understood, we have the following
 \begin{proposition} \label{prop:curvature-intersection}
  The curvature of the intersection connection $\nabla^{int}_{ \langle\Lcal, \Mcal\rangle}$ on $\langle\Lcal, \Mcal\rangle$ is given by \eqref{eqn:intersection-curvature}.
 \end{proposition}

 \begin{proof}
 By \eqref{eqn:extension}, the curvature of the canonical extension is given by
 $$
 -d\int_{\tilde\sigma(s)}^{\tilde z}\nabla_\GM\nu
 $$
 From Proposition \ref{prop:curv-inter}, the only term that survives the fiber integration is
 $$
 \nabla_\GM\nu_L\cup \nabla_\GM\nu_M.
 $$
 \end{proof}

 The following is then an immediate consequence of the curvature formula.
 \begin{corollary}  \label{cor:hol-intersection}
 If the flat relative connections on $\Lcal$ and $\Mcal$ are of type $(1,0)$ (see Definition \ref{def:type}), then the intersection connection on $\langle\Lcal, \Mcal\rangle$ is holomorphic.
 \end{corollary}
 
Next, we turn to trace connections, where the calculation is a bit more involved.  Of course, as shown in Section \ref{sec:relation}, the trace connection is a special case of an intersection connection. However, the following gives a more general formula where $\nu_M$ is not necessarily associated to a unitary connection.
\begin{proposition} \label{prop:curvature-trace}
Assume a rigidification of $\Lcal$.  Then the trace connection $\nabla^{tr}_{ \langle\Lcal, \Mcal\rangle}$
 on $\langle\Lcal,\Mcal\rangle$ has curvature:
 \begin{align}
 \begin{split} \label{eqn:curvature-trace}
F_{\nabla^{tr}_{ \langle\Lcal, \Mcal\rangle}}
&=\frac{1}{2\pi i}
\pi_\ast\biggl\{
(\nabla_{\GM}\nu_L)' \cup (\nabla_{\GM}\nu_M)'' 
-(\nabla_{\GM}\nu_L)'' \cup \overline{(\nabla_{\GM}\nu_M)'' } \\
&\qquad\qquad
+2\real \nu_M\cup\left( \KS(\Xcal/S)\wedge \nabla_{\GM}\nu_L \right)
\biggr\}.
\end{split}
\end{align}
\end{proposition}

\begin{remark} 
The trace connection on $\langle\Lcal, \Mcal\rangle$ is independent of a choice of relative connection on $\Mcal$. Using \eqref{eqn:ks-cc} and \eqref{eqn:ks-skew}, one verifies that \eqref{eqn:curvature-trace} is indeed independent of the choice of $\nabla_{\Xcal/S}^M$. Moreover,  by specializing to the Chern connection on $\Mcal$,  \eqref{eqn:curvature-trace} reduces to \eqref{eqn:intersection-curvature}, as it must by the comment above.
\end{remark}

We also point out the following 
\begin{corollary} \label{cor:flat}
 If $(\nabla_{\GM}\nu)''$ vanishes identically then the trace
connection on $\langle \Lcal, \Lcal\rangle$ is flat.
\end{corollary}

\begin{proof}
Differentiate the equation
$$
0=(\nabla_{\GM}\nu)''=\Pi''\nabla_{\GM}\nu
$$
to find 
$$
0=\nabla_\GM\Pi''\wedge \nabla_{\GM}\nu+ \Pi''\nabla_{\GM}^2\nu \\
=-\KS(\Xcal/S)\wedge \nabla_{\GM}\nu+ \Pi''\nabla_{\GM}^2\nu\ .
$$
Since the $\nabla_\GM^2=0$, the result follows.
\end{proof}

\begin{proof}[Proof of Proposition \ref{prop:curvature-trace}]
First, since the calculation is local in $S$ for the analytic topology, we can work over a contractible open subsets $S^{\circ}$. Let $\tilde \nu_L$ be a lift of $\nu_L$ (the classifying morphism of $\Lcal,\nabla^L_{\Xcal/S})$),  and $\chi_s:\pi_1(\Xcal_s, \sigma(s))\to \CBbb^\times$ denote the holonomy representation of the flat connection in the fiber over $s\in S^{\circ}$: $\chi_s(\gamma)=\exp(\int_\gamma\tilde\nu_s)$. Similarly for $\Mcal$.
Choose a homology basis as in Section \ref{sec:canonical}. Let $m$ be a meromorphic section of $\Lcal$, whose divisor is finite and \'etale over (an open subset of) $S$. After \'etale base change, we may assume that $\Div m$ is given by sections. Using \eqref{eqn:WR-families}, we then have
\begin{align*}
2\pi i\sum_j \ord_{p_j}m\int_{\sigma(s)}^{p_j} (\nabla_{\GM}\nu_L)' &=
\sum_{i=1}^g\left\{
\int_{\alpha_i}(\nabla_{\GM}\nu_L)'
\int_{\beta_i}\frac{dm}{m}-\int_{\alpha_i}\frac{dm}{m}\int_{\beta_i}(\nabla_{\GM}\nu_L)'
\right\} \\
2\pi i\sum_j \ord_{p_j}m\int_{\sigma(s)}^{p_j} (\nabla_{\GM}\nu_L)'' &=
-\sum_{i=1}^g\left\{
\int_{\alpha_i}(\nabla_{\GM}\nu_L)''
\int_{\beta_i}\overline{\frac{dm}{m}}-\int_{\alpha_i}\overline{\frac{dm}{m}}\int_{\beta_i}(\nabla_{\GM}\nu_L)''
\right\}.
\end{align*}
Hence,
\begin{align*} 
2\pi i\sum_j \ord_{p_j}m&\int_{\sigma(s)}^{p_j} \nabla_{\GM}\nu_L =
\sum_{i=1}^g\biggl\{
 \real\log\chi_M(\beta_i)\int_{\alpha_i}\left[(\nabla_{\GM}\nu_L)'
 -
(\nabla_{\GM}\nu_L)''\right] \\
&\qquad\qquad
 -\real\log\chi_M(\alpha_i)\int_{\beta_i}\left[(\nabla_{\GM}\nu_L)'-
(\nabla_{\GM}\nu_L)''\right]\biggr\} \\
&\quad+ \sum_{i=1}^g
\left\{
 i\imag\log\chi_M(\beta_i)\int_{\alpha_i}\nabla_{\GM}\nu_L
 -i\imag\log\chi_M(\alpha_i)\int_{\beta_i}\nabla_{\GM}\nu_L
\right\}
\end{align*}
(the choice of log is immaterial). 
Using the flatness of the Gauss-Manin connection, 
$$d\int_\gamma\nabla_{\GM}\nu_L=\int_\gamma\nabla^2_{\GM}\nu_L=0
$$
we find
\begin{align}
-2\pi i\tr_{\Div m/S^{\circ}} &F_{\nabla^{tr}_{\langle\Lcal,\Mcal\rangle}}=
\sum_{i=1}^g\biggl\{
 d\real\log\chi_M(\beta_i)\wedge \int_{\alpha_i}\left[(\nabla_{\GM}\nu_L)'-
(\nabla_{\GM}\nu_L)''\right]\label{eqn:tracecurvature}  \\
&\qquad\qquad
 -d\real\log\chi_M(\alpha_i)\wedge \int_{\beta_i}\left[(\nabla_{\GM}\nu_L)'-
(\nabla_{\GM}\nu_L)''\right]\biggr\}\notag \\
&+ \sum_{i=1}^g
\left\{
 id\imag\log\chi_M(\beta_i)\wedge \int_{\alpha_i}\nabla_{\GM}\nu_L
 -id\imag\log\chi_M(\alpha_i)\wedge \int_{\beta_i}\nabla_{\GM}\nu_L\right\}      \notag    \\
& +\sum_{i=1}^g\biggl\{
 \log|\chi_M(\beta_i)|\int_{\alpha_i}\nabla_{\GM}\left[(\nabla_{\GM}\nu_L)'-
(\nabla_{\GM}\nu_L)''\right] \notag\\
&\qquad\qquad
 -\log|\chi_M(\alpha_i)|\int_{\beta_i}\nabla_{\GM}\left[(\nabla_{\GM}\nu_L)'-
(\nabla_{\GM}\nu_L)''\right] 
\biggr\}\notag.
\end{align}
Now
\begin{align*}
d\real\log\chi_M(\gamma)&=\frac{1}{2}\int_\gamma
\nabla_{\GM}\nu_M+\overline{\nabla_{\GM}\nu_M} \\
id\imag\log\chi_M(\gamma)&=\frac{1}{2}\int_\gamma
\nabla_{\GM}\nu_M-\overline{\nabla_{\GM}\nu_M}.
\end{align*}
Substituting this into \eqref{eqn:tracecurvature}, and using \eqref{eqn:ks},  we have
\begin{align*}
-2\pi i\cdot &\tr_{\Div m/S^{\circ}} F_{\nabla^{tr}_{\langle\Lcal,\Mcal\rangle}}=
\frac{1}{2}\sum_{i=1}^g\biggl\{
\int_{\beta_i}\left[ \nabla_{\GM}\nu_M+\overline{\nabla_{\GM}\nu_M}\right]
\wedge \int_{\alpha_i}\left[ (\nabla_{\GM}\nu_L)'-(\nabla_{\GM}\nu_L)''\right] \\
&\qquad-
 \int_{\alpha_i}\left[ \nabla_{\GM}\nu_M+\overline{\nabla_{\GM}\nu_M}\right]
\wedge \int_{\beta_i}\left[
(\nabla_{\GM}\nu_L)'-(\nabla_{\GM}\nu_L)''\right]\biggr\} \\
&\quad+
\frac{1}{2}\sum_{i=1}^g\biggl\{
\int_{\beta_i}\left[ \nabla_{\GM}\nu_M-\overline{\nabla_{\GM}\nu_M}\right]
\wedge\int_{\alpha_i}\nabla_{\GM}\nu_L \\
&\qquad- \int_{\alpha_i}\left[ \nabla_{\GM}\nu_M-\overline{\nabla_{\GM}\nu_M}\right]
\wedge\int_{\beta_i}\nabla_{\GM}\nu_L\biggr\} \\
&\quad+
2\sum_{i=1}^g\left\{
\log|\chi_M(\beta_i)|\int_{\alpha_i}\KS(\Xcal/S)\wedge\nabla_{\GM}\nu_L
-\log|\chi_M(\alpha_i)|\int_{\beta_i}\KS(\Xcal/S)\wedge\nabla_{\GM}\nu_L
\right\}\\
&=
\sum_{i=1}^g\left\{\int_{\beta_i}\nabla_{\GM}\nu_M\wedge\int_{\alpha_i}(\nabla_{\GM}\nu_L)'
-\int_{\beta_i}\overline{\nabla_{\GM}\nu_M}\wedge\int_{\alpha_i}(\nabla_{\GM}\nu_L)''\right\}\\
&\quad
-\sum_{i=1}^g\left\{\int_{\alpha_i}\nabla_{\GM}\nu_M\wedge\int_{\beta_i}(\nabla_{\GM}\nu_L)'
-\int_{\alpha_i}\overline{\nabla_{\GM}\nu_M}\wedge\int_{\beta_i}(\nabla_{\GM}\nu_L)''\right\}\\
&\quad+
2\sum_{i=1}^g\left\{
\log|\chi_M(\beta_i)|\int_{\alpha_i}\KS(\Xcal/S)\wedge\nabla_{\GM}\nu_L
-\log|\chi_M(\alpha_i)|\int_{\beta_i}\KS(\Xcal/S)\wedge\nabla_{\GM}\nu_L
\right\}.
\end{align*}
By the Riemann bilinear relations the right hand side is
\begin{align*}
&=
\sum_{i=1}^g\left\{\int_{\beta_i}(\nabla_{\GM}\nu_M)''\wedge\int_{\alpha_i}(\nabla_{\GM}\nu_L)'
-\int_{\beta_i}\overline{(\nabla_{\GM}\nu_M)''}\wedge\int_{\alpha_i}(\nabla_{\GM}\nu_L)''\right\}\\
&\quad
-\sum_{i=1}^g\left\{\int_{\alpha_i}(\nabla_{\GM}\nu_M)''
\wedge\int_{\beta_i}(\nabla_{\GM}\nu_L)'
-\int_{\alpha_i}\overline{(\nabla_{\GM}\nu_M)''}\wedge\int_{\beta_i}(\nabla_{\GM}\nu_L)''\right\}\\
&\quad+
2\sum_{i=1}^g\left\{
\log|\chi_M(\beta_i)|\int_{\alpha_i}\KS(\Xcal/S)\wedge\nabla_{\GM}\nu_L
-\log|\chi_M(\alpha_i)|\int_{\beta_i}\KS(\Xcal/S)\wedge\nabla_{\GM}\nu_L
\right\}.
\end{align*}
Collecting terms and applying the bilinear relations again, the formula now follows.
\end{proof}

\section{Examples and Applications}\label{section:examples}

\subsection{Reciprocity for trivial fibrations} \label{sec:trivial}
Throughout this section, we consider trivial families  $\Xcal=X\times S$, where $X$ is  a fixed compact Riemann surface $X$ of genus $g\geq 1$ with a prescribed point $\sigma\in X$.  Using the Hodge splitting we shall give explicit formulas illustrating the main construction of this paper in this simple case.
In particular, we shall give a direct proof of Weil reciprocity for the connection defined in Section \ref{sec:canonical}. 

The \emph{deRham moduli space} is defined by:
\begin{align*}
M_{dR}(X)&=\{\text{moduli of rank 1 holomorphic connections on $X$}\}\\
&\simeq H^1(X,\CBbb)/H^1(X,2\pi i\ZBbb)\ .
\end{align*}
As an algebraic variety, $M_{dR}(X)$ depends only on the underlying topological, and not the Riemann surface, structure of $X$.
We shall always assume a rigidification, or trivialization of our bundles at $\sigma$. If we take as base $S=M_{dR}(X)$, then there is a universal line bundle $\Lcal\to \Xcal$ equipped a universal relative connection. 
Choose a symplectic homology basis $\{\alpha_j, \beta_j\}_{j=1}^g$, and normalized abelian differentials $\omega_j$ with period matrix $\Omega$. Then $J(X)=\CBbb^g/\ZBbb^g+\ZBbb^g\Omega$.  
Given $[\nabla]\in M_{dR}(X)$, we have its associated character 
$$\chi_\nu: \pi_1(X,\sigma)\to \CBbb^\times: \gamma\mapsto\exp\left(\int_\gamma\nu\right)
$$
 (cf.\ \eqref{eqn:equivariant1}). We regard $\chi_\nu$ as an element of the \emph{Betti moduli space}
\begin{equation}\label{eqn:MB}
M_B(\pi_1(X,\sigma)):=\Hom(\pi_1(X,\sigma), \GL(1,\CBbb))\simeq (\CBbb^\times)^{2g}\ ,
\end{equation}
 with its structure as an algebraic variety.  The Riemann-Hilbert correspondence above gives a complex analytic (though not algebraic) isomorphism 
 $$
 M_{dR}(X)\simeq M_B(\pi_1(X,\sigma))\ .
 $$

 As before, we have chosen a lift of $\nu$ from $H^1(X,\CBbb)/H^1(X,2\pi i\ZBbb) $ to $H^1(X,\CBbb)$. In fact, we choose a harmonic representative of this class in $\Acal^1_X$, and continue to denote this by  $\nu$. Since we have chosen a basis $\{\omega_i\}$ for $H^{1,0}(X)$, we have local holomorphic coordinates $(t_i,s_i)$ for $M_{dR}(X)$. We shall write: $\nu=\sum_{i=1}^g t_i\omega_i+ s_i\overline{\omega_i}$, and so
 \begin{equation}\label{eqn:nabla}
 \nabla_{\GM}\nu=\sum_{i=1}^g  \omega_i\otimes dt_i+\bar\omega_i\otimes ds_i.
 \end{equation}
 Let $\widetilde X$ be the universal cover of $X$.  
According to the discussion in Section \ref{sec:canonical}, we view  sections $\ell$ of $ \Lcal$ as functions $\tilde\ell$ on $\widetilde X\times M_{dR}(X)$ that satisfy
\begin{equation}\label{eqn:equivariant}
\tilde\ell(\gamma z,\nu)=\exp\left(\int_\gamma \nu\right)\tilde \ell(z,\nu)
\end{equation}
(note that the bundle is invariant with respect to the integral lattice $H^1(X,2\pi i\ZBbb)$).  The universal connection 
$\nabla : \Omega^0(\mathcal X, \Lcal)\to \Omega^1(\mathcal X, \Lcal)$
is defined as follows (see \eqref{eqn:extension}): given $\tilde\ell$ satisfying \eqref{eqn:equivariant}, let
\begin{equation} \label{eqn:universal}
\nabla \ell (z,\nu)= d\tilde\ell(z,\nu) - \int_{\sigma}^z \nabla_\GM\nu \cdot \tilde\ell(z,\nu).
\end{equation}
  One can check directly that $ \nabla\ell (z,\nu)$ indeed satisfies the correct equivariance, and that the connection is independent of the choice of fundamental domain.  A change of homology basis has the same effect as pulling $\nabla_\GM\nu$ back by the corresponding action on $M_{dR}(X)$; and therefore $\nabla$ is independent of this choice.
 
 \begin{remark} \label{rem:betti}
This is the connection defined in \eqref{eqn:extension}. Notice that this is \emph{not} a holomorphic connection (see Remark \ref{rem:holomorphic}): $\left(\nabla_{\GM}\nu\right)''=\sum_{i=1}^g \bar\omega_i(z)\otimes ds_i$.
\end{remark}
The flat  connection corresponding to $\nu$  is $\nabla=d+\nu$, and  $\dbar_\nabla=\dbar+\nu''$ is the corresponding $\dbar$-operator for the holomorphic line bundle $\Lcal$ defined by $\nabla$. 
The map $\pi: M_{dR}(X)\to J(X)$ which takes a holomorphic connection to its underlying holomorphic line bundle realizes $M_{dR}(X)$ as an affine bundle over $J(X)$.  We wish to write this map explicitly. First, we identify the Jacobian variety $J(X)$ with the space of flat $\U(1)$-connections, or equivalently, as the space of $U(1)$-representations of $\pi_1(X,\sigma)$. 
 The Chern connection on $\Lcal_\nu$ is $d_A= d+ \nu''-\overline{\nu''}$, and this defines a \emph{unitary} character $\chi_u:\pi_1(X)\to \U(1)$. Let
\begin{align*}
2\pi i a_j&=\log\chi_u(\alpha_j)=\int_{\alpha_j} \nu''-\overline{\nu''} \\
2\pi i b_j&=\log\chi_u(\beta_j)=\int_{\beta_j} \nu''-\overline{\nu''}.
\end{align*}
Then the point $[u]\in J(X)$ corresponding to $[\nabla]\in M_{dR}(X)$ is given by
\begin{equation} \label{eqn:u}
u=b-a^t\Omega.
\end{equation}
In terms of these coordinates, one calculates: 
\begin{align}
2\pi i a_j &= s_j-\bar s_j\label{eqn:alpha} \\
2\pi i b_j & =\sum_{k=1}^g s_k\overline\Omega_{kj}-\bar s_k\Omega_{kj} \\
u_j & = -\frac{1}{\pi} s_k \imag \Omega_{kj} \label{eqn:uj} \\
s_j&= -\pi u_k(\imag \Omega)^{-1}_{kj} \label{eqn:s-u}
\end{align}

\begin{remark} \label{rem:unitary}
Notice that there is a smooth section $\jmath: J(X)\to M_{dR}(X)$ defined by
$$
[u]\mapsto [\nu] : \nu=\sum_{i=1}^g (-\bar s_i)\omega_i+s_i\bar\omega_i
$$
where $s_j$ is given by \eqref{eqn:s-u}.  The image $U_{dR}(X)\subset M_{dR}(X)$, which consists of the unitary connections, is a totally real submanifold.
\end{remark}

Next, we express meromorphic sections of $\Lcal_\nu\to X$ in terms of meromorphic functions
on $\widetilde X$ satisfying the equivariance \eqref{eqn:equivariant}.  Let $E(z,w)$ be the Schottky prime form associated with $\{\alpha_i,\beta_i\}$ (cf.\ \cite{Fay:73}).  For a meromorphic section $\ell$ of $\Lcal_\nu$ with divisor $\sum_{i=1}^N p_i-q_i$ we have
\begin{equation} \label{eqn:divisor}
\sum_{i=1}^N \int^{p_i}_{q_i}\vec \omega = u + m + n^t\Omega\quad ,\quad m,n\in \ZBbb^g
\end{equation}

\begin{lemma}
Define 
\begin{equation*}
\tilde \ell(z)=\frac{\prod_{i=1}^N E(z,p_i)}{\prod_{i=1}^N E(z,q_i)}\exp\left\{ 2\pi i(a-n)^t\int_{\sigma}^z\vec\omega+ \int_{\sigma}^z \nu'+\overline{\nu''}
\right\} 
\end{equation*}
Then $\tilde \ell$ is a meromorphic function on $\widetilde X$ with multipliers $\chi_\nu$ and divisor  projecting to $\Div(\ell)$.
\end{lemma}
A particular case of the above formula is a meromorphic function $f(z)$ with divisor $\sum_{i=1}^M x_i-y_i$ and
$$
\sum_{i=1}^M \int^{x_i}_{y_i}\vec \omega = \tilde m+ \tilde n^t\Omega \quad ,\quad \tilde m,\tilde n\in \ZBbb^g
$$
Then $f$ can be expressed
\begin{equation} \label{eqn:f}
f(z)=\frac{\prod_{i=1}^M E(z,x_i)}{\prod_{i=1}^M E(z,y_i)}\exp\left\{ -2\pi i \tilde n^t\int_{\sigma}^z\vec\omega\right\} 
\end{equation}
With this understood, we are ready to give a direct proof of (WR) in this setting:
\begin{proposition} \label{prop:reciprocity}
Let $\ell$ be a meromorphic section of the universal bundle $\Lcal\to \Xcal$ and $f$ a meromorphic function on $\Xcal$. Then
$$
\tr_{\Div f/M_{dR}}\left(\frac{ \nabla\ell }{\ell}\right) = \tr_{\Div \ell/M_{dR}}\left(\frac{df }{f}\right) 
$$
\end{proposition}

\begin{proof}
Fix $\nu=\sum_{i=1}^g t_i \omega_i + s_i\bar\omega_i$. Then $\nu'+\overline{\nu''}=\sum_{i=1}^g(t_i+\bar s_i)\omega_i$. Using \eqref{eqn:alpha} we see that, up to a nonzero multiplicative constant,
\begin{equation} \label{eqn:l-tilde}
\tilde \ell(z)=\frac{\prod_{i=1}^N E(z,p_i)}{\prod_{i=1}^N E(z,q_i)}\exp\left\{ (t+s)^t \int_{\sigma}^z\vec\omega-2\pi in^t\int_{\sigma}^z\vec\omega\right\} 
\end{equation}
(note that $m,n$ and $\tilde m,\tilde n$ are locally constant). Hence
\begin{align*}
\tilde \ell(\Div f)&=\frac{\prod_{i=1}^N\prod_{j=1}^M E(x_j,p_i)E(y_j,q_i)}{\prod_{i=1}^N E(x_j,q_i)E(y_j,p_i)}\exp\left\{ (t+s)^t \sum_{i=1}^M\int^{x_i}_{y_i}\vec\omega-2\pi in^t\sum_{i=1}^M\int^{x_i}_{y_i}\vec\omega\right\}  \\
&
=\frac{\prod_{i=1}^N\prod_{j=1}^M E(x_j,p_i)E(y_j,q_i)}{\prod_{i=1}^N\prod_{j=1}^M E(x_j,q_i)E(y_j,p_i)}\exp\left\{ (t+s)^t(\tilde m+\tilde n^t\Omega)-2\pi in^t(\tilde n^t\Omega)\right\}
\end{align*}
since $n^t \tilde m\in \ZBbb$.
Similarly, using \eqref{eqn:uj},
$$
f(\Div \ell)=\frac{\prod_{i=1}^N\prod_{j=1}^M E(x_j,p_i)E(y_j,q_i)}{\prod_{i=1}^N\prod_{j=1}^M E(x_j,q_i)E(y_j,p_i)}
\exp\left\{ 2i\tilde  n^t(s^t\imag \Omega) -2\pi i \tilde n^t(n^t\Omega)\right\}.\\
$$
Now $ (t+s)^t(\tilde m+\tilde n^t\Omega)-2i \tilde n^t(s^t\imag \Omega) = \sum_{i=1}^g t_j v_j+ s_j \bar v_j$, where $v=\tilde m+\tilde n^t\Omega$, so that
$$
\tilde \ell(\Div f)=f(\Div \ell)\exp\left( \sum_{i=1}^g t_j v_j+ s_j \bar v_j\right).
$$
Differentiating with respect to $(z,\nu)$,
$$
d\tilde \ell(\Div f)=\frac{d(f(\Div \ell))}{f(\Div \ell)  }\ell(\Div f)
+
\left( \sum_{i=1}^g dt_j v_j+ ds_j v_j\right) \ell(\Div f).
$$
On the other hand,
$$
(\int_{z_0}^z \nabla_\GM\nu)(\Div f)=\sum_{i=1}^M \int_{y_i}^{x_i} dt_i\omega_i+ds_i\bar \omega_i= \sum_{i=1}^g dt_j v_j+ ds_j v_j.
$$
The result now follows from the definition of $ \nabla$.
\end{proof}
\subsection{Holomorphic extension of analytic torsion}
As mentioned in the Introduction, one motivation for this paper was to derive an interpretation of the holomorphic extension of analytic torsion in terms of Deligne pairings.  In this section, we review the construction of torsion and  give explicit formulas, generalizing those in \cite{Fay:81} and \cite{Hitchin:13}.  In the next section, we explain the relationship with our construction.

First, we review the definition of analytic torsion for the non-self-adjoint operators we consider. Fix an arbitrary  hermitian, holomorphic line bundle $\Mcal\to X$ with Chern connection $\nabla=\partial_\nabla+\dbar_\nabla$.
Given a flat connection on $\Lcal \to X$ with holonomy $\chi$, we regard smooth sections $\ell$ of $\Lcal\otimes \Mcal $ as $\chi$-equivariant sections $\tilde \ell$ of the pull-back of $\Mcal$ to $\widetilde X$ satisfying \eqref{eqn:equivariant}.  Pick a lift $\nu\in H^1(X,\CBbb)$ of the character,
 and set
$$
G_\nu(z)= \exp\left(\int_{\sigma}^z\nu\right).
$$
Then for any $\tilde \ell$, notice that $G_\nu(z)\tilde\ell(z)$ is a well-defined smooth section of $\Mcal\to X$. Define the operators 
\begin{align*}
D'' &: \Omega^{p,0}(X,\Mcal)\lra \Omega^{p,1}(X,\Mcal) : \alpha\mapsto G_\nu^{-1}\dbar_\nabla \left(G_\nu \alpha\right) \\
D' &: \Omega^{0,q}(X, \Mcal)\lra \Omega^{1,q}(X,\Mcal) : \beta\mapsto G_\nu^{-1}\partial_\nabla\left( G_\nu \beta\right).
\end{align*}
Fix a conformal metric on $X$, and let $\ast$ denote the Hodge operator.
We define the \emph{laplacian} associated to $\nu$ and $\Mcal$ by ${ \square}_{\chi_\nu\otimes \Mcal} (s)=-2i\ast D' D'' (s)$, for smooth sections $s$ of $\Mcal$. This is an elliptic operator that is independent of the choice of base point $\sigma$. In case $\nu$ is unitary, $G_\nu$ has absolute value $=1$, and via \eqref{eqn:equivariant} gives a unitary equivalence between ${ \square}_{\chi_\nu\otimes \Mcal}$  and the ordinary $\dbar$-laplacian  
 for  $\Lcal\otimes\Mcal$. In particular, the spectra of these two operators is the same in this case.
 For $\nu$ not unitary, ${ \square}_{\chi_\nu\otimes\Mcal}$ is not a  symmetric operator. Nevertheless, the  following holds.

\begin{lemma} \label{lem:finite}
For  $\nu$ in a compact set there are at most finitely many eigenvalues $\lambda$ of ${ \square}_{\chi_\nu\otimes\Mcal}$ with $\real \lambda\leq 0$.
\end{lemma}

Since the symbol of ${ \square}_{\chi_\nu\otimes\Mcal}$ is the same as that of the scalar laplacian, the zeta regularization procedure applies to give a well-defined determinant $\det { \square}_{\chi_\nu\otimes\Mcal}$.
For a nice explanation of this,  we refer to \cite[Section 2.5]{AbanovBraverman:05}. We only point out here that by Lemma \ref{lem:finite},  $\det { \square}_{\chi_\nu\otimes\Mcal}$ is independent of a choice of \emph{Agmon} angle. 
Assume ${ \square}_{\chi_\nu\otimes\Mcal}$ has no zero eigenvalues.  If $\{\lambda_i\}_{i=1}^N$ are the eigenvalues of ${ \square}_{\chi_\nu\otimes\Mcal}$ with negative real part, $\{\mu_i\}$ the eigenvalues with positive real part, then by Lidskii's theorem
$$
\zeta_{{ \square}_{\chi_\nu\otimes\Mcal}}(s)=\lambda_1^{-s}+\cdots+\lambda_N^{-s}+\sum_{i=1}^\infty \mu_i^{-s}
$$
where any choice of logarithm is used to define the powers $\lambda_i^{-s}$ and the usual logarithm (real on the positive real axis) is used to define the rest.  Any other choice is of the form
$$
\tilde\zeta_{{ \square}_{\chi_\nu\otimes\Mcal}}(s)=\lambda_1^{-s}e^{-2\pi i k_1 s}+\cdots+\lambda_N^{-s}e^{-2\pi i k_N s}+\sum_{i=1}^\infty \mu_i^{-s}
$$
for integers $k_i$.  But then
$$
-\zeta_{{ \square}_{\chi_\nu\otimes\Mcal}}'(0)=-\tilde\zeta_{{ \square}_{\chi_\nu\otimes\Mcal}}'(0)+2\pi i\sum_{i=1}^N k_i
$$
and so $\det { \square}_{\chi_\nu\otimes\Mcal}=\exp(-\zeta_{{ \square}_{\chi_\nu\otimes\Mcal}}'(0))$ is independent of this choice.  We also note that a different choice of lift $\tilde \nu$ gives $G_{\tilde \nu}=G_{\nu}\cdot F$, where pointwise $\vert F\vert =1$.  Hence, $F$ gives a unitary automorphism of $L^2(X)$ such that ${ \square}_{\chi_{\tilde \nu}\otimes \Mcal}=F\circ{ \square}_{\chi_\nu\otimes\Mcal}\circ F^{-1}$; hence, 
the eigenvalues of ${ \square}_{\chi_{\tilde\nu}\otimes\Mcal}$ are the same as those of ${ \square}_{\chi_\nu\otimes\Mcal}$, and so the determinants agree.
Finally,  
since ${ \square}_{\chi_\nu\otimes\Mcal}$ depends holomorphically on $\nu$, so does  $\det { \square}_{\chi_\nu\otimes\Mcal}$ (see \cite{KontsevichVishik:93}), and since it agrees with the usual determinant when $\nu$ is unitary, $\det { \square}_{\chi_\nu\otimes\Mcal}$ is a holomorphic extension of the usual analytic torsion.  

As in Section \ref{sec:trivial}, 
let $X$ be a compact genus $g\geq 1$ Riemann surface with a conformal metric and a choice of symplectic homology basis $\{\alpha_j, \beta_j\}_{j=1}^g$.  This gives a period matrix $\Omega$, theta function $\vartheta(Z,\Omega)$,  and a Riemann divisor $\kappa$ of degree $g-1$, $2\kappa=\omega_X$.
Let $\chi_u: \pi_1(X)\to \U(1)$ be a unitary character whose holomorphic line bundle corresponds to the point $u\in J(X)$ as in \eqref{eqn:u}.  
The choice of conformal metric gives $\kappa$ a hermitian structure. 
Then the torsion of the $\dbar$-laplacian on $\chi_u\otimes \kappa$ is given by
\begin{equation} \label{eqn:twisted-torsion-real}
T(\chi_u\otimes \kappa)=\det{ \square}_{\chi_u\otimes\kappa}= C(X) \Vert \vartheta\Vert^2(u,\Omega) 
\end{equation}
where $C(X)$ is a constant depending on the Riemann surface $X$ and the conformal metric.
 Recall the definition of the norm:
$$
 \Vert \vartheta\Vert^2(u, \Omega) = \exp\left(-2\pi \imag
u^T(\imag\Omega)^{-1}\imag u\right)\vert \vartheta\vert^2(u, \Omega).
$$
In terms of periods, this is
\begin{equation} \label{eqn:norm-theta}
 \Vert \vartheta\Vert^2(u, \Omega) = \exp\left(-2\pi a^T( \imag\Omega )a\right)\vert \vartheta\vert^2(b-a^T\Omega, \Omega). 
\end{equation}
Now suppose $\chi: \pi_1(X,\sigma)\to \CBbb^\times$ is a complex character with periods $\chi(\alpha_j)=\exp(2\pi i a_j)$, $\chi(\beta_j)=\exp(2\pi i b_j)$, $a_j, b_j\in \CBbb/\ZBbb$.  Then we have the following definition:

\begin{equation} \label{eqn:twisted-holo-torsion}
T(\chi\otimes\kappa):=C(X) \exp\left(-2\pi a^T(\imag\Omega)a\right)\vartheta(b-a^T\Omega, \Omega)
\vartheta(b-a^T\overline\Omega, -\overline\Omega).
\end{equation}
By the transformation properties of the theta function, one verifies that  the expression in \eqref{eqn:twisted-holo-torsion} indeed depends on the values of $a_j, b_j$ modulo $\ZBbb$.
The subspace $U_B(\pi_1(X,\sigma))\subset M_B(\pi_1(X,\sigma))$ of unitary characters $  (S^1)^{2g}\subset (\CBbb^\times)^{2g}$ is totally real.
The following is clear:
\begin{proposition}[cf.\ \cite{Hitchin:14}] \label{prop:twisted-holo-torsion}
The function $\nu\mapsto T(\chi_\nu\otimes \kappa)$ is a holomorphic extension of the torsion on unitary characters. In particular, $T(\chi_\nu\otimes \kappa)=\det{ \square}_{\chi_\nu\otimes \kappa}$.
\end{proposition}

 Next we consider the holomorphic extension of the torsion  $T(\chi)$. For this we need to choose a basis of Prym differentials $\eta_i(z,\chi)$ on $X$ and  $\eta_i(\bar z, \chi^{-1})$ on $\overline X$, $i=1,\dots, g-1$ ($\chi$ nontrivial).  We choose these to vary holomorphically  in $\chi$, and for convenience  we require $\eta_i(\bar z, \chi^{-1})=\overline{\eta_i(z,\chi)}$ for $\chi$ unitary.  For $\chi$ unitary have a natural inner product
\begin{equation} 
\langle \eta_i(\chi), \eta_j(\chi)\rangle =\int_X \eta_i(z,\chi)\wedge \overline{\eta_j(z,\chi)}.
\end{equation}
For general characters $\chi$, define the pairing on Prym differentials on $X$ and $\overline X$ by
 \begin{equation} 
( \eta_i(\chi), \eta_j(\chi^{-1})) =\int_X \eta_i(z,\chi)\wedge \eta_j(\bar z,\chi^{-1}).
\end{equation}

 Choose generic points $p_1, \ldots, p_g$, and set
 $$
 u_0=\kappa-\sum_{i=1}^{g-1} p_i.
 $$
 Then for $\chi_u$ unitary, the torsion is given by
 \begin{align} 
 \begin{split} \label{eqn:torsion-real}
 T(\chi_u)
 &=4\pi^2C(X)\left| \det\omega_i(p_j)\right|^2\exp\left(4\pi \imag u_0 \cdot a
-2\pi a^T(\imag\Omega) a\right)
 \\
 &\qquad\times  \frac{\det \langle \eta_i(\chi), \eta_j(\chi)\rangle 
}{\left| \det\eta_i(p_j,\chi)\right|^2}
 \frac{\left|\vartheta (u+u_0, \Omega)\right|^2 }{\left|\sum_{i=1}^g\partial_{Z_i}\vartheta(u_0, \Omega)\omega_i(p_g) \right|^2} 
 \end{split}
 \end{align}
where it is understood that in the expression,  $\det\eta_i(p_j)$,  $1\leq j\leq g-1$.  As before this leads to the definition of holomorphic torsion.  For $\chi$ an arbitrary character, define
\begin{align} 
 \begin{split} \label{eqn:torsion-holo}
T(\chi)&=
4\pi^2C(X)\left| \det\omega_i(p_j)\right|^2 \exp\left(4\pi \imag u_0 \cdot
\alpha -2\pi \alpha^T(\imag\Omega) \alpha\right)
  \\
&\qquad \times
\frac{\det ( \eta_i(\chi), \eta_j(\chi^{-1}) )}{\det \eta_i(p_j,\chi)\det\eta_i(\bar p_j, \chi^{-1})}
\frac{\vartheta (\beta-\alpha^T\Omega+u_0, \Omega)\vartheta (\beta-\alpha^T\overline\Omega-\overline u_0, -\overline\Omega)}{\left|\sum_{i=1}^g\partial_{Z_i}\vartheta(u_0, \Omega)\omega_i(p_g) \right|^2}. \end{split}
\end{align}

\begin{proposition}[cf.\ \cite{Fay:81}] \label{prop:holo-torsion}
The function $\chi\mapsto T(\chi)$ is a holomorphic extension to $M_{dR}(X)$ of the torsion on unitary characters.
\end{proposition}

\subsection{Holomorphic torsion and the Deligne isomorphism}
We next explain how the holomorphic extension of analytic torsion is related to the Deligne isomorphism \eqref{eq:deligne-iso} and the intersection connection. 
 To begin, from \eqref{eqn:nabla} and \eqref{eqn:universal}  we  have that
the curvature of the universal connection is
$$
F_{ \nabla} = -\sum_{i=1}^g  \omega_i\wedge dt_i+\bar\omega_i\wedge ds_i \in \Acal^2_\Xcal\ .
$$
Computing directly  from this, or alternatively using Proposition \ref{prop:curvature-intersection}, it follows that
the curvature of the intersection connection on $\langle \Lcal, \Lcal\rangle$ is given by 
\begin{equation} \label{eqn:curv-inter}
F_{\nabla^{int}_{ \langle \Lcal, \Lcal\rangle}}=-\frac{2}{\pi}\sum_{i,j=1}^g \imag\Omega_{ij} (dt_i\wedge ds_j).
\end{equation}
Note that the intersection connection is holomorphic, coming from the fact that $\nabla_\GM \nu$ is of type $(1,0)$, as in Corollary
\ref{cor:hol-intersection}.

 Let us suppose, to simplify the following discussion, that the genus $g\geq 2$. Choose a uniformization $X=\Gamma\backslash\HBbb$, where $\HBbb\subset\CBbb$ is the upper half plane and $\Gamma\subset\PSL(2,\RBbb)$ is a cocompact lattice $\simeq \pi_1(X,\sigma)$. Then let $\overline X=\Gamma\backslash\LBbb$, where $\LBbb\subset\CBbb$ is the lower half plane.
  If $\overline \Xcal=\overline X\times M_{dR}(\overline X)$, and $\overline\pi:\overline\Xcal\to M_{dR}(\overline X)$ the projection,  we define the universal bundle $\overline\Lcal\to \overline\Xcal$, where the fiber over $\overline X\times \{\nu\}$ is the line bundle associated to the character $\chi_\nu^{-1}$. 
  Then $\langle \overline\Lcal, \overline\Lcal\rangle$ is also a holomorphic line bundle on $M_{dR}(\overline X)$.  By the Riemann-Hilbert correspondence, there are complex analytic isomorphisms:
  $$
  M_{dR}(X)\isorightarrow M_B(\Gamma) \isoleftarrow M_{dR}(\overline X)\ ,
  $$
  where $M_B(\Gamma)$ is defined in \eqref{eqn:MB}.
  We therefore regard $\langle \Lcal, \Lcal\rangle$ and $\langle \overline\Lcal, \overline\Lcal\rangle$ as holomorphic bundles on $M_B(\Gamma)$.
On $\overline X$, the imaginary part of the period matrix $\imag \Omega$ is unchanged, but the coordinates $(t_i, s_j)\mapsto (-s_j, -t_i)$.  Hence, by \eqref{eqn:curv-inter},
$$
F_{\nabla^{int}_{ \langle\overline \Lcal,\overline \Lcal\rangle}}=\frac{2}{\pi}\sum_{i,j=1}^g \imag\Omega_{ij} (dt_i\wedge ds_j).
$$
In particular, the intersection connection on $\langle \Lcal, \Lcal\rangle \otimes \langle \overline\Lcal, \overline\Lcal\rangle$ is flat!
%
%
%

Next, we have
\begin{lemma} \label{lem:natural-iso}
 For any choice of theta characteristic $\kappa$, there is the following functorial isomorphism
\begin{equation} \label{eqn:natural-iso}
	\left[ \det R\pi_{\ast}(\Lcal\otimes \kappa)\otimes\det R\pi_{\ast}(\kappa)^{-1}
\right]^{\otimes 12}\isorightarrow
	\langle\Lcal,\Lcal\rangle^{\otimes 6}\ .
\end{equation}
\end{lemma}

\begin{proof}
From compatibility of the Deligne pairing with tensor products,
\begin{align*}
\langle \Lcal\otimes\kappa, \Lcal\otimes\kappa\otimes \omega_{\Xcal/S}^{-1}\rangle &\simeq
\langle \Lcal\otimes\kappa, \Lcal\otimes\kappa^{-1}\rangle \\
&\simeq\langle \Lcal, \Lcal\otimes\kappa^{-1}\rangle \otimes \langle \kappa, \Lcal\otimes\kappa^{-1}\rangle \\
&\simeq\langle \Lcal, \Lcal\rangle\otimes \langle \Lcal, \kappa^{-1}\rangle
 \otimes \langle \kappa, \Lcal\rangle\otimes  \langle \kappa, \kappa^{-1}\rangle \\
 &\simeq\langle \Lcal, \Lcal\rangle\otimes\langle \kappa, \kappa\rangle^{-1}.
\end{align*}
Similarly, 
$$
\langle\omega_{\Xcal/S}, \omega_{\Xcal/S}\rangle\simeq \langle \kappa, \kappa\rangle^{\otimes 4}.
$$
By \eqref{eq:deligne-iso},
\begin{align*}
\det R\pi_{\ast}(\Lcal\otimes \kappa)^{\otimes 12}&\simeq \langle\omega_{\Xcal/S}, \omega_{\Xcal/S}\rangle\otimes
\langle \Lcal\otimes\kappa, \Lcal\otimes\kappa\otimes \omega_{\Xcal/S}^{-1}\rangle^{\otimes 6} \\
&\simeq \langle \kappa, \kappa\rangle^{4}\otimes \langle \Lcal, \Lcal\rangle^{\otimes 6}\otimes\langle \kappa, \kappa\rangle^{-6} \\
&\simeq  \langle \Lcal, \Lcal\rangle^{\otimes 6}\otimes\langle \kappa, \kappa\rangle^{-2}.
\end{align*}
On the other hand,
$$
\det R\pi_{\ast}( \kappa)^{\otimes 12}\simeq \langle \kappa, \kappa\rangle^{\otimes 4}\otimes \langle \kappa, \kappa^{-1}\rangle^{\otimes 6}\simeq \langle \kappa, \kappa\rangle^{-2}.
$$
The result follows.
\end{proof}
\begin{remark} \label{rem:iso}
There is a refinement of Deligne's isomorphism to virtual bundles of virtual rank 0, such as $\Lcal\otimes\kappa-\kappa$. In this case, the lemma can be refined to a more natural looking isomorphism, canonical up to sign
\begin{displaymath}
	\left[ \det R\pi_{\ast}(\Lcal\otimes \kappa)\otimes\det R\pi_{\ast}(\kappa)^{-1}
\right]^{\otimes 2}\isorightarrow
	\langle\Lcal,\Lcal\rangle. 
\end{displaymath}
Consequently, the isomorphism of the lemma is canonical and there is no sign ambiguity (since we take the 6th power of the latter).
\end{remark}
%
%
%
%
%

We may now give a geometric interpretation of the holomorphic extension of torsion.
To simplify the notation, let
$$
\lambda(X,\kappa)=\det R\pi_{\ast}(\Lcal\otimes \kappa)\otimes\det R\pi_{\ast}(\kappa)^{-1}.
$$
Considering both $X$ and $\overline X$,
by \eqref{eqn:natural-iso}  we have a canonical isomorphism
\begin{equation} \label{eqn:deligne-isomorphism}
\phi : \left[ \lambda(X,\kappa)\otimes \lambda(\overline X,\bar\kappa)\right]^{\otimes 12}\isorightarrow 
\left[\langle\Lcal,\Lcal\rangle\otimes\langle\overline\Lcal,\overline\Lcal\rangle\right]^{\otimes 6}.
\end{equation}
By the previous discussion, the right hand side admits a holomorphic (in fact, flat) connection. 
On the other hand, $\lambda(X,\kappa)\otimes \lambda(\overline X,\bar\kappa)$ has a canonical holomorphic connection given by the form 
$$-\partial\log T(\chi\otimes \kappa)$$
 in the 
canonical (up to a constant) frame, given by the relation with theta functions (see \ref{eqn:twisted-holo-torsion}).
With this understood, we have the following

\begin{theorem} \label{thm:flat}
The Deligne isomorphism $\phi$ in \eqref{eqn:deligne-isomorphism} is flat with respect to the connections defined above.
\end{theorem}

%
%
%
%
%
%

\begin{proof}
We first show, by explicit calculation, that $\phi$ is flat when restricted to  the unitary connections $U_{dR}(X)\subset M_{dR}(X)$. 
Recall from Lemma \ref{lem:chern} that the connection on $\langle\Lcal, \Lcal\rangle$ coincides with the Chern connection along $U_{dR}(X)$.  The connection we have defined on $\Ical_\kappa(X)$ differs from the Quillen connection by the additive term
\begin{equation} \label{eqn:difference}
\partial\log\left[ \frac{T_X(\chi_u\otimes \kappa) T_{\overline X}(\chi_u^{-1}\otimes \overline\kappa)  }{T(\chi\otimes \kappa)}\right].
\end{equation}
Since the Deligne isomorphism (for the Quillen metric) is an isometry, it suffices to show that the expression in \eqref{eqn:difference} vanishes when $\chi$ is unitary.
Let $\nabla=d +B$ be a flat connection, where $B=\sum_{i=1}^g t_i\omega_i+ s_i\bar\omega_i$.  Let $\chi_B\in M_{dR}(X)$ be the associated holonomy.  Notice that
\begin{align*}
2\pi ia_j&=\int_{A_j}b=t_j+s_j \\
2\pi ib_j&=\int_{B_j}b=\sum_{k=1}^g t_k\Omega_{kj}+s_k\overline\Omega_{kj}.
\end{align*}
From this expression and the interchange $\Omega\mapsto -\overline \Omega$, we see that the character $\chi^{-1}$ on $\overline X$ corresponds to the change of coordinates $(t_j, s_j)\mapsto (-s_j, -t_j)$. 


Next, consider the map $M_{dR}(X)\to J(X)$.  This takes $[\nabla]$ to the isomorphism class $[\nabla'']$ of the underlying holomorphic line bundle. In terms of flat connections, $[d+B]\mapsto [d+B''-\overline{B''}]$.  In terms of the coordinates introduced above  (see other note).
Rewriting \eqref{eqn:norm-theta} and \eqref{eqn:twisted-holo-torsion} in these coordinates we have
 \begin{align*}
 &T_X(\chi_u\otimes \kappa) =C(X)  \exp\left((1/2\pi)(s-\bar s)^T
(\imag\Omega)(s-\bar s)\right)\vert \vartheta\vert^2((1/\pi)(\imag\Omega) s, \Omega)  \\
& T_{\overline X}(\chi_u^{-1}\otimes \overline\kappa) =C(X)
\exp\left((1/2\pi)(t-\bar t)^T (\imag\Omega)(t-\bar t)\right)\vert
\vartheta\vert^2((1/\pi)(\imag\Omega) t, -\overline\Omega)  \\
T(\chi)&=C(X) \exp\left((1/2\pi)(t+ s)^T (\imag\Omega)(t+s)\right)
\vartheta((1/\pi)(\imag\Omega) s, \Omega)
\vartheta((1/\pi)(\imag\Omega) t, -\overline\Omega).
 \end{align*}
 We now calculate:
 \begin{align*}
 \partial_{\chi}\log T(\chi_u\otimes \kappa)
 &= \frac{1}{\pi}\sum_{i, j=1}^g(\imag \Omega)_{ij} (s_i-\bar s_i)ds_j \\
 &\qquad \qquad  +\frac{1}{\pi}\sum_{i, j=1}^g
\partial_{Z_i}\vartheta((1/\pi)(\imag\Omega) s, \Omega) (\imag \Omega)_{ij}ds_j \notag
 \end{align*}
  \begin{align*}
 \partial_{\chi}\log T_{\overline X}(\chi_u^{-1}\otimes \overline\kappa)
 &= \frac{1}{\pi}\sum_{i, j=1}^g(\imag \Omega)_{ij} (t_i-\bar t_i)dt_j \\
 &\qquad \qquad  +\frac{1}{\pi}\sum_{i, j=1}^g
\partial_{Z_i}\vartheta((1/\pi)(\imag\Omega) t, -\overline\Omega)(\imag \Omega)_{ij}dt_j \notag
 \end{align*}
\begin{align*}
 \partial_{\chi}\log T(\chi\otimes \kappa)
 &= \frac{1}{\pi}\sum_{i, j=1}^g(\imag \Omega)_{ij} (t_i+s_i)(dt_j+ds_j) \\
 &\qquad \qquad  +\frac{1}{\pi}\sum_{i, j=1}^g
\partial_{Z_i}\vartheta((1/\pi)(\imag\Omega) s, \Omega) (\imag \Omega)_{ij}ds_j \notag \\
 &\qquad \qquad  +\frac{1}{\pi}\sum_{i, j=1}^g
\partial_{Z_i}\vartheta((1/\pi)(\imag\Omega) t, -\overline\Omega)(\imag \Omega)_{ij}dt_j. \notag
  \end{align*}
  Hence, restricted to the unitary connections $U_{dR}(X)\subset M_{dR}(X)$ defined  by $t_i=-\bar s_i$ (see Remark \ref{rem:unitary}),
  $$ \partial_{\chi}\log T(\chi\otimes \kappa)=\partial_{\chi}\log T(\chi_u\otimes \kappa)+\partial_{\chi}\log T_{\overline X}(\chi_u^{-1}\otimes \overline\kappa)$$
 This proves the claim.  It follows that $\nabla\phi$ restricted to $U_{dR}(X)$ vanishes.  But since $U_{dR}(X)$ is totally real and $\nabla\phi$ is holomorphic, we conclude that $\nabla\phi\equiv 0$.
 \end{proof}
 
 \begin{remark}
 An analogous result to Theorem \ref{thm:flat} holds for the holomorphic torsion $T(\chi)$ and the determinant bundle $\det R\pi_\ast(\Lcal)$.  The idea of the proof is the same, where the calculation making use of \eqref{eqn:torsion-holo} is somewhat more lengthy. 
 \end{remark}
 
%
%
%
%
%
\subsection{The hyperholomorphic line bundle on twistor space}
In this section we show how the intersection connection leads quite naturally to the construction of a meromorphic connection on the hyperholomorphic line bundle over the twistor space of $M_{dR}(X)$.  This result is inspired by Hitchin's exposition in \cite{Hitchin:13, Hitchin:14}, to which we refer for more context and detail.

We begin with a quick review of the basic set-up.  Recall that $M_{dR}(X)$ has a hyperk\"ahler structure (for much more on this, see \cite{GoldmanXia:08}).  In terms of the coordinates introduced above, the symplectic structures are:
\begin{align*}
\Phi_1&=\frac{i}{2\pi}\sum_{i,j=1}^g \imag \Omega_{ij} ( dt_i\wedge d\bar t_j+ ds_i\wedge d\bar s_j) \\
\Phi_2+i\Phi_3&=\frac{1}{\pi}\sum_{i,j=1}^g \imag\Omega_{ij} ds_i\wedge dt_j.
\end{align*}
Let  $Z=M_{dR}(X)\times \PBbb^1$  denote the twistor space of $M_{dR}(X)$, and $\lambda: Z\to \PBbb^1$ the projection.  
Then $Z$ has the structure of a complex manifold with respect to which $\lambda$ is holomorphic, but the tautological complex structure is not a product.
The fiber $\lambda^{-1}(1)$ is biholomorphic to $M_{dR}(X)$, whereas the fiber $\lambda^{-1}(0)$ is biholomorphic to $T^{\ast} J(X)$,  the space of rank $1$ Higgs bundles on $X$.  Similarly, $\lambda^{-1}(\infty)\simeq T^{\ast} J(\overline X)$.
Each fiber has a holomorphic symplectic form given by 
\begin{equation} \label{eqn:hklr}
\Phi=\Phi_2+i\Phi_3+2i\lambda\Phi_1+\lambda^2(\Phi_2-i\Phi_3)
\end{equation}
(see \cite[Theorem 3.3]{HKLR:87}).

Next, recall the following (see \cite{Simpson:97}).

\begin{definition}[Deligne] \label{def:lambda}
Let $S$ be smooth algebraic, and set $\Xcal =X\times S$. Suppose we are given a function $\lambda: S\to \ABbb^1$. Then a \emph{$\lambda$-connection} on a line bundle $\Lcal\to \Xcal$ is a $\CBbb$-linear map $\nabla_\lambda: \Lcal\to \Lcal\otimes \omega_{\Xcal/S}$ of $\Ocal_\Xcal$-modules satisfying
$$
\nabla_\lambda(f\ell)= \lambda df\otimes \ell + f\cdot\nabla_\lambda\ell.
$$
for $f\in \Ocal_\Xcal$ and $\ell\in \Lcal$.
\end{definition}
By a result of Simpson, 
the functor which associates to $\lambda: S\to \ABbb^1$ the set of rank one $\lambda$-connections on $\Xcal$ is representable by a scheme $M_{Hod}(X)$ with a morphism $\lambda: M_{Hod}(X)\to \ABbb^1$.  By considering $M_{Hod}(\overline X)$ and a gluing procedure with respect to the anti-holomorphic involution $\lambda\mapsto -\bar\lambda^{-1}$, one constructs the \emph{Deligne moduli space of $\lambda$-connections} $\lambda: M_{Del}(X)\to \PBbb^1$.  Moreover, there is a biholomorphism
$M_{Del}(X)\simeq Z$.  This is achieved by finding holomorphic sections $\ABbb^1\to M_{Hod}(X)$ of $\lambda$, compatible with the anti-holomorphic involution.  For example, in the case of the flat connection $\nabla=d+\nu$, $\nu$ harmonic, the family of $\lambda$-connections is given by:
\begin{align}
\begin{split} \label{eqn:lambda}
\nabla_\lambda^{0,1}&= \dbar + \frac{1}{2}( (\lambda+1)\nu''+(\lambda-1)\overline{\nu'}) \\
\nabla_\lambda^{1,0}&=\partial + \frac{1}{2}((1+\lambda)\nu'+(1-\lambda)\overline{\nu''}).
\end{split}
\end{align}

Let $\Xcal=X\times M_{Del}(X)$, $\pi: \Xcal\to M_{Del}(X)$ the projection.  We furthermore assume a rigidification.
Then the universal bundle $\Lcal\to \Xcal$ admits a universal $\lambda$-connection. Let $\kappa$ be a theta characteristic as in the previous section, and use the same notation for the pull-back to $\Xcal\to X$.  We define the \emph{hyperholomorphic line bundle}
on $M_{Del}(X)$ by
\begin{equation} \label{eqn:hyper}
\Lcal_Z:= \det R\pi_\ast(\Lcal\otimes\kappa)\otimes \det R\pi_\ast(\kappa)^{-1}.
\end{equation}

Consider the divisor $D=D_0\cup D_\infty=\lambda^{-1}(0)\cup \lambda^{-1}(\infty)$.
We shall use the construction of this paper to obtain an explicit realization of the following property of the hyperholomorphic line bundle.
\begin{theorem}[{Hitchin, cf.\ \cite[Theorem 3]{Hitchin:14}}] \label{thm:hitchin}
The line bundle $\Lcal_Z$ admits a meromorphic connection with logarithmic singularities along the divisor $D$. The curvature of this connection restricted to the fibers of $Z-D\to \CBbb^\times$ is  $\lambda^{-1}\Phi$, where $\Phi$ is the HKLR form \eqref{eqn:hklr}. The residue of the connection at $\lambda=0$ (resp.\ $\lambda=\infty$) is the Liouville or tautological $1$-form on $T^{\ast} J(X)$ (resp.\ $T^{\ast} J(\overline X)$).
\end{theorem}

\begin{proof}
There is a holomorphic map $Z-D\to M_{dR}(X)$ obtained by sending a holomorphic bundle with $\lambda$-connection $\nabla_\lambda$ to the same holomorphic bundle with holomorphic connection $\lambda^{-1} \nabla_\lambda$.
By Remark \ref{rem:iso}, $\Lcal_Z^{\otimes 12}$ is naturally isomorphic to the pull-back of $\langle\Lcal, \Lcal\rangle^{\otimes 6}$, and therefore the pull-back of the intersection connection gives a holomorphic connection on  $\Lcal_Z$ over $Z-D$. 
The statement about the curvature follows from the fact that the HKLR form is the pull-back of the holomorphic symplectic form on $M_{dR}(X)$.  We shall verify this  directly using the coordinates above.  Let $(\tau_i, \sigma_i)$ be coordinates on $M_{dR}(X)$.
It will be convenient to  locally parametrize $Z-D_\infty$ by $(t_i,s_i,\lambda)$, where the $\lambda$-connection is given by,
$$
\nabla_\lambda=\begin{cases} \dbar + a'' +\lambda\psi'' \\
\lambda(\partial + a') + \psi'.
\end{cases}
$$
Here,  $a''=\sum_{i=1}^g s_i \bar\omega_i$, $a'=-\overline{a''}$, $\psi'=\sum_{i=1}^g t_i\omega_i$, $\psi''=\overline{\psi'}$. 
In these coordinates, the map $Z-D\to M_{dR}(X)$ is given by
$$
\tau_i=-\bar s_i+\lambda^{-1}t_i\quad ,\quad \sigma_i=s_i+\lambda\bar t_i.
$$
 Then
\begin{align*}
d\tau_i&= -d\bar s_i+\lambda^{-1}dt_i-\lambda^{-2}t_id\lambda \\
d\sigma_i&=ds_i+\lambda d\bar t_i +\bar t_i d\lambda
\end{align*}
from which
\begin{align*}
 d\tau_i\wedge d\sigma_j &=
 ds_i\wedge d\bar s_j+dt_i\wedge d\bar t_j +\lambda^{-1} d t_i\wedge ds_j -\lambda d\bar s_i\wedge d\bar t_j \\
 &\qquad + \bar t_j (-d\bar s_i + \lambda^{-1}dt_i)\wedge d\lambda 
+ \lambda^{-2}t_i(d s_j + \lambda d\bar t_j) \wedge d\lambda.
\end{align*}
Using \eqref{eqn:curv-inter}, it follows that restricted to the fibers,
$$
F_{\Lcal_Z}\biggr|_{\rm fiber}=2i\Phi_1 +\lambda^{-1}(\Phi_2+i\Phi_3) + \lambda(\Phi_2-i\Phi_3).
$$
For the residue at $\lambda=0$, note that from \eqref{eqn:l-tilde},
$$
\frac{d\tilde\ell}{\tilde \ell}=(d\tau)^t\int_{\sigma}^t\vec\omega+\tau^t d\int_{\sigma}^z\vec\omega + \text{ regular terms} 
$$
while
$$
-\int_{\sigma}^z \nabla_\GM\nu=-(d\tau)^t\int_{\sigma}^z\vec\omega-(d\sigma)^t\int_{\sigma}^z\overline{\vec\omega}.
$$
It follows that
\begin{align*}
\tr_{\Div m/S}\left( \frac{\nabla\ell}{\ell}\right)&= \tau^t d\left(\tr_{\Div m/S}\int_{\sigma}^z\vec\omega \right)     + \text{ regular terms} \\
&=-\frac{\lambda^{-1}}{\pi}\sum_{i,j=1}^g (\imag\Omega)_{ij}t_ids_j + \text{ regular terms}
\end{align*}
The residue of the connection at $\infty$ is calculated similarly.
This concludes the proof.
\end{proof}

\bibliographystyle{amsplain}
\bibliography{./papers}{}

\providecommand{\bysame}{\leavevmode\hbox to3em{\hrulefill}\thinspace}
\providecommand{\MR}{\relax\ifhmode\unskip\space\fi MR }
\providecommand{\MRhref}[2]{%
  \href{http://www.ams.org/mathscinet-getitem?mr=#1}{#2}
}
\providecommand{\href}[2]{#2}
\begin{thebibliography}{10}

\bibitem{AbanovBraverman:05}
Alexander Abanov and Maxim Braverman, \emph{Topological calculation of the
  phase of the determinant of a non self-adjoint elliptic operator}, Comm.
  Math. Phys. \textbf{259} (2005), no.~2, 287--305.

\bibitem{BeilinsonSchechtman:88}
A.~A. Be{\u\i}linson and V.~V. Schechtman, \emph{Determinant bundles and
  {V}irasoro algebras}, Comm. Math. Phys. \textbf{118} (1988), no.~4, 651--701.
  \MR{962493 (90m:32048)}

\bibitem{BlochEsnault:00}
Spencer Bloch and H{\'e}l{\`e}ne Esnault, \emph{A {R}iemann-{R}och theorem for
  flat bundles, with values in the algebraic {C}hern-{S}imons theory}, Ann. of
  Math. (2) \textbf{151} (2000), no.~3, 1025--1070. \MR{1779563 (2001k:14025)}

\bibitem{Deligne:87}
Pierre Deligne, \emph{Le d\'eterminant de la cohomologie}, Current trends in
  arithmetical algebraic geometry ({A}rcata, {C}alif., 1985), Contemp. Math.,
  vol.~67, Amer. Math. Soc., Providence, RI, 1987, pp.~93--177.

\bibitem{Elkik:89}
Ren\'ee Elkik, \emph{Fibr\'es d'intersections et int\'egrales de classes de
  {C}hern}, Ann. Sci. \'Ecole Norm. Sup. (4) \textbf{22} (1989), no.~2,
  195--226.

\bibitem{Fay:73}
John Fay, \emph{Theta functions on {R}iemann surfaces}, Lecture Notes in
  Mathematics, Vol. 352, Springer-Verlag, Berlin-New York, 1973.

\bibitem{Fay:81}
\bysame, \emph{Analytic torsion and {P}rym differentials}, Riemann surfaces and
  related topics: {P}roceedings of the 1978 {S}tony {B}rook {C}onference
  ({S}tate {U}niv. {N}ew {Y}ork, {S}tony {B}rook, {N}.{Y}., 1978), Ann. of
  Math. Stud., vol.~97, Princeton University Press, 1981, pp.~107--122.

\bibitem{GilletSoule:89}
Henri Gillet and Christophe Soul{\'e}, \emph{Arithmetic {C}how groups and
  differential characters}, Algebraic {$K$}-theory: connections with geometry
  and topology ({L}ake {L}ouise, {AB}, 1987), NATO Adv. Sci. Inst. Ser. C Math.
  Phys. Sci., vol. 279, Kluwer Acad. Publ., Dordrecht, 1989, pp.~29--68.
  \MR{1045844 (91j:14017)}

\bibitem{GilletSoule:92}
\bysame, \emph{An arithmetic {R}iemann-{R}och theorem}, Invent. Math.
  \textbf{110} (1992), no.~3, 473--543.

\bibitem{GoldmanXia:08}
William~M. Goldman and Eugene~Z. Xia, \emph{Rank one {H}iggs bundles and
  representations of fundamental groups of {R}iemann surfaces}, Mem. Amer.
  Math. Soc. \textbf{193} (2008), no.~904, viii+69.

\bibitem{GriffithsHarris:78}
Philip Griffiths and Joseph Harris, \emph{Principles of algebraic geometry},
  Wiley Interscience, New York, 1978.

\bibitem{Hitchin:13}
Nigel Hitchin, \emph{On the hyperk\"ahler/quaternion {K}\"ahler
  correspondence}, Comm. Math. Phys. \textbf{324} (2013), no.~1, 77--106.

\bibitem{Hitchin:14}
\bysame, \emph{The hyperholomorphic line bundle}, Algebraic and complex
  geometry, Springer Proc. Math. Stat., vol.~71, Springer, Cham, 2014,
  pp.~209--223.

\bibitem{HKLR:87}
Nigel Hitchin, Anders Karlhede, Ulf Lindstr{\"o}m, and Martin Ro{\v{c}}ek,
  \emph{Hyper-{K}\"ahler metrics and supersymmetry}, Comm. Math. Phys.
  \textbf{108} (1987), no.~4, 535--589.

\bibitem{Kim:07}
Young-Heon Kim, \emph{Holomorphic extensions of {L}aplacians and their
  determinants}, Adv. Math. \textbf{211} (2007), no.~2, 517--545.

\bibitem{KontsevichVishik:93}
Maxim Kontsevich and Simeon Vishik, \emph{Geometry of determinants of elliptic
  operators}, Functional analysis on the eve of the 21st century, {V}ol.\ 1
  ({N}ew {B}runswick, {NJ}, 1993), Progr. Math., vol. 131, Birkh\"auser, 1993,
  pp.~173--197.

\bibitem{MazurMessing:74}
Barry Mazur and William Messing, \emph{Universal extensions and one dimensional
  crystalline cohomology}, Lecture Notes in Mathematics, Vol. 370,
  Springer-Verlag, Berlin-New York, 1974.

\bibitem{McIntyreTeo:08}
Andrew McIntyre and Lee-Peng Teo, \emph{Holomorphic factorization of
  determinants of {L}aplacians using quasi-{F}uchsian uniformization}, Lett.
  Math. Phys. \textbf{83} (2008), no.~1, 41--58.

\bibitem{Quillen:85}
Daniel Quillen, \emph{Determinants of {C}auchy-{R}iemann operators on {R}iemann
  surfaces}, Funktsional. Anal. i Prilozhen. \textbf{19} (1985), no.~1, 37--41,
  96.

\bibitem{Simpson:97}
Carlos Simpson, \emph{The {H}odge filtration on nonabelian cohomology},
  Algebraic geometry---{S}anta {C}ruz 1995, Proc. Sympos. Pure Math., vol.~62,
  Amer. Math. Soc., Providence, RI, 1997, pp.~217--281.

\bibitem{Voisin:07}
Claire Voisin, \emph{Hodge theory and complex algebraic geometry. {II}},
  english ed., Cambridge Studies in Advanced Mathematics, vol.~77, Cambridge
  University Press, Cambridge, 2007, Translated from the French by Leila
  Schneps.

\end{thebibliography}

\end{document}